\documentclass[a4paper, 11pt, leqno]{amsart}
\pdfoutput=1 
 
\usepackage{lmodern}
\usepackage[T1]{fontenc}
\usepackage[pdftex]{hyperref}
\usepackage{amsmath,amsthm, amssymb, mathrsfs} 
\usepackage[inline]{enumitem}

\calclayout

\allowdisplaybreaks[3]

\title[Spherical functions for small $K$-types]
{\boldmath Spherical functions for small $K$-types}

\keywords{small $K$-types, spherical functions, hypergeometric functions}
\subjclass{22E45, 33C67, 43A90}

\author{Hiroshi Oda}
\address{Faculty of Engineering, Takushoku University,
815-1 Tatemachi, Hachioji, Tokyo 193-0985, Japan}%[Corresponding author]
\email{hoda@la.takushoku-u.ac.jp}
\thanks{The first author was supported by JSPS KAKENHI Grant Number 18K03346}
\author{Nobukazu Shimeno}
\address{School of Science \& Technology, Kwansei Gakuin University, 
2-1 Gakuen, Sanda, Hyogo 669-1337, Japan}
\email{shimeno@kwansei.ac.jp}

\newcommand{\bbN}{{\mathbb N}}
\newcommand{\bbZ}{{\mathbb Z}}
\newcommand{\bbQ}{{\mathbb Q}}
\newcommand{\bbR}{{\mathbb R}}
\newcommand{\bbC}{{\mathbb C}}
\newcommand{\bbH}{{\mathbb H}}

\newcommand{\bsm}{\boldsymbol m}
\newcommand{\bsk}{\boldsymbol k}

\newcommand{\bska}{\boldsymbol \kappa}

\newcommand{\spt}{{\mathrm{split}}}
\newcommand{\real}{{\mathrm{real}}}
\newcommand{\cpt}{{\mathrm c}}
\newcommand{\reg}{{\mathrm{reg}}}
\newcommand{\alg}{{\mathrm{alg}}}

\DeclareMathOperator{\Ad}{Ad}
\DeclareMathOperator{\ad}{ad}
\DeclareMathOperator{\pr}{pr}
\DeclareMathOperator{\End}{End}
\DeclareMathOperator{\Mat}{Mat}
\DeclareMathOperator{\Ker}{Ker}
\DeclareMathOperator{\Lie}{Lie}

\DeclareMathOperator{\diag}{diag}
\DeclareMathOperator{\SL}{SL}

\DeclareMathOperator{\SO}{SO}
\DeclareMathOperator{\Sp}{Sp}
\DeclareMathOperator{\SU}{SU}
\DeclareMathOperator{\Spin}{Spin}
\DeclareMathOperator{\Hom}{Hom}
\DeclareMathOperator{\id}{id}
\DeclareMathOperator{\trace}{Tr}

\DeclareMathOperator{\lRad}{l-Rad}
\DeclareMathOperator{\rRad}{r-Rad}
\DeclareMathOperator{\ord}{ord}
\DeclareMathOperator{\sgn}{sgn}

\newcommand{\simarrow}{\xrightarrow{\smash[b]{\lower 0.7ex\hbox{$\sim$}}}}
\newcommand{\trans}{{}^{\mathrm t}}

\numberwithin{equation}{section}
\theoremstyle{plain}
 \newtheorem{thm}{Theorem}[section]
 \newtheorem{cor}[thm]{Corollary}
 \newtheorem{lem}[thm]{Lemma}
 \newtheorem{prop}[thm]{Proposition}
 
\theoremstyle{definition}
 \newtheorem{defn}[thm]{Definition}

\theoremstyle{remark}
 \newtheorem{rem}[thm]{Remark}

\begin{document}
\maketitle
\begin{abstract}
For a connected semisimple real Lie group $G$ of non-compact type,
Wallach introduced a class of $K$-types called \emph{small}.
We classify all small $K$-types for all simple Lie groups
and prove except just one case
that each \emph{elementary spherical function} for each small $K$-type $(\pi,V)$
can be expressed as a product of hyperbolic cosines and a \emph{Heckman-Opdam hypergeometric function}.
As an application, the inversion formula for the spherical transform
on $G\times_K V$ is obtained from Opdam's theory on \emph{hypergeometric Fourier transform}s.
\end{abstract}

\section{Introduction}\label{sec:intro}
Let $G$ be a connected real semisimple Lie group with finite center.
Let $G=KAN$ be an Iwasawa decomposition of $G$.
$K$-bi-invariant $C^\infty$ functions on $G$ are called \emph{spherical function}s and
are important in the analysis of functions on the Riemannian symmetric space $G/K$.
In particular, a key role is played by those spherical functions that are \emph{elementary}.
Here a spherical function $\phi$ is called elementary if it is non-zero and satisfies
the functional equation
\begin{equation*}
\int_K\phi(xky)dk=\phi(y)\phi(x)
\qquad
\text{for any }x,y\in G
\end{equation*}
($dk$ is the normalized Haar measure on $K$),
or equivalently, if it takes $1$ at $1_G\in G$ and is a joint eigenfunction of the algebra $\boldsymbol D$ of the invariant differential operators on $G/K$ (cf.~\cite{HC1, Hel2}).
It is well known that the \emph{spherical transform} (also called the \emph{Harish-Chandra transform}) defined by elementary spherical functions essentially gives
the irreducible decomposition of $L^2(G/K)$.

Now suppose $(\pi,V)$ is any irreducible unitary representation of $K$ (a \emph{$K$-type}\/ for short).
When we consider the analysis of sections of the vector bundle $G\times_K V$
in a parallel way to the case of $G/K$ (which corresponds to the trivial $K$-type),
there naturally appears a notion of elementary spherical functions for $(\pi,V)$.
Unfortunately the general theory for such functions, which has been developed by \cite{Godement,War,Tirao,GV} and others,
has some inevitable complexity.
But it can be considerably reduced when the algebra $\boldsymbol D^\pi$ of the invariant differential operators on $G\times_KV$ is commutative (cf.~\cite{Camporesi2}).
We know from \cite[Theorem 3]{Deit} that $\boldsymbol D^\pi$ is commutative
if and only if $V$ decomposes multiplicity-freely as an $M$-module
($M$ is the centralizer of $A$ in $K$).
In what follows we assume that $(\pi,V)$ satisfies this condition.

\begin{defn}\label{defn:pi-sph}
An $\End_\bbC V$-valued $C^\infty$ function $\phi$ on $G$ is called \emph{$\pi$-spherical}
if
\begin{equation}\label{K-biinvariance}
\phi(k_1gk_2)=\pi(k_2^{-1})\phi(g)\pi(k_1^{-1})
\quad
\text{for any }g\in G\text{ and }k_1,k_2\in K.
\end{equation}
The space of $\pi$-spherical functions is denoted by $C^\infty(G,\pi,\pi)$.
A $\pi$-spherical function $\phi$ is called elementary when it is non-zero and satisfies
\begin{equation}\label{functional_eq}
\int_K \phi(xky)\trace{\pi(k)}dk=\frac1{\dim V}\phi(y)\phi(x)
\quad
\text{for any }x,y\in G.
\end{equation}
\end{defn}
As we see in \S\ref{subsec:Dpi} in detail,
$\boldsymbol D^\pi$ naturally acts on $C^\infty(G,\pi,\pi)$.
\begin{thm}[{\cite[Theorem 3.8]{Camporesi2}}]\label{prop:elementarity}
A given $\phi\in C^\infty(G,\pi,\pi)$ is elementary
if and only if it takes $\id_V$ at $1_G$ and is a joint eigenfunction of $\boldsymbol D^\pi$.
\end{thm}
The theorem is certainly central 
when we investigate
analytical properties of elementary $\pi$-spherical functions.
However,
to get even a little of explicit results,
two more things seem necessary, namely,
the structure of $\boldsymbol D^\pi$
and
a version of \emph{Chevalley restriction theorem} for $C^\infty(G,\pi,\pi)$.
As shown below,
both of them are available if $(\pi,V)$ is \emph{small} in the sense of Wallach.
\begin{defn}[{ \cite[\S11.3]{Wal}}]
A $K$-type $(\pi,V)$ is called small
if $V$ is irreducible as an $M$-module.
\end{defn}
In this paper we restrict ourselves to the study of
elementary $\pi$-spherical functions for small $K$-types.
So hereafter let $(\pi,V)$ be a small $K$-type.
First, we have
a lot of the same results as in the case of the trivial $K$-type.
For example, 
there is a generalization of the \emph{Harish-Chandra isomorphism} by Wallach (Theorem \ref{thm:HChomo}):
\[
\gamma^\pi : \boldsymbol D^\pi \simarrow S(\mathfrak a_\bbC)^W.
\]
Here $S(\mathfrak a_\bbC)$ is the symmetric algebra for the complexification of the Lie algebra $\mathfrak a$ of $A$.
(As usual, when a capital English letter denotes a Lie group,
the corresponding German small letter denotes its Lie algebra.)
$W$ is the \emph{Weyl group} defined, as usual, by $W=\tilde M/M$
with $\tilde M$ being the normalizer of $A$ in $K$.
$S(\mathfrak a_\bbC)^W$
is the subalgebra of 
$S(\mathfrak a_\bbC)$ consisting of the $W$-invariants.
Furthermore we have
\begin{thm}\label{thm:unique-sph}
For any $\lambda \in \mathfrak a_\bbC^*$ (the dual linear space of $\mathfrak a_\bbC$), there exists a unique $\phi^\pi_\lambda\in C^\infty(G,\pi,\pi)$ such that
\begin{equation}\label{eq:unique-sph}
\phi^\pi_\lambda(1_G)=\id_V\quad\text{and}\quad
D\, \phi^\pi_\lambda = \gamma^\pi(D)(\lambda)\, \phi^\pi_\lambda
\quad\text{for any }D\in \boldsymbol D^\pi.
\end{equation}
Moreover, $\phi^\pi_\lambda$ is real analytic on $G$.
Thus the elementary $\pi$-spherical functions are parameterized by $\lambda\in \lower0.8ex\hbox{$W$}\backslash \mathfrak a_\bbC^*$.
\end{thm}
The proof of the theorem is given in \S\ref{subsec:JEF}, together with two integral formulas for $\phi^\pi_\lambda$.

Now, since one easily sees from \eqref{defn:pi-sph} the restriction $\phi|_A$
of any $\phi\in C^\infty(G,\pi,\pi)$ to $A$ takes values in $\End_M V$ and since the $\bbC$-algebra $\End_M V$ is isomorphic to $\bbC$ by Schur's lemma, $\phi|_A$ is identified with a $\bbC$-valued $C^\infty$ function on $A$.
Let $\varUpsilon^\pi(\phi)\in C^\infty(\mathfrak a)$ be the composition of $\phi|_A$
with $\exp : \mathfrak a \simarrow A$.
\begin{thm}[the Chevalley restriction theorem]\label{thm:Chevalley}
The restriction map $\varUpsilon^\pi$
is a linear bijection between $C^\infty(G,\pi,\pi)$
and $C^\infty(\mathfrak a)^W$.
\end{thm}
The proof is given in \S\ref{subsec:Chevalley}.
Through this bijection,
\eqref{eq:unique-sph} 
is translated into a condition on $\varUpsilon^\pi(\phi^\pi_\lambda)$
to be a joint eigenfunction of a commuting family of differential operators on $\mathfrak a$.
The family consists of 
the \emph{$\pi$-radial part}s of $D\in\boldsymbol D^\pi$.
The $\pi$-radial part of the \emph{Casimir operator}, which has a prominent role in the family,
is expressed in a uniform way by use of a parameter $\bska^\pi$ associated with $(\pi, V)$ (Theorem \ref{thm:CasimirRad}).
The parameter $\bska^\pi$,
whose precise definition is given in \S\ref{subsec:kappa},
is a $W$-invariant function on
the set $\varSigma=\varSigma(\mathfrak g,\mathfrak a)$ of restricted roots.
(In general, a Weyl group invariant function on a root system is called a \emph{multiplicity function}.)
%All these things are discussed in \S\ref{sec:spherical}.

Of course, we could proceed further with our study
analogously to the case of the trivial $K$-type,
which would include
calculation of the \emph{$c$-function} for each individual $(\pi,V)$ by rank-one reduction.
But we take an alternative route,
attaching importance to the fact that in almost all cases the system of differential equations
for $\varUpsilon^\pi(\phi^\pi_\lambda)$ coincides with
a \emph{hypergeometric system} of Heckman and Opdam \cite{HO}
up to twist by a nowhere-vanishing function.
In general, their hypergeometric system is defined for any triple of a root system $\varSigma'$ in $\mathfrak a^*$, a multiplicity function $\bsk$ on $\varSigma'$, and $\lambda\in\mathfrak a_\bbC^*$.
($\varSigma'$ spans $\mathfrak a^*$ by definition.
Throughout the paper a root system is assumed crystallographic.)
If $\bsk$ satisfies a certain regularity condition,
their system has a unique Weyl group invariant analytic solution $F(\varSigma',\bsk,\lambda)$ such that $F(\varSigma',\bsk,\lambda;0)=1$.
The solution $F(\varSigma',\bsk,\lambda)$ is called
a \emph{hypergeometric function associated to the root system} $\varSigma'$,
which is thought of as
a deformation of the elementary spherical function for the trivial $K$-type
by an arbitrary complex root multiplicity $\bsk$.
When $\dim \mathfrak a^*=1$, $F(\varSigma',\bsk,\lambda)$
reduces to a \emph{Jacobi function},
which is studied by Koornwinder and Flensted-Jensen prior to \cite{HO} (cf.~\cite{Koornwinder}).
A Jacobi function is essentially a Gauss hypergeometric function.
We review the definition and fundamental properties of $F(\varSigma',\bsk,\lambda)$ in more detail in \S\ref{sec:HO}.

To state our main result, we fix some notation.
Let $\varSigma^+$ be the positive system corresponding to $N$.
For any $\alpha\in\varSigma$ let $\mathfrak g_\alpha$ be the restricted root space
and put $\bsm_\alpha=\dim \mathfrak g_\alpha$.
Then $\bsm:\varSigma\ni\alpha\mapsto \bsm_\alpha\in\bbZ_{>0}$ is a multiplicity function on $\varSigma$.
We put
\begin{equation}\label{eq:GKdelta}
\tilde\delta_{G/K} = \prod_{\alpha \in \varSigma^+}\biggl|\frac{\sinh \alpha}{||\alpha||}\biggr|^{\bsm_\alpha}.
\end{equation}
Here we consider $\mathfrak a$ and $\mathfrak a^*$ as inner product spaces by the Killing form $B(\cdot,\cdot)$ of $\mathfrak g$.
Likewise, for any root system $\varSigma'$ in $\mathfrak a^*$ and 
multiplicity function $\bsk$ on $\varSigma'$ we put
\begin{equation}\label{eq:HOdelta}
\tilde\delta(\varSigma',\bsk)=\prod_{\alpha\in \varSigma'^+} \biggl|\frac{\sinh(\alpha/2)}{||\alpha/2||}\biggr|^{2\bsk_\alpha}
\end{equation}
where $\varSigma'^+$ is some positive system of $\varSigma'$.
The main result of this paper is the following:
\begin{thm}\label{thm:main}
Suppose $(\pi,V)$ is a small $K$-type of
a non-compact real simple Lie group $G$ with finite center.
If $G$ is a simply-connected split Lie group $\tilde G_2$ of type $G_2$,
we further suppose $\pi$ is not the small $K$-type $\pi_2$
specified in Theorem \ref{thm:G2}.
Then there exist a root system $\varSigma^\pi$ in $\mathfrak a^*$ and a multiplicity function $\bsk^\pi$ on $\varSigma^\pi$ such that
\begin{equation}\label{eq:main}
\varUpsilon^\pi(\phi^\pi_\lambda) = \tilde\delta_{G/K}^{-\frac12}\, \tilde\delta(\varSigma^\pi,\bsk^\pi)^{\frac12}\,F(\varSigma^\pi,\bsk^\pi,\lambda)
\quad\text{for any }\lambda\in\mathfrak a_\bbC^*.
\end{equation}
\end{thm}
The proof of the theorem is divided into two large steps.
As the first step, 
we derive in \S\ref{sec:MC} a simple condition
on $\varSigma^\pi$ and $\bsk^\pi$ for the validity of \eqref{eq:main}
under the assumption that $\varSigma^\pi\subset \varSigma\cup2\varSigma$ (Proposition \ref{prop:simplified}).
This condition consists of only a few equations between $\bsk^\pi$, $\bsm$ and $\bska^\pi$.
Thus $\bska^\pi$ encodes all the information on $(\pi,V)$ needed for our purpose.
As the second step,
we determine the values of $\bska^\pi$
for each small $K$-type $(\pi,V)$ of each non-compact simple real Lie group $G$
in order to find a pair of $\varSigma^\pi$ and $\bsk^\pi$ using the condition in the first step.
The existence of such a pair
is actually confirmed in each case except $\pi_2$ of $\tilde G_2$.
In this process all small $K$-types are classified for each $G$.
This generalizes a result of Lee which
classifies small $K$-types for each \emph{split} real simple Lie group (cf.~\cite{LePhD, SWL}).
The case-by-case analysis in this step is carried out in \S\ref{sec:CC}.
We also prove in \S\ref{CC:G2} that
in the case of $\pi_2$ of $\tilde G_2$,
\eqref{eq:main} never holds for any choice of $\varSigma^\pi$ and $\bsk^\pi$.

As a detailed explanation of Theorem \ref{thm:main},
the concrete information obtained in the second step is summarized in \S\ref{sec:list}. 
That is, 
the classification of small $K$-types
and one or two possible choices of $\varSigma^\pi$ and $\bsk^\pi$
for each small $K$-type (except $\pi_2$ of $\tilde G_2$).
Now, for each our choice of $\varSigma^\pi$ and $\bsk^\pi$,
the factor $\tilde\delta_{G/K}^{-\frac12}\, \tilde\delta(\varSigma^\pi,\bsk^\pi)^{\frac12}$ in \eqref{eq:main} extends to
a nowhere-vanishing real analytic function on $\mathfrak a$.
Indeed, this can be written as a product of hyperbolic cosines (Proposition \ref{prop:deltaregular}),
whose concrete form in each case is also presented in \S\ref{sec:list}.

If $G$ is of Hermitian type,
a small $K$-type is nothing but a one-dimensional unitary representation of $K$ (\S\ref{subsec:Hermitian}, Theorem \ref{thm:Hermite}).
Hence in this case Theorem \ref{thm:main} restates results of \cite[Chapter 5]{He:white} and \cite{S:Plancherel}.
If $G$ has real rank one, then the hypergeometric function in \eqref{eq:main}
is a Jacobi function.
Hence in the rank one case, one can see 
some known results are essentially equivalent to
Theorem \ref{thm:main}.
For example,
a result of \cite{FJ} for the one-dimensional $K$-types of the universal cover of $\SU(p,1)$,
that of \cite{CP} for the lowest-dimensional non-trivial small $K$-types of $\Spin(2p,1)$ $(p\ge2)$
and that for the small $K$-types of $\Sp(p,1)$ obtained by \cite{Takahashi}, \cite{S:hyperbolic} and \cite{vDP}.

According to Oshima \cite{Oshima},
a commuting family of $W$-invariant differential operators on $\mathfrak a$
is necessarily equal to a system of hypergeometric differential operators
up to a \emph{gauge transform}
under the conditions:
\begin{enumerate*}[label=(\arabic*)]
\item\label{CIS1} $W$ is of a classical type;
\item\label{CIS2} the symbols of the operators span $S(\mathfrak a_\bbC)^W$ (\emph{complete integrability}); and
\item\label{CIS3} the operators have regular singularities at every infinity.
\end{enumerate*}
In view of this, our result is not so surprising since the family of $\pi$-radial parts of $D \in \boldsymbol D^\pi$ satisfies \ref{CIS2} and \ref{CIS3}.
Also, the exceptional case in Theorem \ref{thm:main} suggests the possibility that there might be a new class of completely integrable systems associated with the Weyl group of type $G_2$.

Now, Formula \eqref{eq:main} enables us to reduce
a large part of analytic theory for elementary $\pi$-spherical functions 
to the one for Heckman-Opdam hypergeometric functions. 
For example, \emph{Harish-Chandra's $c$-function} for $G\times_KV$ equals
a scalar multiple of the $c$-function for Heckman-Opdam hypergeometric functions (Theorem \ref{thm:2cs}),
and the \emph{$\pi$-spherical transform} (the spherical transform for $G\times_KV$) 
is a composition of a multiplication operator and
a \emph{hypergeometric Fourier transform} introduced by \cite{Op:Cherednik} (Theorem \ref{thm:2Fourier}).
Using them we can obtain
the explicit formula of Harish-Chandra's $c$-function
and the inversion formula of the $\pi$-spherical transform.
These applications are discussed in \S\ref{sec:application}.

\section{Detailed description of the main result}\label{sec:list}
In this section we list all small $K$-types for each non-compact real simple Lie group $G$ up to equivalence.
(Actually the classification is given for each real simple Lie algebra of non-compact type
since a small $K$-type for $G$
is always lifted to that for any finite cover of $G$.)
In addition, for each small $K$-type $\pi$ other than the one exception stated in \S\ref{sec:intro}, we present one or two combinations of
$\varSigma^\pi$ and $\bsk^\pi$ for which \eqref{eq:main} is valid.
Though
there may be some or infinitely many other choices of such $\varSigma^\pi$ and $\bsk^\pi$ (cf.~\S\ref{CC:complex}),
we do not pursue all the possibilities.
The results of this section will be proved in \S\ref{sec:CC}.

\subsection{\boldmath The trivial $K$-type}\label{subsec:trivial}
First of all, the trivial $K$-type $(\pi,V)$ is small for any $G$.
In this case, \eqref{eq:main} holds with $\varSigma^\pi=2\varSigma$ and $\bsk^\pi : \varSigma^\pi\ni 2\alpha\mapsto \frac{\bsm_\alpha}2$, whence $\tilde\delta_{G/K}^{-\frac12}\,\tilde\delta(\varSigma^\pi,\bsk^\pi)^{\frac12}=1$ and $\varUpsilon^\pi(\phi^\pi_\lambda)=F(\varSigma^\pi,\bsk^\pi,\lambda)$.
In the rest of this section we basically treat only non-trivial small $K$-types.

\subsection{\boldmath Simple Lie groups having no non-trivial small $K$-type}\label{subsec:onlytrivial}
There is no non-trivial small $K$-type in each of the following cases:
\begin{itemize}[leftmargin=*]
\item $G$ is a complex simple Lie group;
\item $\mathfrak g\simeq \mathfrak{sl}(p,\mathbf H)\ (p\ge2)$;
\item $\mathfrak g\simeq \mathfrak{sp}(p,q)\ (p\ge q\ge 2)$;
\item $\mathfrak g\simeq \mathfrak{so}(2r+1,1)\ (r\ge 1)$;
\item $\mathfrak g\simeq \mathfrak e_{6(-26)}\ (\text{E\,IV})$;
\item $\mathfrak g\simeq \mathfrak f_{4(-20)}\ (\mathrm{F\,II})$.
\end{itemize}

\subsection{\boldmath The case $\mathfrak g=\mathfrak{sp}(p,1)\ (p\ge1)$}\label{subsec:sp(p,1)}
Suppose $G=\Sp(p,1)\ (p \ge 1)$ and $K=\Sp(p)\times \Sp(1)$. Then $G$ is simply-connected.
Let $\pr_1$ and $\pr_2$ be the projections of $K$
to $\Sp(p)$ and $\Sp(1)$ respectively.
For the irreducible representation $(\pi_n,\bbC^n)$ of\/ $\Sp(1)\simeq \SU(2)$
of dimension $n=1,2,\dotsc$,
the $K$-type $\pi_n\circ\pr_2$ is small.
If $p=1$ then $\pi_n\circ\pr_1$ is also a small $K$-type.
There are no other small $K$-types.
Let $\varSigma=\{\pm\alpha,\pm2\alpha\}$ if $p\ge2$
and $\varSigma=\{\pm2\alpha\}$ if $p=1$.
Let $\pi=\pi_n\circ\pr_2$ if $p\ge2$ and $\pi=\pi_n\circ\pr_1$ or $\pi_n\circ\pr_2$ if $p=1$.
Then putting $\varSigma^\pi=\{\pm2\alpha,\pm4\alpha\}$,
$\bsk^\pi_{2\alpha}=2p-1\pm n$ and $\bsk^\pi_{4\alpha}=\frac12 \mp n$,
we have \eqref{eq:main} and
$\tilde\delta_{G/K}^{-\frac12}\,\tilde\delta(\varSigma^\pi,\bsk^\pi)^{\frac12}=(\cosh \alpha)^{-1\mp n}$.

\subsection{\boldmath The case $\mathfrak g=\mathfrak{so}(2r,1)\ (r\ge2)$}\label{subsec:so(2r,1)}
Suppose $G=\Spin(2r,1)\ (r \ge 2)$ and $K=\Spin(2r)$. Then $G$ is simply-connected.
For $s=0,1,2,\dotsc$,
the irreducible representation $\pi_s^\pm$ 
of $K=\Spin(2r)$
with highest weight $(s/2,\ldots,s/2,\pm s/2)$ in the standard notation is small.
There are no other small $K$-types.
Let $\varSigma=\{\pm\alpha\}$.
Let $\pi=\pi_s^\pm$.
Then putting $\varSigma^\pi=\varSigma\cup 2\varSigma =\{\pm\alpha,\pm2\alpha\}$,
$\bsk^\pi_{\alpha}=-s$ and $\bsk^\pi_{2\alpha}=r+s-\frac12$,
we have \eqref{eq:main} and
$\tilde\delta_{G/K}^{-\frac12}\,\tilde\delta(\varSigma^\pi,\bsk^\pi)^{\frac12}=(\cosh \frac\alpha2)^s$.

\subsection{\boldmath The case $\mathfrak g=\mathfrak{so}(p,q)\ (p>q\ge3)$}\label{subsec:so(p,q)}
Suppose $G$ is the double cover of $\Spin(p,q)$
$(p>q\ge3)$.
Thus $K=\Spin(p)\times\Spin(q)$ and $G$ is simply-connected.
Let $\pr_1$ and $\pr_2$ be the projections of $K$
to $\Spin(p)$ and $\Spin(q)$ respectively.
We may assume $\varSigma=\{\pm e_i\,|\,1\le i\le q\} \cup \{\pm e_i\pm e_j\,|\,1\le i<j\le q\}$
for some orthogonal basis $\{e_i\,|\,1\le i\le q\}$ of $\mathfrak a^*$
with $||e_1||=\cdots=||e_q||$.\\
\begin{enumerate*}[label=(\roman*), itemjoin=\!]
\item \label{sopq1}
Let $\sigma$ denote the spin representation of $\Spin(q)$ if $q$ is odd,
and either of two half-spin representations of $\Spin(q)$ if $q$ is even.
Then $\pi=\sigma\circ\pr_2$ is a small $K$-type.
For this $\pi$,
we can choose $\varSigma^\pi=\{\pm 2e_i\,|\,1\le i\le q\} \cup \{\pm e_i\pm e_j\,|\,1\le i<j\le q\}$ and $\bsk^\pi$ with $\bsk^\pi_{\pm2e_i}=\frac{p-q}2$ and $\bsk^\pi_{\pm e_i\pm e_j}=\frac12$
so that \eqref{eq:main} holds and
$\tilde\delta_{G/K}^{-\frac12}\,\tilde\delta(\varSigma^\pi,\bsk^\pi)^{\frac12}=\prod_{1\le i<j\le q}(\cosh \frac{e_i-e_j}2\cosh \frac{e_i+e_j}2)^{-\frac12}$.\\
\item \label{sopq2}
In addition, 
if $p$ is even and $q$ is odd,
then $\pi=\sigma\circ\pr_1$ with either half-spin representation $\sigma$ of $\Spin(p)$
is a small $K$-type.
For this $\pi$,
we can choose $\varSigma^\pi=\{\pm e_i, \pm 2e_i\,|\,1\le i\le q\}
\linebreak[1] %%%%%%%%%%%%%%%%%%%%%%%%%%%%%%%%%%%%%%%%%%%%%%%%
\cup \{\pm e_i\pm e_j\,|\,1\le 
i<j\le q\}$ and $\bsk^\pi$ with $\bsk^\pi_{\pm e_i}=p-q$, $\bsk^\pi_{\pm2e_i}=-\frac{p-q}2$ and $\bsk^\pi_{\pm e_i\pm e_j}=\frac12$
so that \eqref{eq:main} holds and
$\tilde\delta_{G/K}^{-\frac12}\,\tilde\delta(\varSigma^\pi,\bsk^\pi)^{\frac12}=
\prod_{i=1}^q (\cosh \frac{e_i}2) ^{-p+q}
\prod_{1\le i<j\le q}(\cosh \frac{e_i-e_j}2\cosh \frac{e_i+e_j}2)^{-\frac12}$.
\end{enumerate*}

There are no other non-trivial small $K$-types.

\subsection{The Hermitian type}\label{subsec:Hermitian}
Suppose $G$ is a non-compact simple Lie group of Hermitian type.
Let
\[
\varSigma_{\mathrm{long}}=\{\alpha\in\varSigma\,|\,\alpha\text{ with the longest length}\}
\quad\text{and}\quad
\varSigma_{\mathrm{middle}}=\varSigma\setminus (\tfrac12\varSigma_{\mathrm{long}}\cup \varSigma_{\mathrm{long}}).
\]
A $K$-type $(\pi,V)$ is small
if and only if $\dim V=1$.
Hence the set of small $K$-types is naturally identified with a rank one lattice in $\sqrt{-1}\mathfrak z^*$,
where $\mathfrak z$ denotes the center of $\mathfrak k$.
Let $G_\alg$ be the analytic subgroup
 for $\mathfrak g$ in the connected, simply-connected, complex Lie group with Lie algebra $\mathfrak g_\bbC$.
Let $\pi_0\in \sqrt{-1}\mathfrak z^*$ be a generator of the rank one lattice in the case when $G=G_\alg$.
Then any $K$-type $\pi$ is identified with $\nu\pi_0\in\sqrt{-1}\mathfrak z^*$
for some $\nu\in\bbQ$.
For this $\pi$,
we can choose $\varSigma^\pi=\varSigma_{\mathrm{long}}\cup 2\varSigma_{\mathrm{middle}} \cup2\varSigma_{\mathrm{long}}$ and $\bsk^\pi$ with
\[\left\{
\begin{aligned}
&\bsk^\pi_\alpha=\tfrac12 \bsm_{\frac\alpha2} \pm \nu,\quad
\bsk^\pi_{2\alpha}=\tfrac12 \mp \nu &&\text{for } \alpha\in\varSigma_{\mathrm{long}},\\
&\bsk^\pi_{2\alpha}=\tfrac12 \bsm_\alpha &&\text{for } \alpha\in\varSigma_{\mathrm{middle}}
\end{aligned}
\right.\]
so that \eqref{eq:main} holds and
$\tilde\delta_{G/K}^{-\frac12}\,\tilde\delta(\varSigma^\pi,\bsk^\pi)^{\frac12}=
\prod_{\alpha\in\varSigma_{\mathrm{long}}\cap\varSigma^+} (\cosh \frac\alpha2) ^{\mp\nu}$.

\subsection{\boldmath The case $\varSigma$ is of type $F_4$}\label{subsec:F4}
Let $G$ be a simply-connected real simple Lie group with $\varSigma$ of type $F_4$.
We exclude the complex simple Lie group of type $F_4$,
which is covered in \S\ref{subsec:onlytrivial}.
There are the following four possibilities: \smallskip
\begin{center}
\begin{tabular}{c|cccc}
$\mathfrak{g}$ &
$\mathfrak{f}_{4(4)}\ (\mathrm{F\,I})$&
$\mathfrak{e}_{6(2)}\ (\mathrm{E\,II})$&
$\mathfrak{e}_{7(-5)}\ (\mathrm{E\,VI})$&
$\mathfrak{e}_{8(-24)}\ (\mathrm{E\,IX})$\\
\hline
$\mathfrak{k}$&
$\mathfrak{sp}(3)\oplus\mathfrak{su}(2)$&
$\mathfrak{su}(6)\oplus\mathfrak{su}(2)$&
$\mathfrak{so}(12)\oplus\mathfrak{su}(2)$&
$\mathfrak{e}_7\oplus\mathfrak{su}(2)$\\
\end{tabular}
\end{center}\smallskip
Thus $K$ is the direct product of a simple compact group $K_1$ and $K_2:=\SU(2)$.
Let $\pr_2$ denote the projection $K\to \SU(2)$.
Let $(\sigma,\bbC^2)$ be the irreducible representation of $\SU(2)$
of dimension $2$.
Then $\pi:=\sigma\circ\pr_2$ is the only non-trivial small $K$-type.
Let $\varSigma_{\mathrm{short}}\sqcup\varSigma_{\mathrm{long}}$ 
be the division of $\varSigma$ according to the root lengths.
Putting $\varSigma^\pi=2\varSigma_{\mathrm{short}}\cup\varSigma_{\mathrm{long}}$,
$\bsk^\pi_{2\alpha}=\frac12\bsm_\alpha$ for $\alpha\in\varSigma_{\mathrm{short}}$
and
$\bsk^\pi_{\alpha}=\frac12$ for $\alpha\in\varSigma_{\mathrm{long}}$,
we have \eqref{eq:main} and
$\tilde\delta_{G/K}^{-\frac12}\,\tilde\delta(\varSigma^\pi,\bsk^\pi)^{\frac12}
=\prod_{\alpha\in\varSigma_{\mathrm{long}}\cap\varSigma^+}(\cosh \frac\alpha2)^{-\frac12}$.

\subsection{\boldmath Split Lie groups with simply-laced $\varSigma$}\label{subsec:split}
Let $G$ be a split real simple Lie group.
We assume its restricted root system $\varSigma$
is \emph{simply-laced} with rank $\ge2$.
We also assume $G$ is simply-connected.
(For example, if $\mathfrak g=\mathfrak{sl}(p,\bbR)$ 
then $G$ is the double cover of $\SL(p,\bbR)$.)\smallskip
\begin{center}
\begin{tabular}{c|ccccc}
$\mathfrak g$&
$\mathfrak{sl}(p,\bbR)\ (p\ge3)$&
$\mathfrak{so}(p,p)\ (p\ge3)$&
$\mathfrak e_{6(6)}\ (\text{E\,I})$&
$\mathfrak e_{7(7)}\ (\text{E\,V})$&
$\mathfrak e_{8(8)}\ (\text{E\,VIII})$\\
\hline
$\varSigma$&
$A_{p-1}$&
$D_p$&
$E_6$&
$E_7$&
$E_8$\\
$K$&
$\Spin(p)$&
$\Spin(p)\times\Spin(p)$&
$\Sp(4)$&
$\SU(8)$&
$\Spin(16)$
\end{tabular}
\end{center}
\begin{flushright}
($\mathfrak{sl}(4,\bbR)\simeq\mathfrak{so}(3,3)$)
\end{flushright}
\smallskip
\begin{thm}[{\cite[Theorem 1]{SWL}}]\label{thm:split_sl}
Let $\sigma$ be the spin representation of\/ $\Spin(p)$ ($p\ge3$) if $p$ is odd
and either of two half-spin representations of\/ $\Spin(p)$ if $p$ is even.
Then $\sigma$ is a small $K$-type when $\mathfrak g=\mathfrak{sl}(p,\bbR)$
and $\sigma\circ\pr$
is a small $K$-type when $\mathfrak g=\mathfrak{so}(p,p)$
and $\pr$ is either of two projections of $K$ to $\Spin(p)$.
If $\mathfrak g=
\mathfrak e_{6(6)}$,
then the $8$-dimensional standard (vector) representation of\/ $\Sp(4)$
is a small $K$-type.
If $\mathfrak g=
\mathfrak e_{7(7)}$,
then the $8$-dimensional standard representation of\/ $\SU(8)$
and its contragredient are small $K$-types.
If $\mathfrak g=
\mathfrak e_{8(8)}$,
then the pullback of the $16$-dimensional standard representation of\/ $\SO(16)$ to\/ $\Spin(16)$
is a small $K$-type.
There are no other non-trivial small $K$-types.
\end{thm}

Let $\pi$ be one of the small $K$-types stated above.
Then putting $\varSigma^\pi=\varSigma$ and $\bsk^\pi_\alpha=\frac12$ ($\alpha\in\varSigma$),
we have \eqref{eq:main} and
$\tilde\delta_{G/K}^{-\frac12}\,\tilde\delta(\varSigma^\pi,\bsk^\pi)^{\frac12}
=\prod_{\alpha\in\varSigma^+}(\cosh \frac\alpha2)^{-\frac12}$.

\subsection{\boldmath The split Lie group of type $G_2$}\label{subsec:G2}
Let $G$ be the simply-connected split real simple Lie group $\tilde G_2$ of type $G_2$.
Then $K=K_1\times K_2$ with $K_1\simeq K_2 \simeq \SU(2)$.
Let $\mathfrak t$ be a Cartan subalgebra of $\mathfrak k$.
The root system $\varDelta_{\mathfrak k}$ for $(\mathfrak k_\bbC,\mathfrak t_\bbC)$
is written as $\{\pm\alpha_1\}\sqcup\{\pm\alpha_2\}$.
We assume $\{\pm\alpha_i\}$ is the root system of $((\mathfrak k_i)_\bbC,(\mathfrak t\cap\mathfrak k_i)_\bbC)$ ($i=1,2$)
and $||\alpha_1||<||\alpha_2||$ with respect to the norm induced from $B(\cdot,\cdot)$.
Let $\pr_i$ be the projection of $K$ to $K_i$ ($i=1,2$).
\begin{thm}[{\cite[Theorem 1]{SWL}}]\label{thm:G2}
Let $\sigma$ denote the irreducible representation of\/ $\SU(2)$ of dimension $2$.
Then both $\pi_1:=\sigma\circ\pr_1$ and  $\pi_2:=\sigma\circ\pr_2$
are small $K$-types.
There are no other non-trivial small $K$-types.
\end{thm}
If $\pi=\pi_1$, then
putting $\varSigma^\pi=\varSigma$ and $\bsk^\pi_\alpha=\frac12$ ($\alpha\in\varSigma$),
we have \eqref{eq:main} and
$\tilde\delta_{G/K}^{-\frac12}\,\tilde\delta(\varSigma^\pi,\bsk^\pi)^{\frac12}
=\prod_{\alpha\in\varSigma^+}(\cosh \frac\alpha2)^{-\frac12}$.
In contrast,
if $\pi=\pi_2$, then
\eqref{eq:main} never holds for any choice of $\varSigma^\pi$ and $\bsk^\pi$.

\section{Spherical functions}\label{sec:spherical}
Suppose $(\pi,V)$ is a small $K$-type of
a connected real semisimple Lie group $G=KAN$ with finite center.
In this section we study general properties of (elementary) $\pi$-spherical functions.

\subsection{The Chevalley restriction theorem}\label{subsec:Chevalley}
Let $\theta$ be the Cartan involution corresponding to $K$.
The differential of $\theta$ for $\mathfrak g$ is denoted by the same symbol.
Let $\mathfrak g=\mathfrak k+\mathfrak s$ be the Cartan decomposition by $\theta$.

Let us prove Theorem \ref{thm:Chevalley}.
Suppose first $\phi\in C^\infty(G,\pi,\pi)$.
Then $\varUpsilon^\pi(\phi)$ is $W$-invariant
since for $w=\tilde w M \in W$ $(\tilde w\in\tilde M)$ and $H \in \mathfrak a$ we have
\begin{align*}
\varUpsilon^\pi(\phi)(w H)
&=
\phi(\tilde w (\exp H) \tilde w^{-1})
=
\pi(\tilde w)\phi( \exp H)\pi( \tilde w^{-1})\\
&=
\phi( \exp H)\pi(\tilde w)\pi( \tilde w^{-1})
=
\varUpsilon^\pi(\phi)(H).
\end{align*}
Here we note $\phi( \exp H)$ is a scalar operator.
This shows $\varUpsilon^\pi$ maps $C^\infty(G,\pi,\pi)$ into $C^\infty(\mathfrak a)^W$.
The injectivity of $\varUpsilon^\pi$ is clear from $G=KAK$.
In order to show the surjectivity, take any $f\in C^\infty(\mathfrak a)^W$.
Then by the ordinary Chevalley restriction theorem (\cite{Da}), $f$ extends to some $\tilde f \in C^\infty(\mathfrak s)^K$.
Now $\phi(k\exp X):=\tilde f(X)\pi(k^{-1})$ ($k\in K,X\in\mathfrak s$) is a $\pi$-spherical function
satisfying $\varUpsilon^\pi(\phi)=f$.
The theorem is thus proved.

\subsection{The algebra of invariant differential operators}\label{subsec:Dpi}
Let $G\times_KV$ be the homogeneous vector bundle on $G/K$ associated to $(\pi,V)$.
The space $C^\infty(G\times_KV)$ of $C^\infty$ sections of $G\times_KV$ is identified with
\[
\bigl\{
f:G\xrightarrow{C^\infty} V\,\bigm|\,
f(xk)=\pi(k^{-1})f(x)\quad\text{for }x\in G\text{ and }k\in K
\bigr\},
\]
on which $g\in G$ acts by $\ell(g):f\mapsto f(g^{-1}\cdot)$.
Accordingly, $\Hom_K(V, C^\infty(G\times_KV))$
is identified with $C^\infty(G,\pi,\pi)$ by
\begin{align*}
\Hom_K(V, C^\infty(G\times_KV))
&\simeq
(C^\infty(G\times_KV) \otimes V^* )^K\\
&\simeq 
\{\phi : G\xrightarrow{C^\infty} V\otimes V^*\,|\,
\phi(k_1gk_2) = \pi(k_2^{-1}) \otimes \pi^*(k_1) \,\phi(g)
\}\\
&\simeq C^\infty(G,\pi,\pi).
\end{align*}
Here $(\pi^*,V^*)$ is the contragredient representation of $(\pi,V)$.
Let $U(\mathfrak g_\bbC)$ be the complexified universal enveloping algebra of $\mathfrak g$ and $U(\mathfrak g_\bbC)^K$ its subalgebra consisting of the $\Ad(K)$-invariant elements.
As a left-invariant differential operator on $G$,
each element of $U(\mathfrak g_\bbC)^K$ acts on $C^\infty(G\times_KV)$.
This defines a homomorphism from $U(\mathfrak g_\bbC)^K$
to the algebra $\boldsymbol D^\pi$ of $\ell(G)$-invariant differential operators on $G\times_KV$.
It is known that the homomorphism is onto and its kernel equals $U(\mathfrak g_\bbC)^K \cap U(\mathfrak g_\bbC) \Ker \pi^*$ (cf.~\cite{Deit}).
Here $\Ker \pi^*$ denotes the kernel of $\pi^*:U(\mathfrak k_\bbC)\to \End_\bbC V^*$.
Thus we have $\boldsymbol D^\pi \simeq U(\mathfrak g_\bbC)^K\bigm/U(\mathfrak g_\bbC)^K \cap U(\mathfrak g_\bbC) \Ker \pi^*$.
Note that $U(\mathfrak g_\bbC)^K$ and $\boldsymbol D^\pi$ act on $C^\infty(G,\pi,\pi)$ naturally.

Now $U(\mathfrak g_\bbC)\simeq  U(\mathfrak n_\bbC)\otimes U(\mathfrak a_\bbC)\otimes U(\mathfrak k_\bbC)$ by the Poincar\'e-Birkhoff-Witt theorem.
Let us consider the following linear map:
\begin{align*}
\gamma^\pi : U(\mathfrak g_\bbC) &\simeq 1\otimes U(\mathfrak a_\bbC)\otimes U(\mathfrak k_\bbC) \oplus \mathfrak n_\bbC U(\mathfrak g_\bbC)\\
&\xrightarrow{\text{projection}}
1\otimes U(\mathfrak a_\bbC)\otimes U(\mathfrak k_\bbC)
\simeq
S(\mathfrak a_\bbC)\otimes U(\mathfrak k_\bbC)\\
&\xrightarrow{(f(\lambda)\mapsto f(\lambda+\rho))\otimes \pi^*}
S(\mathfrak a_\bbC)\otimes \End_\bbC V^*,
\end{align*}
where $\rho:=\frac12\sum_{\alpha\in\varSigma^+}\bsm_\alpha \alpha \in \mathfrak a^*$.
Note that $\gamma^\pi(D) \in S(\mathfrak a_\bbC) \otimes \End_M V^*$
for any $D\in U(\mathfrak g_\bbC)^M$
and that $S(\mathfrak a_\bbC) \otimes \End_M V^* \simeq S(\mathfrak a_\bbC) \otimes \bbC \simeq S(\mathfrak a_\bbC)$ since $(\pi^*,V^*)$ is also a small $K$-type.
Hence
we consider $\gamma^\pi(D) \in S(\mathfrak a_\bbC)$ when $D \in U(\mathfrak g_\bbC)^M$. The following generalization of Harish-Chandra's celebrated exact sequence is given in \cite[\S11.3.3]{Wal}:
\begin{thm}\label{thm:HChomo}
The restriction of $\gamma^\pi$ to $U(\mathfrak g_\bbC)^K$ is an algebra homomorphism
into $S(\mathfrak a_\bbC)^W$
and the sequence
\[
0\to U(\mathfrak g_\bbC)^K \cap U(\mathfrak g_\bbC) \Ker\pi^*
\to U(\mathfrak g_\bbC)^K
\xrightarrow{\gamma^\pi} S(\mathfrak a_\bbC)^W \to 0
\]
is exact.
\end{thm}
This in particular induces an algebra isomorphism $\boldsymbol D^\pi\simarrow S(\mathfrak a_\bbC)^W$ (also denoted by $\gamma^\pi$).

\subsection{Joint eigenfunctions and integral formulas}\label{subsec:JEF}
Suppose $\lambda\in\mathfrak a_\bbC^*$ and put
\[\mathscr A(G\times_K V,\lambda) = \{f \in C^\infty(G\times_KV)\,|\, Df=\gamma^\pi(D)(\lambda)\,f\quad \text{for any }D\in\boldsymbol D^\pi\}.\]
This is a $\ell(G)$-submodule of the space $\mathscr A(G\times_K V)$
of the real analytic sections of $G\times_K V$.
In fact, 
letting $\varOmega_{\mathfrak g}$ and $\varOmega_{\mathfrak k}$ respectively
the Casimir elements of $U(\mathfrak g_\bbC)$ and $U(\mathfrak k_\bbC)$ relative to the Killing form $B(\cdot,\cdot)$,
we see $\mathscr A(G\times_K V,\lambda)$ is annihilated by
the elliptic operator $\varOmega_{\mathfrak g}-2\varOmega_{\mathfrak k}-\gamma^\pi(\varOmega_{\mathfrak g}-2\varOmega_{\mathfrak k})(\lambda)$ on $G$.
Any $\pi$-spherical function $\phi^\pi_\lambda$ satisfying \eqref{eq:unique-sph}
is an $\End_\bbC V$-valued real analytic function since it is identified with an element of $\Hom_K(V, \mathscr A(G\times_K V,\lambda))$.

We shall prove Theorem \ref{thm:unique-sph} with the following integral formula:
\begin{equation}
\phi^\pi_\lambda(g)=\int_K e^{(\lambda+\rho)(A(kg))}\, \pi(u(kg)^{-1}k)\,dk\label{eq:integral-rep}
\end{equation}
Here 
for $g\in G$ we define $A(g)\in\mathfrak a$ and $u(g)\in K$ to be the unique elements such that $g\in N\exp A(g)\,u(g)$.
It is easy to see $\phi^\pi_\lambda$ defined by \eqref{eq:integral-rep} is a $\pi$-spherical function.
We claim this $\phi^\pi_\lambda$ satisfies \eqref{eq:unique-sph}.
Indeed, it is clear that $\phi^\pi_\lambda(1_G)=\id_V$.
Also, one easily sees
$\psi_\lambda^\pi(g):= e^{(\lambda+\rho)(A(g))}\, \pi^*(u(g))$ is an $\End_\bbC V^*$-valued $C^\infty$ function on $G$ satisfying
\begin{align*}
\psi_\lambda^\pi(gk_2)&=\psi_\lambda^\pi(g)\,\pi^*(k_2)\quad\text{for }(g,k_2)\in G\times K,\\
\psi_\lambda^\pi(namg)&=e^{(\lambda+\rho)(\log a)}\pi^*(m)\psi_\lambda^\pi(g)
\quad\text{for }(n,a,m,g)\in N\times A\times M\times G,\\
D\,\psi_\lambda^\pi&=\gamma^\pi(D)(\lambda)\,\psi_\lambda^\pi
\quad\text{for }D\in U(\mathfrak g_\bbC)^K.
\end{align*}
The claim follows since $\phi_\lambda^\pi(g)=\int_K \trans \psi_\lambda^\pi(kg)\pi(k)dk$.

To complete the proof,
we must show the uniqueness of $\phi^\pi_\lambda$.
To do so, suppose $\phi$ is a $\pi$-spherical functions satisfying \eqref{eq:unique-sph}.
Define a linear map $\varPhi : U(\mathfrak g_\bbC) \ni D \mapsto D(\phi^\pi_\lambda-\phi)(1_G) \in \End_\bbC V$.
Then 
we have
\begin{equation}\label{eq:XD}
\varPhi(XD)=-\varPhi(D)\pi(X)\text{ and }\varPhi(DX)=-\pi(X)\varPhi(D)
\text{ for any }X\in\mathfrak k_\bbC\text{ and }D\in U(\mathfrak g_\bbC).
\end{equation}
This implies $\varPhi([X,D])=[\pi(X),\varPhi(D)]$
and hence $\varPhi(\Ad(k)D)=\pi(k)\varPhi(D)\pi(k^{-1})$ for $k\in K$.
Since $\pi(U(\mathfrak k_\bbC))=\End_\bbC V$ by Burnside's theorem,
\eqref{eq:XD} also implies $\varPhi(U(\mathfrak g_\bbC))$ is
a two-sided ideal of $\End_\bbC V$.
Hence $\varPhi(U(\mathfrak g_\bbC))$ is either $\End_\bbC V$ or $\{0\}$.
In order to show the former case never happens, assume
$\varPhi(D)=\id_V$ for some $D$.
Then with $\tilde D:=\int \Ad(k)D\,dk \in U(\mathfrak g_\bbC)^K$ we have
$\varPhi(\tilde D)=\int \pi(k)\varPhi(D)\pi(k^{-1})\,dk=\int_K\id_V\,dk=\id_V$.
On the other hand,
$\varPhi(\tilde D)=\tilde D(\phi^\pi_\lambda-\phi)(1_G)
=\gamma^\pi(\tilde D)(\lambda)\, \phi^\pi_\lambda(1_G) - \gamma^\pi(\tilde D)(\lambda)\, \phi(1_G)=\gamma^\pi(\tilde D)(\lambda) \id_V - \gamma^\pi(\tilde D)(\lambda) \id_V=0$, contradicting the last calculation.
Thus $\varPhi(U(\mathfrak g_\bbC))=\{0\}$ and $D(\phi^\pi_\lambda-\phi)(1_G)=0$
for any $D\in U(\mathfrak g_\bbC)$.
Since $\phi^\pi_\lambda-\phi$ is real analytic,
we have $\phi^\pi_\lambda-\phi=0$, the desired uniqueness.

\begin{prop}
For any $\lambda\in\mathfrak a_\bbC^*$ one has
\begin{align}
\phi^{\pi}_\lambda(\theta g)&=\phi^\pi_{-\lambda} (g)\qquad(g\in G),
\label{eq:thetasph}\\
\varUpsilon^\pi(\phi^{\pi}_{\lambda})(-H)&=
\varUpsilon^\pi(\phi^{\pi}_{-\lambda})(H)\qquad(H\in\mathfrak a).
\label{eq:sphnegative}
\end{align}
\end{prop}
\begin{proof}
Note $\phi^{\pi}_\lambda(\theta g)\in C^\infty(G,\pi,\pi)$. 
Take $\tilde w_0\in\tilde M$ so that $w_0=\tilde w_0 M$ is the longest element of $W$.
Then it easily follows from the definition of $\gamma^\pi$
that $\gamma^\pi(\theta D)(\lambda)=\gamma^\pi(\theta\Ad(\tilde w_0)D)(\lambda)=\gamma^\pi(D)(-w_0\lambda)=\gamma^\pi(D)(-\lambda)$ for any $D\in U(\mathfrak g_\bbC)^K$.
Thus Theorem \ref{thm:unique-sph} implies \eqref{eq:thetasph} and hence
\eqref{eq:sphnegative}.
\end{proof}

\begin{prop}
For any $\lambda\in\mathfrak a_\bbC^*$ one has
\begin{align}
\trans \phi^\pi_{\lambda} (g^{-1})&=\phi^{\pi^*}_{-\lambda}(g)\qquad(g\in G),
\label{eq:dualsph}\\
\varUpsilon^\pi(\phi^{\pi}_{\lambda})&=
\varUpsilon^{\pi^*}(\phi^{\pi^*}_{\lambda}).
\label{eq:dualsphrest}
\end{align}
\end{prop}
\begin{proof}
We use $\psi^{\pi}_\lambda(g)$ defined above.
For each $g\in G$ put
$\varPsi_g(x):=\trans \psi_{-\lambda}^{\pi^*}(xg)\, \psi_\lambda^\pi(x)$.
Since
\[
\varPsi_g(namx)=e^{2\rho(\log a)}\varPsi_g(x)
\quad\text{for }(n,a,m,x)\in N\times A\times M\times G,
\]
we have $\int_K \varPsi_g(k)dk=\int_K \varPsi_g(kg^{-1})dk$
by \cite[Ch.~I, Lemma 5.19]{Hel2}.
Hence by \eqref{eq:integral-rep}
\begin{align*}
\phi^{\pi^*}_{-\lambda}(g)
&=\int_K \trans \psi_{-\lambda}^{\pi^*}(kg)\psi_\lambda^\pi(k)dk 
=\int_K \trans \psi_{-\lambda}^{\pi^*}(k)\psi_\lambda^\pi(kg^{-1})dk\\
&={\vphantom{\Bigm|}}^{\mathrm t}\biggl(
\int_K \trans \psi_\lambda^\pi(kg^{-1})\psi_{-\lambda}^{\pi^*}(k)dk 
\biggr) 
=\trans \phi^\pi_{\lambda} (g^{-1}).
\end{align*}
Thus we get  \eqref{eq:dualsph}.

Now for any $a\in A$ it follows
from \eqref{eq:thetasph} and \eqref{eq:dualsph} that
\[
\phi_\lambda^\pi(a)
=
\phi_{-\lambda}^\pi(a^{-1})
=\trans\phi_{\lambda}^{\pi^*}(a),
\]
which proves \eqref{eq:dualsphrest}.
\end{proof}

\begin{rem}
The uniqueness of $\phi^\pi_\lambda$ is also obtained by general theory
(cf.~\cite[Theorem 3.8]{Camporesi2}, \cite[Theorem 1.4.5]{GV}).
By use of \eqref{eq:dualsph} we can rewrite \eqref{eq:integral-rep} as
\begin{equation}
\phi^\pi_\lambda(g)=
\int_Ke^{(\lambda-\rho)(H(gk))}\,\pi(k\kappa(gk)^{-1})\,dk. 
\label{eq:eisenstein}
\end{equation}
Here given $x\in G$, define $\kappa(x)\in K$ and $H(x)\in\mathfrak{a}$ by $x\in \kappa(x)e^{H(x)}N$. 
Formula \eqref{eq:integral-rep} (or \eqref{eq:eisenstein}) is a special case of the integral representations of elementary spherical functions (or more generally \emph{Eisenstein integral}s) given by Harish-Chandra (cf. \cite[\S 6.2.2, \S 9.1.5]{War}, \cite[(42)]{Camporesi2}, \cite[(14.20)]{Kn0}).
\end{rem}

\subsection{\boldmath The multiplicity function $\bska^\pi$}\label{subsec:kappa}

For any $\lambda \in \mathfrak a^*$ let $H_\lambda\in\mathfrak a$ be the unique element
such that $\lambda=B(H_\lambda,\cdot)$.
For each $\alpha\in\varSigma$ we choose an orthonormal basis $\bigl\{X_\alpha^{(1)},\ldots,X_\alpha^{(\bsm_\alpha)}\bigr\}$ of $\mathfrak g_\alpha$
with respect to the inner product $-\frac{||\alpha||^2}{2}B(\cdot,\theta \cdot)$.
Note that $\bigl[X_\alpha^{(i)},\theta X_\alpha^{(i)}\bigr] = -\alpha^\vee$ for $i=1,\ldots,\bsm_\alpha$ ($\alpha^\vee:=\frac{2H_\alpha}{||\alpha||^2}$).
\begin{lem}\label{lem:kappadef}
The element $\sum_{i=1}^{\bsm_\alpha}\bigl(X_\alpha^{(i)}+\theta X_\alpha^{(i)}\bigr)^2 \in U(\mathfrak g_\bbC)$ does not depend on the choice of $\bigl\{X_\alpha^{(i)}\bigr\}$ and is $\Ad(M)$-invariant. 
Moreover for any $\tilde w \in \tilde M$
we have $\Ad(\tilde w)\sum_{i=1}^{\bsm_\alpha}\bigl(X_\alpha^{(i)}+\theta X_\alpha^{(i)}\bigr)^2=\sum_{i=1}^{\bsm_{w\alpha}}\bigl(X_{w\alpha}^{(i)}+\theta X_{w\alpha}^{(i)}\bigr)^2$ with $w:=\tilde w M \in W$.
\end{lem}

\begin{proof}
The element $\sum_{i=1}^{\bsm_\alpha}\bigl(X_\alpha^{(i)}\bigr)^{\otimes2}\in \mathfrak g_\alpha^{\otimes2}$ is independent of the choice of $\bigl\{X_\alpha^{(i)}\bigr\}$
and is $M$-invariant since $M$ acts isometrically on $\mathfrak g_\alpha$.
Now the first assertion follows
since $\mathfrak g_\alpha^{\otimes2}\ni X\otimes Y \mapsto (X+\theta X)(Y+\theta Y) \in U(\mathfrak g_\bbC)$ is $M$-linear.
The second assertion is also immediate since $\bigl\{\Ad(\tilde w)X_\alpha^{(i)}\bigr\}$ is an orthonormal basis of $\mathfrak g_{w\alpha}$.
\end{proof}

Since $\pi$ is small, it follows from 
Schur's lemma that $\frac1{\bsm_\alpha}\sum_{i=1}^{\bsm_\alpha}\pi\bigl(X_\alpha^{(i)}+\theta X_\alpha^{(i)}\bigr)^2$ is a scalar operator.
We denote this value by $\bska^\pi_\alpha$, namely
\begin{equation*}
\bska^\pi_\alpha \id_V = \frac1{\bsm_\alpha}\sum_{i=1}^{\bsm_\alpha}\pi\bigl(X_\alpha^{(i)}+\theta X_\alpha^{(i)}\bigr)^2.
\end{equation*}
By the second statement of Lemma \ref{lem:kappadef}, $\varSigma\ni\alpha\mapsto \bska^\pi_\alpha$ is a multiplicity function.
The next lemma is useful to calculate $\bska^\pi_\alpha$'s in various examples:

\begin{lem}\label{lem:kappa}
Suppose $\alpha\in\varSigma$.
For any $X_\alpha \in\mathfrak g_\alpha$ such that
$-\frac{||\alpha||^2}2 B(X_\alpha,\theta X_\alpha)=1$
it holds that
\[
\bska_\alpha^\pi=\frac{1}{\dim V} \trace (\pi(X_\alpha+\theta X_\alpha)^2).
\]
\end{lem}

\begin{proof}
By the following theorem,
$\trace (\pi(X_\alpha+\theta X_\alpha)^2)$
takes a constant value for any $X_\alpha$ in the unit sphere of $\mathfrak g_\alpha$.
The rest of the proof is easy.
\end{proof}

\begin{thm}[{\cite[Theorem 2.1.7]{Ko}}]\label{thm:Ko}
For any $\alpha\in\ \varSigma$ with $\bsm_\alpha>1$,
$M_0$ (the identity component of $M$) acts transitively on the unit sphere of\/ $\mathfrak g_\alpha$.
\end{thm}
Note $M_0$ acts trivially on $\mathfrak g_\alpha$ for $\alpha\in\ \varSigma$ with $\bsm_\alpha=1$.
\begin{prop}\label{prop:ND}
For any $\alpha\in \varSigma$, $\bska^\pi_\alpha\le0$.
Furthermore,
$\bska^\pi_\alpha=0$ if and only if 
$\pi(X_\alpha+\theta X_\alpha)=0$ for any $X_\alpha\in\mathfrak g_\alpha$.
\end{prop}
\begin{proof}
Suppose $\alpha\in\varSigma$ and take $X_\alpha\in\mathfrak g_\alpha$ so that $-B(X_\alpha,\theta X_\alpha)=\frac2{||\alpha||^2}$.
Then $\pi(X_\alpha+\theta X_\alpha)$ is diagonalizable
since it is skew-Hermitian 
with respect to the inner product of $V$.
With a basis $\{v_1,\ldots,v_n\}$ of $V$
and real numbers $\lambda_1,\ldots,\lambda_n$, write
\[
\pi(X_\alpha+\theta X_\alpha) v_j=\sqrt{-1}\lambda_j v_j\qquad(j=1,\ldots,n).
\]
From Lemma \ref{lem:kappa} we have
\[
\bska_\alpha^\pi=-\frac{1}{\dim V} \sum_{j=1}^n \lambda_j^2.
\]
Hence $\bska^\pi_\alpha\le0$ and the equality holds if and only if
$\pi(X_\alpha+\theta X_\alpha)=0$.
This, together with Theorem \ref{thm:Ko}, proves the proposition.
\end{proof}

\subsection{The radial part of the Casimir operator}

Let $\varOmega_{\mathfrak m}$ and $\varOmega_{\mathfrak a}$ be the Casimir elements of $U(\mathfrak m_\bbC)$ and $U(\mathfrak a_\bbC)$ relative to $B(\cdot,\cdot)$, respectively.
Note $\pi(\varOmega_{\mathfrak m})$ is a scalar operator.
We denote its value by $\varpi^\pi$.
\begin{thm}\label{thm:CasimirRad}
For any $\phi\in C^\infty(G,\pi,\pi)$ it holds that
\begin{equation*}
\varUpsilon^\pi((\varOmega_{\mathfrak g}-\varpi^\pi)\,\phi)
=
\biggl(
\varOmega_{\mathfrak a} + 
\sum_{\alpha\in\varSigma^+}
\bsm_\alpha
\biggl(
\coth\alpha\,H_\alpha
-
\frac{\bska_\alpha^\pi||\alpha||^2}{4\cosh^2\frac{\alpha}2}
\biggr)\biggr)\,
\varUpsilon^\pi(\phi)
\end{equation*}
on $\mathfrak a_\reg:=\{H\in\mathfrak a\,|\,\alpha(H)\ne0\text{ for any }\alpha\in\varSigma\}$.
\end{thm}
\begin{proof}
Suppose $H\in\mathfrak a_\reg$.
It follows from \cite[Lemma 22]{HC2} and \cite[Proposition 9.1.2.11]{War} (see also \cite[Proposition 2.3]{S:Plancherel}) that the equality 
\begin{multline*}
\varOmega_{\mathfrak g}
=
\varOmega_{\mathfrak a}
+
\varOmega_{\mathfrak m}
+
\sum_{\alpha\in\varSigma^+}
\biggl(
\bsm_\alpha\coth\alpha(H)\, H_\alpha
+
\frac{||\alpha||^2}{4}\sum_{i=1}^{\bsm_\alpha}
\biggl(
\frac1{\sinh^2 \alpha(H)}\, \bigl(X_\alpha^{(i)}+\theta X_\alpha^{(i)}\bigr)^2 \\
-2 \frac{\coth \alpha(H)}{ \sinh \alpha(H)}\, \bigl(\Ad(e^{-H})\bigl(X_\alpha^{(i)}+\theta X_\alpha^{(i)}\bigr)\bigr) \bigl(X_\alpha^{(i)}+\theta X_\alpha^{(i)}\bigr)\\
+ \frac1{\sinh^2 \alpha(H)}\, \bigl(\Ad(e^{-H})\bigl(X_\alpha^{(i)}+\theta X_\alpha^{(i)}\bigr)\bigr)^2
\biggr)
\biggr)
\end{multline*}
holds in $U(\mathfrak g_\bbC)$.
Since $\phi(e^H)$ is a scalar operator, we have for each $\alpha\in\varSigma^+$
\begin{align*}
\sum_{i=1}^{\bsm_\alpha}
\bigl(
\bigl(\Ad(e^{-H})\bigl(X_\alpha^{(i)}+\theta X_\alpha^{(i)}\bigr)\bigr)
& \bigl(X_\alpha^{(i)}+\theta X_\alpha^{(i)}\bigr)\,
\phi
\bigr)(e^H)\\
&= \sum_{i=1}^{\bsm_\alpha}
\pi\bigl(X_\alpha^{(i)}+\theta X_\alpha^{(i)}\bigr)\,
\phi(e^H)\,
\pi\bigl(X_\alpha^{(i)}+\theta X_\alpha^{(i)}\bigr) \\
&=\phi(e^H)\,
\sum_{i=1}^{\bsm_\alpha} \pi\bigl(X_\alpha^{(i)}+\theta X_\alpha^{(i)}\bigr)^2\\
&=\bsm_\alpha\,\bska^\pi_\alpha \,\varUpsilon^\pi(\phi)(H)
\end{align*}
and in the same way
\begin{align*}
\sum_{i=1}^{\bsm_\alpha}
\bigl(
\bigl(X_\alpha^{(i)}+\theta X_\alpha^{(i)}\bigr)^2\,
\phi
\bigr)(e^H)
=
\sum_{i=1}^{\bsm_\alpha}
\bigl(
\bigl(\Ad(e^{-H})\bigl(X_\alpha^{(i)}+\theta X_\alpha^{(i)}\bigr)\bigr)^2\,
\phi
\bigr)(e^H)
=
\bsm_\alpha\,\bska^\pi_\alpha \,\varUpsilon^\pi(\phi)(H).
\end{align*}
Hence we calculate
\begin{align*}
\varUpsilon^\pi((\varOmega_{\mathfrak g}&-\varpi^\pi)\,\phi)(H)
-\varOmega_{\mathfrak a}\varUpsilon^\pi(\phi)(H)\\
&=
((\varOmega_{\mathfrak g}-\varOmega_{\mathfrak a}-\varOmega_{\mathfrak m})\,\phi)(e^H)\\
&=
\sum_{\alpha\in\varSigma^+} \bsm_\alpha\biggl(
\coth \alpha(H)\,H_\alpha\\
&\qquad\qquad
+
\frac{\bska^\pi_\alpha||\alpha||^2}{4}
\biggl(
\frac1{\sinh^2 \alpha(H)}
-2 \frac{\coth \alpha(H)}{ \sinh \alpha(H)}
+ \frac1{\sinh^2 \alpha(H)}
\biggr) 
\biggr)\, \varUpsilon^\pi(\phi)(H)\\
&=
\sum_{\alpha\in\varSigma^+} \bsm_\alpha\biggl(
\coth \alpha(H)\,H_\alpha
-
\frac{\bska^\pi_\alpha||\alpha||^2}{4\cosh^2\frac{\alpha(H)}2}
\biggr)\, \varUpsilon^\pi(\phi)(H).\qedhere
\end{align*}
\end{proof}

\subsection{Radial parts of general invariant differential operators}\label{subsec:RP}

Let $\mathscr R$ be the unital algebra of functions on $\mathfrak a_\reg$
generated by $(1\pm e^{\alpha})^{-1}$ ($\alpha\in\varSigma^+$).
The Weyl group $W$ acts on $\mathscr R$ naturally.
The algebra of differential operators on $\mathfrak a_\reg$
with coefficients in $\mathscr R$ is identified with $\mathscr R \otimes S(\mathfrak a_\bbC)$ as a linear space.
Let $\mathfrak a_-:=\{H\in\mathfrak a\,|\,\alpha(H)<0\text{ for any }\alpha\in\varSigma^+\}$ and $\bbN\varSigma^+:=\{\sum_{\alpha\in\varSigma^+} n_\alpha \alpha \,|\, n_\alpha\in\bbN \}$.
(Here $\bbN=\{0,1,2,\ldots\}$.)
Then each element $f\in \mathscr R$ is expanded
in a power series $\sum_{\mu\in \bbN\varSigma^+} a_\mu e^\mu$ which absolutely converges on $\mathfrak a_-$.
We define $\mathscr M$ to be the maximal ideal of $\mathscr R$
consisting of the power series without constant term.
\begin{prop}\label{prop:GeneralRad}
For any $D \in U(\mathfrak g_\bbC)^K$
there exists a unique differential operator $\varDelta^\pi(D) \in \mathscr R \otimes S(\mathfrak a_\bbC)$
such that for any $\phi \in C^\infty(G,\pi,\pi)$
\begin{equation}\label{RadFormula}
\varUpsilon^\pi(D\,\phi)
=\varDelta^\pi(D)\,\varUpsilon^\pi(\phi)
\end{equation}
on $\mathfrak a_\reg$.
Moreover, $\varDelta^\pi(D)$ is $W$-invariant and
is of the form
\begin{equation}\label{eq:GeneralRad}
\varDelta^\pi(D)=\gamma^\pi(D)(\cdot-\rho) + \sum_{j=1}^k f_j\, E_j
\quad\text{with }f_1,\ldots,f_k \in \mathscr M
\text{ and }E_1,\ldots,E_k \in S(\mathfrak a_\bbC).
\end{equation}
\end{prop}
\begin{rem}
We call $\varDelta^\pi(D)$ in the theorem the $\pi$-radial part of $D$.
It follows from Theorem \ref{thm:CasimirRad} that
\begin{equation}\label{eq:CasimirRad}
\varDelta^\pi(\varOmega_{\mathfrak g}-\varpi^\pi)
=
\varOmega_{\mathfrak a} + 
\sum_{\alpha\in\varSigma^+}
\bsm_\alpha
\biggl(
\coth\alpha\,H_\alpha
-
\frac{\bska_\alpha^\pi||\alpha||^2}{4\cosh^2\frac{\alpha}2}
\biggr).
\end{equation}
\end{rem}
\begin{proof}[Proof of Proposition \ref{prop:GeneralRad}]
First,
we claim that for an arbitrary $D\in U(\mathfrak g_\bbC)$
(not necessarily $K$-invariant)
there exist
$f_j \in \mathscr M$,
$E'_j \in S(\mathfrak a_\bbC)$
and $U_j, U'_j \in U(\mathfrak k_\bbC)$ ($j=1,\ldots,k$)
such that
\begin{equation}\label{eq:D_expre}
D=D' + \sum_{j=1}^k
f_j(H)\, (\Ad(e^{-H})(U_j))\,E'_j\,U'_j
\quad
\text{for any }H\in\mathfrak a_\reg,
\end{equation}
where
$D' \in U(\mathfrak a_\bbC) \otimes U(\mathfrak k_\bbC)$
is a unique element such that $D-D'\in\mathfrak n_\bbC U(\mathfrak g_\bbC)$.
Indeed, since for each $\alpha \in \varSigma^+$ and $i=1,\ldots,\bsm_\alpha$
one has
\[
X_\alpha^{(i)}
=\frac{e^{\alpha(H)}}{1-e^{2\alpha(H)}}
\Ad(e^{-H})(X_\alpha^{(i)}+\theta X_\alpha^{(i)})
-\frac{e^{2\alpha(H)}}{1-e^{2\alpha(H)}} (X_\alpha^{(i)} + \theta X_\alpha^{(i)}),
\]
the claim can be
easily shown by induction on the degree of $D$.

Next, 
for any $H\in\mathfrak a_\reg$ the linear map
\[
\eta_H : S(\mathfrak a_\bbC) \otimes (U(\mathfrak k_\bbC)\otimes_{U(\mathfrak m_\bbC)}U(\mathfrak k_\bbC))
\ni E' \otimes U \otimes U'
\mapsto
(\Ad(e^{-H})(U))\, E'\, U' \in U(\mathfrak g_\bbC)
\]
is a well-defined bijection.
Note $\eta_H$ is an $M$-homomorphism.
Now suppose $D\in U(\mathfrak g_\bbC)^K$.
Since $D$ and $D'$ are $M$-invariant,
we have
\[
\sum_{j=1}^k
f_j(H)\,E'_j \otimes U_j\otimes U'_j
=
\eta_H^{-1}(D-D') \in
S(\mathfrak a_\bbC) \otimes (U(\mathfrak k_\bbC)\otimes_{U(\mathfrak m_\bbC)}U(\mathfrak k_\bbC))^M.
\]
For each $j=1,\ldots, k$, take $U_{j,\nu}$ and $U'_{j,\nu} \in U(\mathfrak k_\bbC)$ ($\nu=1,\ldots,k_j$)
so that
\[
\int_M \Ad(m) (U_j)\otimes \Ad(m) (U_j')\,dm
=\sum_{\nu=1}^{k_j} U_{j,\nu} \otimes U'_{j,\nu}
\in (U(\mathfrak k_\bbC)\otimes U(\mathfrak k_\bbC))^M,
\]
where $dm$ is the normalized Haar measure on $M$.
Then we can rewrite \eqref{eq:D_expre} in the following way:
\[
D=D' + \sum_{j=1}^k
\eta_H \biggl(
f_j(H)\, E'_j \otimes
\sum_{\nu=1}^{k_j}
U_{j,\nu} \otimes U'_{j,\nu}
\biggr)
=
D' + \sum_{j=1}^k
f_j(H) \sum_{\nu=1}^{k_j}
 (\Ad(e^{-H})(U_{j,\nu}))\,E'_j\,U'_{j,\nu}.
\]
Now suppose $\phi$ is any $\pi$-spherical function.
Since $U\,\trans \phi=\trans \phi\,\pi^*(U)$
for any $U\in U(\mathfrak k_\bbC)$,
we have $D'\,\trans\phi=\gamma^\pi(D)(\cdot-\rho)\,\trans\phi$
and hence $D'\,\phi=\gamma^\pi(D)(\cdot-\rho)\,\phi$,
where $\gamma^\pi(D)(\cdot-\rho) \in S(\mathfrak a_\bbC)$ acts on
$\trans\phi$ and $\phi$
as a differential operator.
On the other hand,
we have
\begin{align*}
\sum_{\nu=1}^{k_j}
 (\Ad(e^{-H})(U_{j,\nu}))\,E'_j\,U'_{j,\nu}\,\trans\phi(e^H)
&=
\sum_{\nu=1}^{k_j}
\pi^*(U_{j,\nu}) \,
\bigl(E'_j\,\trans\phi(e^H)\bigr)
\,\pi^*(U'_{j,\nu})\\
&=\pi^*\biggl(
\sum_{\nu=1}^{k_j}U_{j,\nu}U'_{j,\nu}
\biggr)
E'_j\,\trans\phi(e^H).
\end{align*}
Here 
the second equality holds
since $\trans\phi(e^H)$ and $E'_j\,\trans\phi(e^H)$
are scalar operators.
Note that
$\pi^*\biggl(
\sum_{\nu=1}^{k_j}U_{j,\nu}U'_{j,\nu}
\biggr) \in \End_M V^*$ is also a scalar operator.
Let $E_j\in S(\mathfrak a_\bbC)$ be the product of this scalar value and $E'_j$
($j=1,\ldots, k$).
Then the operator $\varDelta^\pi(D)$ defined by \eqref{eq:GeneralRad}
satisfies \eqref{RadFormula}.

Finally, 
since for each $w\in W$ any compactly supported $C^\infty$ function on $\{e^H\,|\,H\in w\mathfrak a_-\}$
is (uniquely) extended to a $\pi$-spherical function on $G$ by Theorem \ref{thm:Chevalley}, we get the uniqueness of $\varDelta^\pi(D)$.
The $W$-invariance of $\varDelta^\pi(D)$ easily follows from this.
\end{proof}
\begin{rem}
Since the action of $D\in U(\mathfrak g_\bbC)^K$ on $\pi$-spherical functions factors through
$U(\mathfrak g_\bbC)^K \to \boldsymbol D^\pi$,
$\{\varDelta^\pi(D)\,|\,D\in U(\mathfrak g_\bbC)^K\}$
is a commutative subalgebra of $(\mathscr R\otimes S(\mathfrak a_\bbC))^W$.
We denote this subalgebra by $\varDelta^\pi(\boldsymbol D^\pi)$.
\end{rem}
Theorems \ref{thm:unique-sph}, \ref{thm:Chevalley}, and Proposition \ref{prop:GeneralRad} imply
\begin{cor}\label{cor:spherical}
Suppose $\lambda\in \mathfrak a_\bbC^*$.
The subspace of $C^\infty(\mathfrak a)^W$ consisting of those $f$ satisfying
\[
\varDelta^\pi(D)\,f=\gamma^\pi(D)(\lambda)\,f
\quad\text{for any }D\in U(\mathfrak a_\bbC)^K
\]
equals $\bbC\,\varUpsilon^\pi(\phi^\pi_\lambda)$
and is a subspace of $\mathscr A(\mathfrak a)$ (the space of real analytic functions on $\mathfrak a$).
\end{cor}

\section{Heckman-Opdam hypergeometric functions}\label{sec:HO}
In this section $\mathfrak a$ denotes any finite-dimensional linear space with inner product $B(\cdot,\cdot)$.
Let $(\cdot,\cdot)$ be the symmetric bilinear form on $\mathfrak a_\bbC^*$
induced from $B(\cdot,\cdot)$.
Let $\varSigma'$ be a (possibly non-reduced) crystallographic root system in $\mathfrak a^*$.
Its Weyl group is denoted by $W'$.
In a series of papers starting from \cite{HO}, Heckman and Opdam define and study the hypergeometric 
function $F(\varSigma',\bsk,\lambda)\in\mathscr A(\mathfrak a)$
associated to $\varSigma'$.
Here $\bsk$ is a $\bbC$-valued multiplicity function on $\varSigma'$ with some regularity condition
and $\lambda\in\mathfrak a_\bbC^*$.
The hypergeometric function $F(\varSigma',\bsk,\lambda)$
is a natural generalization of $\varUpsilon^\pi(\phi^\pi_\lambda)$
with the trivial $K$-type $\pi$,
allowing the root multiplicities $\bsm$ to be generic complex numbers  $\bsk$. 
In this section
we review the definition and some fundamental properties of $F(\varSigma',\bsk,\lambda)$.
We refer the reader to \cite{Op:book, He:white} for details. 

\subsection{Hypergeometric differential operators}\label{subsec:HGD}
Let $\mathcal K(\varSigma')$ be the space of multiplicity functions on $\varSigma'$.
This is a linear space with dimension equal to the number of the $W'$-orbits in $\varSigma'$.
Regarding $\mathfrak a$ as an Abelian Lie algebra equipped with an inner product,
we define $\varOmega_{\mathfrak a}$ and $H_\alpha$ ($\alpha\in\varSigma'$) as in \S\ref{sec:spherical}.
Fix a positive system $\varSigma'^+\subset\varSigma'$.
Let $\mathscr R'$ denote the unital algebra generated by $(1- e^{\alpha})^{-1}$ ($\alpha\in\varSigma'^+$).
Let $\mathscr M'$ be the maximal ideal of $\mathscr R'$
consisting of those $f$ that are expanded in the form
$f=\sum_{\mu\in \bbN\varSigma^+\setminus\{0\}} a_\mu e^\mu$
on $\mathfrak a_-:=\{H\in\mathfrak a\,|\,\alpha(H)<0\text{ for any }\alpha\in\varSigma'^+\}$.
As in \S\ref{subsec:RP}, $\mathscr R' \otimes S(\mathfrak a_\bbC)$
denotes the algebra of differential operators with coefficients in $\mathscr R'$.

Suppose $\bsk\in \mathcal K(\varSigma')$.
Put
\begin{equation}\label{eq:HOL1}
L(\varSigma',\bsk)=\varOmega_{\mathfrak a} + \sum_{\alpha\in\varSigma'^+} \bsk_\alpha\coth\frac{\alpha}2\, H_\alpha,
\end{equation}
which belongs to $(\mathscr R' \otimes S(\mathfrak a_\bbC))^{W'}$.
\begin{rem}
In the setting of \S\ref{sec:spherical}, let $\pi$ be the trivial $K$-type.
Then the radial part of the Casimir operator given by \eqref{eq:CasimirRad}
equals  $L(2\varSigma,\bsk)$ with $\bsk_{2\alpha}=\bsm_\alpha$ ($\forall\alpha\in\varSigma$).
\end{rem}
Note 
$D\mapsto \tilde\delta(\varSigma',\bsk)^{\frac12}\circ D\circ \tilde\delta(\varSigma',\bsk)^{-\frac12}$
with $\tilde\delta(\varSigma',\bsk)$ in \eqref{eq:HOdelta}
defines an algebra automorphism of $(\mathscr R' \otimes S(\mathfrak a_\bbC))^{W'}$.
\begin{prop}[{\cite[Proposition 2.2]{HO}}, {\cite[Theorem 2.1.1]{He:white}}]\label{prop:Ltwist}
Putting
$\rho(\bsk)=
\linebreak[4] %%%%%%%%%%%%%%%%%%%%%%%%%%%%%%%%%%%%%%%%%%%%%%%%
\frac12\sum_{\alpha\in\varSigma'^+}\bsk_\alpha \alpha$,
we have
\begin{equation}\label{eq:Ltwist}
\tilde\delta(\varSigma',\bsk)^{\frac12}\circ(L(\varSigma',\bsk)+(\rho(\bsk),\rho(\bsk)))\circ \tilde\delta(\varSigma',\bsk)^{-\frac12}
=
\varOmega_{\mathfrak a} + 
\sum_{\alpha\in\varSigma'^+}\frac{\bsk_\alpha(1-\bsk_\alpha-2\bsk_{2\alpha})||\alpha||^2}{4\sinh^2\frac\alpha2}.
\end{equation}
Here $\bsk_{2\alpha}=0$ if $2\alpha \notin \varSigma'^+$.
\end{prop}

Since $\mathscr M' \otimes S(\mathfrak a_\bbC)$ is an ideal of $\mathscr R'\otimes S(\mathfrak a_\bbC)$,
the linear map
\begin{equation}\label{eq:HOHC}
\gamma_{\rho(\bsk)} : \mathscr R'\otimes S(\mathfrak a_\bbC)=
S(\mathfrak a_\bbC) \oplus \mathscr M' \otimes S(\mathfrak a_\bbC)
\xrightarrow{\text{projection}}
S(\mathfrak a_\bbC) \xrightarrow{f(\lambda) \mapsto f(\lambda+\rho(\bsk))}
S(\mathfrak a_\bbC)
\end{equation}
is an algebra homomorphism.
\begin{prop}[{\cite[Theorem 3.11]{He:Heckman}}, {\cite[Theorem 1.3.12]{He:white}}]\label{prop:HCisomHO}
The restriction of $\gamma_{\rho(\bsk)}$
to the subalgebra
\[
(\mathscr R'\otimes S(\mathfrak a_\bbC))^{W',L(\varSigma',\bsk)}
:=\bigl\{D \in (\mathscr R'\otimes S(\mathfrak a_\bbC))^{W'} \,\bigm|\,
[L(\varSigma',\bsk), D]=0
\bigr\}
\]
is an algebra isomorphism onto $S(\mathfrak a_\bbC)^{W'}$.
In particular $(\mathscr R'\otimes S(\mathfrak a_\bbC))^{W',L(\varSigma',\bsk)}$
is a commuting family of differential operators.
\end{prop}
One easily sees $\mathscr R'$, $L(\varSigma',\bsk)$, $(\mathscr R'\otimes S(\mathfrak a_\bbC))^{W',L(\varSigma',\bsk)}$
and the restriction of $\gamma_{\rho(\bsk)}$ to $(\mathscr R'\otimes S(\mathfrak a_\bbC))^{W',L(\varSigma',\bsk)}$
do not depend on the choice of $\varSigma'^+$.
Now suppose $\lambda\in\mathfrak a_\bbC^*$
and let us consider
the following three conditions for $f\in \mathscr A(\mathfrak a)$:
\begin{enumerate}[label=(HG\arabic*), leftmargin=*]
\item $f\in \mathscr A(\mathfrak a)^{W'}$;\label{HG1}
\item $D\,f=\gamma_{\rho(\bsk)}(D)(\lambda)\,f$ for any
$D\in (\mathscr R'\otimes S(\mathfrak a_\bbC))^{W',L(\varSigma',\bsk)}$;\label{HG2}
\item $f(0)=1$.\label{HG3}
\end{enumerate}
As we see below
these conditions characterize $F(\varSigma',\bsk,\lambda)$.

\subsection{Definition of the hypergeometric functions}
Put
\begin{align}
\tilde{c}(\varSigma',\bsk,\lambda)&=\prod_{\alpha\in\varSigma'_+} 
\frac{\varGamma(\lambda(\alpha^\vee)+\frac12\bsk_{\frac12\alpha})}
{\varGamma(\lambda(\alpha^\vee)+\frac12 \bsk_{\frac12\alpha}+\bsk_\alpha)}
, \\
c(\varSigma',\bsk,\lambda)&=
\frac{\tilde{c}(\varSigma',\bsk,\lambda)}{\tilde{c}(\varSigma',\bsk,\rho(\bsk))},
\label{eq:cfho}
\end{align}
where $\bsk_{\frac12\alpha}=0$ if $\frac12\alpha\notin\varSigma'^+$.

For any $\lambda\in\mathfrak a_\bbC^*$ with the property
\begin{equation*}
2(\lambda,\beta)+(\beta,\beta)\ne 0\quad\text{for any }\beta\in\bbN\varSigma'^+\setminus\{0\},
\end{equation*}
there is a unique formal series in the form
\[
\varPhi(\varSigma',\bsk,\lambda)=e^{\lambda+\rho(\bsk)}
+\sum_{\mu\in\bbN\varSigma'^+\setminus\{0\}} a_\mu e^{\lambda+\rho(\bsk)+\mu}\quad(a_\mu\in \bbC)
\]
such that
\[
L(\varSigma',\bsk)\,\varPhi(\varSigma',\bsk,\lambda)
=\gamma_{\rho(\bsk)}(L(\varSigma',\bsk))(\lambda)\,\varPhi(\varSigma',\bsk,\lambda).
\]
The series actually converges absolutely on $\mathfrak a_-$.
Thus for a generic $\lambda$
\begin{equation}\label{eq:connection}
\tilde F(\varSigma',\bsk,\lambda)=\sum_{w\in W}\tilde c(\varSigma',\bsk,-w\lambda)\varPhi(\varSigma',\bsk,w\lambda).
\end{equation}
is a well-defined real analytic function on $\mathfrak a_-$.
We have immediately from \eqref{eq:connection}
that $\tilde F(\varSigma',\bsk,w\lambda)=\tilde F(\varSigma',\bsk,\lambda)$ for $w\in W'$
and that 
for $H\in\mathfrak a_-$ and a generic $\lambda$ satisfying $(\operatorname{Re}\lambda,\alpha)<0$ ($\forall\alpha\in\varSigma'$)
\begin{equation}\label{eq:limcf2}
 \lim_{t\to\infty}e^{t(-\lambda-\rho(\bsk))(H)}\tilde F(\varSigma',\bsk,\lambda;{tH})=
\tilde c(\varSigma',\bsk,-\lambda). 
 \end{equation}

\begin{thm}[{\cite[Theorem 2.8]{HO4}, \cite[\S4]{He:white}}]
There exists a $W$-invariant open neighborhood $\mathcal U$ of $0\in\mathfrak a$
such that $\tilde F(\varSigma',\bsk,\lambda;H)$ extends to a holomorphic function on
$\mathcal K(\varSigma')\times \mathfrak a_\bbC^* \times (\mathfrak a + \sqrt{-1}\mathcal U)$.
Moreover $f:=\tilde F(\varSigma',\bsk,\lambda; H)|_{H\in\mathfrak a}$ satisfies \ref{HG1} and \ref{HG2}
for each $(\bsk,\lambda)\in\mathcal K(\varSigma')\times\mathfrak a_\bbC^*$.
\end{thm}

\begin{thm}[{\cite[Theorem 6.1]{Op:GSF}}]\label{thm:GSF}
$\tilde F(\varSigma',\bsk,\lambda;0)=\tilde{c}(\varSigma',\bsk,\rho(\bsk))$ for any $(\bsk,\lambda)\in\mathcal K(\varSigma')\times\mathfrak a_\bbC^*$.
\end{thm}

In particular $\tilde{c}(\varSigma',\bsk,\rho(\bsk))$ is an entire holomorphic function on $\mathcal K(\varSigma')$.
\begin{defn}
Put
\[
\mathcal K_\reg(\varSigma')=\{\bsk \in \mathcal K(\varSigma')\,|\, \tilde{c}(\varSigma',\bsk,\rho(\bsk))\ne0\}.
\]
For each $(\bsk,\lambda)\in\mathcal K_\reg(\varSigma')\times\mathfrak a_\bbC^*$
we define 
\[
F(\varSigma',\bsk,\lambda)=\tilde{c}(\varSigma',\bsk,\rho(\bsk))^{-1} \tilde F(\varSigma',\bsk,\lambda).
\]
This is a real analytic function on $\mathfrak a$ satisfying \ref{HG1}--\ref{HG3}.
\end{defn}

\subsection{Regularity conditions on $\bsk$}\label{subsec:RC}
Suppose $\bsk\in\mathcal K(\varSigma')$.
For $H\in \mathfrak a$ a so-called \emph{Cherednik operator} is defined by
\[
T_{\bsk}(H)=\partial(H)+\sum_{\alpha\in\varSigma'^+}\frac{\bsk_\alpha\alpha(H)}{1-e^{-\alpha}}(1-r_\alpha)-\rho(\bsk)(H)
\]
where $\partial(H)$ is the $H$-directional derivative and $r_\alpha$ is the orthogonal reflection in $\alpha=0$.
The operator $T_{\bsk}(H)$ acts on various function spaces including $\mathscr A(\mathfrak a)$ and the algebra $\Hat{\mathscr A}_0$ of formal power series at $0\in\mathfrak a$.
Note $\mathscr A(\mathfrak a)\subset \Hat{\mathscr A}_0$.
By \cite{Ch1}, $T_{\bsk}(H)$'s ($H\in\mathfrak a$) are commutative
and $T_{\bsk}(\cdot)$ extends to an algebra homomorphism $S(\mathfrak a_\bbC)\to \End_\bbC \Hat{\mathscr A}_0$  (see also \cite[Theorem 2.1]{Ch2}).
Suppose $\lambda\in\mathfrak a_\bbC^*$.
In view of \cite[Theorem 2.12]{Op:Cherednik}, $f\in \mathscr A(\mathfrak a)$ satisfies \ref{HG1}--\ref{HG3} if and only if $f$ satisfies the following:
\begin{enumerate}[label=(HG\arabic*\/$'$), leftmargin=*]
\item $f\in \Hat{\mathscr A}_0^{W'}$,\label{HG1'}
\item $T_{\bsk}(D)\,f=D(\lambda)\,f$ for any
$D\in S(\mathfrak a_\bbC)^{W'}$.\label{HG2'}
\item $f(0)=1$.\label{HG3'}
\end{enumerate}
These conditions can be applied to any $f\in\Hat{\mathscr A}_0$.
Define a bilinear form $\langle\cdot,\cdot\rangle_{\bsk} : S(\mathfrak a_\bbC)\times \Hat{\mathscr A}_0 \to \bbC$ by $\langle D, f\rangle_{\bsk}=(T_{\bsk}(D)f)(0)$ and put
\begin{align*}
\lRad_{\bsk}&=\{D\in S(\mathfrak a_\bbC)\,|\,\langle D,f \rangle_{\bsk}=0\quad(\forall f\in \Hat{\mathscr A}_0)\},\\
\rRad_{\bsk}&=\{f\in \Hat{\mathscr A}_0\,|\,\langle D,f \rangle_{\bsk}=0\quad(\forall D\in S(\mathfrak a_\bbC))\}.
\end{align*}

The \emph{Dunkl operator}
for $H\in \mathfrak a$ with multiplicity parameter $\bsk$  (\cite{Du}) is defined by
\[
\bar T_{\bsk}(H)=\partial(H)+\sum_{\alpha\in\varSigma'^{+}}\frac{\bsk_\alpha\alpha(H)}{\alpha}(1-r_\alpha).
\]
The space $\mathscr P(\mathfrak a)$ of polynomial functions on $\mathfrak a$
is identified with $S(\mathfrak a_\bbC)$ via $B(\cdot,\cdot)$,
so that $\bar T_{\bsk}(H)$'s act on $S(\mathfrak a_\bbC)$.
This action also extends to an algebra homomorphism $\bar T_{\bsk}(\cdot): S(\mathfrak a_\bbC)\to\End_\bbC S(\mathfrak a_\bbC)$.
Define a bilinear form $(\cdot,\cdot)_{\bsk}$ on $S(\mathfrak a_\bbC)\times S(\mathfrak a_\bbC)$ by $( D_1, D_2)_{\bsk}=(\bar T_{\bsk}(D_1)D_2)(0)$.
By \cite[Theorem 3.5]{Du2}, $(\cdot,\cdot)_{\bsk} $ is symmetric.

\begin{thm}\label{thm:existence}
The following conditions on $\bsk \in \mathcal K(\varSigma')$ are all equivalent:
\begin{enumerate}[label=(\arabic*), leftmargin=*, topsep=0pt]
\item $\bsk \in \mathcal K_\reg(\varSigma')$; \label{k1}
\item $(\cdot,\cdot)_{\bsk}$ is non-degenerate; \label{k2}
\item $\lRad_{\bsk}=\{0\}$; \label{k3}
\item $\rRad_{\bsk}=\{0\}$;\label{k4}
\item for any $\lambda\in\mathfrak a_\bbC^*$ it holds that
any $f\in \Hat{\mathscr A}_0\setminus\{0\}$ satisfying \ref{HG1'} and \ref{HG2'}
takes a non-zero value at $0$;\label{k5}
\item for any $\lambda\in\mathfrak a_\bbC^*$ there exists some $f\in \mathscr A(\mathfrak a)$ satisfying \ref{HG1}--\ref{HG3};\label{k6}
\item there exists a Zariski dense subset $Z\subset \mathfrak a_\bbC^*$ such that
for any $\lambda\in Z$ there exists some $f\in \Hat{\mathscr A}_0$ satisfying \ref{HG1'}--\ref{HG3'}.\label{k7}
\end{enumerate}
\end{thm}
\begin{cor}\label{cor:existunique}
For any $\bsk\in \mathcal K_\reg(\varSigma')$ and $\lambda\in\mathfrak a_\bbC^*$,
$F(\varSigma',\bsk,\lambda)$ is
the unique real analytic function on $\mathfrak a$ satisfying \ref{HG1}--\ref{HG3} (or equivalently \ref{HG1'}--\ref{HG3'}).
In particular the definition of $F(\varSigma',\bsk,\lambda)$ is independent of the choice of $\varSigma'^+$.
\end{cor}
\begin{proof}
The uniqueness follows from \ref{k5}.
\end{proof}
\begin{cor}\label{cor:-HO}
$F(\varSigma',\bsk,\lambda;-H)=F(\varSigma',\bsk,-\lambda;H)$.
\end{cor}
\begin{proof}
Let $T_{\bsk}^-(\xi)$ ($\xi\in\mathfrak a$) stand for the Cherednik operator
defined by using $-\varSigma'$ as a positive system.
For $f\in\mathscr A(\mathfrak a)$ define $\sigma f\in\mathscr A(\mathfrak a)$ by $(\sigma f)(H)=f(-H)$.
Observe that $T^-_{\bsk}(\xi)\sigma f=\sigma T_{\bsk}(-\xi) f$ for any $\xi\in\mathfrak a$.
Hence for any $f$ satisfying \ref{HG2'} we have
\[
T^-_{\bsk}(D)\sigma f=\sigma T_{\bsk}(D(-\cdot)) f =D(-\lambda)\sigma f
\quad\text{for any }
D\in S(\mathfrak a_\bbC)^{W'}.
\]
Thus the desired equality follows from Corollary \ref{cor:existunique}.
\end{proof}
\begin{cor}\label{cor:HOasymp}
Suppose $\bsk\in\mathcal K_\reg(\varSigma')$.
Then for $H\in\mathfrak a_+:=-\mathfrak a_-$ and
a generic $\lambda$ satisfying $(\operatorname{Re}\lambda,\alpha)>0$ ($\forall\alpha\in\varSigma'$) we have
\begin{equation*}
 \lim_{t\to\infty}e^{t(-\lambda+\rho(\bsk))(H)} F(\varSigma',\bsk,\lambda;{tH})=
c(\varSigma',\bsk,\lambda). 
 \end{equation*}
\end{cor}
\begin{proof}
Immediate from \eqref{eq:limcf2} and the previous corollary.
\end{proof}
\begin{rem}
The prototype of the theorem is
a similar result for Opdam's \emph{generalized Bessel function}s obtained by
Dunkl,  de Jeu and Opdam \cite[\S4]{DJO}.
They also give for each type of irreducible $\varSigma'$
a very explicit description
of the singular set of those $\bsk$ which do not satisfy \ref{k2}.
By the theorem
this singular set equals $\mathcal K(\varSigma')\setminus\mathcal K_\reg(\varSigma')$.
\end{rem}

The rest of this subsection is devoted to the proof of Theorem \ref{thm:existence}.
For $d=0,1,2,\ldots$ let $S_d(\mathfrak a_\bbC)$ be the subspace of  $S(\mathfrak a_\bbC)\,(\simeq\mathscr P(\mathfrak a))$ consisting of homogeneous polynomials of degree $d$.
Since $\bar T_{\bsk}(H)$ is a homogeneous operator of degree $-1$ for any $H\in\mathfrak a\setminus\{0\}$, $S_d(\mathfrak a_\bbC)\perp S_e(\mathfrak a_\bbC)$ with respect to
$(\cdot,\cdot)_{\bsk}$ if $d\ne e$.
Now any $f\in \Hat{\mathscr A}_0$ decomposes into the sum $f=\sum_{d\ge0}f_d$ of its homogeneous parts $f_d\in S_d(\mathfrak a_\bbC)$.
We define $\ord f=\min\{d\,|\, f_d\ne 0\}\in \bbN\cup\{\infty\}$ and put $\Hat{\mathscr A}_{0,> d}=\{f \in \Hat{\mathscr A}_0\,|\,\ord f> d\}$.
Then we have
\[
\Hat{\mathscr A}_{0} / \Hat{\mathscr A}_{0,> d} \simeq 
S_{\le d}(\mathfrak a_\bbC) : =\bigoplus_{e\le d} S_e(\mathfrak a_\bbC).
\]
The next lemma is easily observed:
\begin{lem}\label{lem:DunklCherednik}
Suppose $f \in \Hat{\mathscr A}_0\setminus\{0\}$ with $d=\ord f$.
Then for any $H\in\mathfrak a$, $\ord (T_{\bsk}(H)f)  \ge d-1$ and
\[
T_{\bsk}(H)f\equiv \bar T_{\bsk}(H) f_{d} \pmod{\Hat{\mathscr A}_{0,> d-1}}.
\]
\end{lem}
For $d=0,1,2,\ldots$ let $\langle\cdot,\cdot\rangle_{\bsk}^d$ be the restriction of
$\langle\cdot,\cdot\rangle_{\bsk}$ to $S_{\le d}(\mathfrak a_\bbC) \times S_{\le d}(\mathfrak a_\bbC)$.
Since $\dim S_{\le d}(\mathfrak a_\bbC)<\infty$,
the left and right radicals of $\langle\cdot,\cdot\rangle_{\bsk}^d$
have the same dimension.
Let us consider an auxiliary condition:
\begin{enumerate}[label=(3)$^\prime$]
\item
$\langle\cdot,\cdot\rangle_{\bsk}^d$ is non-degenerate
for each $d=0,1,2,\dotsc$.\label{k3'}
\end{enumerate}

\begin{proof}[Proof of \ref{k2}$\Leftrightarrow$\ref{k3}$\Leftrightarrow$\ref{k3'}$\Rightarrow$\ref{k4}]
Suppose \ref{k2}. For any $D\in S(\mathfrak a_\bbC)\setminus\{0\}$ with $\deg D=d$,
we can take $f\in S_{d}(\mathfrak a_\bbC)$ so that $(D,f)_{\bsk}\ne 0$.
Then from the lemma above we have
$\langle D,f \rangle_{\bsk}=(T_{\bsk}(D)f)(0)=(\bar T_{\bsk}(D)f)(0)= (D,f)_{\bsk}\ne 0$, proving \ref{k3}.

Next, one easily sees \ref{k3}$\Rightarrow$\ref{k3'}$\Rightarrow$\ref{k4} since $S_{\le d}(\mathfrak a_\bbC)\perp \Hat{\mathscr A}_{0,> d}$
with respect to $\langle\cdot,\cdot\rangle_{\bsk}$ by the lemma.

To deduce \ref{k2} from \ref{k3'}, take an arbitrary $f \in S_d(\mathfrak a_\bbC)\setminus\{0\}$ with $d=0,1,2,\dotsc$.
Then \ref{k3'} assures the existence of $D\in S_{\le d}(\mathfrak a_\bbC)$ such that $\langle D,f \rangle_{\bsk}\ne0$.
Since $\ord f=d$,
the lemma again implies $(D,f)_{\bsk}=\langle D,f \rangle_{\bsk}\ne0$.
\end{proof}

The implication \ref{k1}$\Rightarrow$\ref{k6}$\Rightarrow$\ref{k7} is obvious.

\begin{lem}
For any $\bsk_0\in\mathcal K(\varSigma')$, $\tilde F(\varSigma',\bsk_0,\lambda; H)$
is a non-trivial function in $(\lambda,H)\in\mathfrak a_\bbC^*\times \mathfrak a$. 
\end{lem}
\begin{proof}
Because there is a non-empty open subset $U_{\bsk_0}\subset \mathfrak a_\bbC^*$ such that \eqref{eq:limcf2}
holds for any $(\lambda,H)\in U_{\bsk_0}\times \mathfrak a_-$ and  $\bsk=\bsk_0$.
\end{proof}

Hence by Theorem \ref{thm:GSF} we have \ref{k5}$\Rightarrow$\ref{k1}.

\begin{proof}[Proof of \ref{k7}$\Rightarrow$\ref{k3}]
Suppose $\lRad_{\bsk}\ni D\ne 0$.
Since $\lRad_{\bsk}$ is an ideal of $S(\mathfrak a_\bbC)$,
$D':=\prod_{w\in W'} w(D)$ is a non-zero $W'$-invariant element in  $\lRad_{\bsk}$.
Now suppose \ref{k7} is true.
Then we can find $\lambda\in Z$
such that $D'(\lambda)\ne0$.
But for a function $f\in \Hat{\mathscr A}_0$ of \ref{HG1'}--\ref{HG3'}
we have $\langle D',f\rangle_{\bsk}=(T_{\bsk}(D')f)(0)=D'(\lambda)f(0)=D'(\lambda)\ne 0$.
This contradicts that $D'\in\lRad_{\bsk}$.
\end{proof}

The proof is complete if we show \ref{k4}$\Rightarrow$\ref{k5}.
To do so, we needs the \emph{graded Hecke algebra} $\boldsymbol H=\boldsymbol H(\varSigma'^+,\bsk)$ by \cite{Lu}.
The algebra $\boldsymbol H$ is isomorphic to $\bbC W' \otimes S(\mathfrak a_\bbC)$ as a $\bbC$-linear space. Here the group algebra $\bbC W'$ of $W'$ and $S(\mathfrak a_\bbC)$ are identified with subalgebras of $\boldsymbol H$ by $w\mapsto w\otimes 1$ and $D\mapsto 1\otimes D$.
These two subalgebras relate to each other as follows:
\begin{align*}
w \cdot  D &= w\otimes D &&\text{for any }w\in W'\text{ and }D\in S(\mathfrak a_\bbC);\\
H \cdot  r_\alpha&= r_\alpha \cdot r_\alpha(H) - (\bsk_\alpha+2\bsk_{2\alpha})\alpha(H)
&&\text{for any simple root }\alpha\in\varSigma'^+\text{ and any }H\in \mathfrak a_\bbC.
\end{align*}
The center of $\boldsymbol H$ equals $S(\mathfrak a_\bbC)^{W'}$
(\cite[Theorem 6.5]{Lu}).

Thanks to \cite[Theorem 2.4]{Ch2}, $\boldsymbol H\ni w \otimes D\mapsto w\cdot T_{\bsk}(D)\in \End_\bbC\Hat{\mathscr A}_0$ defines an algebra homomorphism (also denoted by $T_{\bsk}(\cdot)$).
Let $\bbC v_0$ be a one-dimensional right trivial $W'$-module.
Then $\bbC v_0 \otimes_{\bbC W'} \boldsymbol H$ is a right $\boldsymbol H$-module.
Identifying $S(\mathfrak a_\bbC)$ with $\bbC v_0 \otimes_{\bbC W'} \boldsymbol H$
by $D\mapsto v_0 \otimes D$,
we consider $\langle\cdot,\cdot\rangle_{\bsk}$ is a bilinear form
on $\bbC v_0 \otimes_{\bbC W'} \boldsymbol H \times \Hat{\mathscr A}_0$.
Observe that $\langle\cdot h,\cdot\rangle_{\bsk}=\langle\cdot,h\cdot\rangle_{\bsk}$ for any $h\in \boldsymbol H$.
For $d=0,1,2,\dotsc$, $v_0\otimes S_{\le d}(\mathfrak a_\bbC)$ is a $W'$-submodule of
$\bbC v_0 \otimes_{\bbC W'} \boldsymbol H$
and 
\[
v_0\otimes S_{\le d}(\mathfrak a_\bbC)\,\bigm/\, v_0\otimes S_{\le d-1}(\mathfrak a_\bbC)\,\simeq\, S_d(\mathfrak a_\bbC) \text{ with natural right $W'$-module structure.}
\]
From this one easily sees
\begin{lem}\label{lem:CHW'}
The subspace of\/ $\bbC v_0 \otimes_{\bbC W'} \boldsymbol H$ consisting of
the right $W'$-invariants
is $v_0\otimes S(\mathfrak a_\bbC)^{W'}$.
\end{lem}

\begin{proof}[Proof of \ref{k4}$\Rightarrow$\ref{k5}]
Let $\lambda\in\mathfrak a_\bbC^*$ and
suppose $f\in \Hat{\mathscr A}_0$ satisfies \ref{HG1'}, \ref{HG2'} and $f(0)=0$.
Take an arbitrary $D\in S(\mathfrak a_\bbC)$ .
Then by the last lemma 
there exists $\tilde D\in S(\mathfrak a_\bbC)^{W'}$ such that
$v_0\otimes \tilde D=\frac1{\#W'}\sum_{w\in W'} v_0\otimes D\cdot w$.
Hence we have
\begin{align*}
\langle D,f\rangle_{\bsk}
&=\langle v_0\otimes D,f\rangle_{\bsk}
=\frac1{\#W'}\sum_{w\in W'} \langle v_0\otimes D,w\cdot f\rangle_{\bsk}
=\frac1{\#W'}\sum_{w\in W'} \langle v_0\otimes D\cdot w, f\rangle_{\bsk}\\
&=\langle v_0\otimes \tilde D,f\rangle_{\bsk}
=(T_{\bsk}(\tilde D)f)(0)
=\tilde D(\lambda) f(0)=0.
\end{align*}
This shows $f \in \rRad_{\bsk}$.
\end{proof}

\section{Matching conditions}\label{sec:MC}
We return to the setting of \S\ref{sec:spherical}.
Thus $(\pi,V)$ is a small $K$-type of a connected non-compact real semisimple Lie group $G$ with finite center.
The purpose of this section is to 
get an easy and concrete condition 
on a root system $\varSigma^\pi$ and a multiplicity function $\bsk^\pi$
for the validity of \eqref{eq:main}.

\subsection{Coincidence of differential operators}
Let $\varSigma'$ is a root system in $\mathfrak a^*$ ($\mathfrak a=\Lie A$)
and $\bsk$ a multiplicity function on $\varSigma'$.
In addition, we suppose $\varSigma' \subset \varSigma \cup 2\varSigma$
and the Weyl group $W'$ for $\varSigma'$ equals $W$.
Let $\varSigma'^+:=\varSigma'\cap(\varSigma^+ \cup 2\varSigma^+)$.
Then the notation in \S\ref{sec:spherical}
and that in \S\ref{sec:HO} are fully compatible.
Note that the algebra $\mathscr R'$, which is generated by $(1- e^\alpha)^{-1}$ ($\alpha\in \varSigma'^+$),
is a subalgebra of $\mathscr R$, which is generated by $(1\pm e^\alpha)^{-1}$ ($\alpha\in \varSigma^+$),
and that $\mathscr M'=\mathscr M\cap\mathscr R'$.

\begin{prop}\label{prop:Omegatwist}
With $\tilde\delta_{G/K}$ in \eqref{eq:GKdelta} it holds that
\begin{multline}\label{eq:Omegatwist}
\tilde\delta_{G/K}^{\frac12}\circ(\varDelta^\pi(\varOmega_{\mathfrak g}-\varpi^\pi)+||\rho||^2)\circ \tilde\delta_{G/K}^{-\frac12}\\
=
\varOmega_{\mathfrak a} + 
\sum_{\alpha\in\varSigma^+}
\frac{\bsm_\alpha ||\alpha||^2}4
\biggl(
\frac{-\bska_\alpha^\pi}{\sinh^2\frac\alpha2}
+
\frac{2-\bsm_\alpha -2\bsm_{2\alpha}+4\bska_\alpha^\pi}{\sinh^2\alpha}
\biggr).
\end{multline}
\end{prop}
\begin{proof}
Letting $\varSigma'=2\varSigma$ and $\bsk_{2\alpha}=\frac12\bsm_{\alpha}$
($\alpha\in \varSigma$), we have
\[
L(\varSigma',\bsk)=\varOmega_{\mathfrak \alpha}
+
\sum_{\alpha\in\varSigma^+} \bsm_\alpha \coth\alpha\,H_\alpha,\ 
\tilde\delta(\varSigma',\bsk)=\tilde\delta_{G/K}
\ \text{ and }\ 
\rho(\bsk)=\rho.
\]
Hence it follows from \eqref{eq:CasimirRad} and the Proposition \ref{prop:Ltwist} that
\[
\tilde\delta_{G/K}^{\frac12}\circ(\varDelta^\pi(\varOmega_{\mathfrak g}-\varpi^\pi)+||\rho||^2)\circ \tilde\delta_{G/K}^{-\frac12}
=
\varOmega_{\mathfrak a} + 
\sum_{\alpha\in\varSigma^+}
\frac{\bsm_\alpha ||\alpha||^2}4
\biggl(
\frac{2-\bsm_\alpha-2\bsm_{2\alpha}}{\sinh^2 \alpha}
-
\frac{\bska^\pi_\alpha}{\cosh^2\frac{\alpha}2}
\biggr).
\]
Using the equality
\begin{equation*}
\frac1{\cosh^2\frac{\alpha}2}
=
\frac1{\sinh^2\frac{\alpha}2}
-
\frac4{\sinh^2\alpha},
\end{equation*}
we get the proposition.
\end{proof}

Let us consider general $\varSigma'$ and $\bsk$ again.
The algebra homomorphism $\gamma_{\rho(\bsk)} :\mathscr R'\otimes S(\mathfrak a_\bbC) \to S(\mathfrak a_\bbC)$ defined by \eqref{eq:HOHC}
can be regarded as a part of the algebra homomorphism
\[
\gamma_{\rho(\bsk)} : \mathscr R\otimes S(\mathfrak a_\bbC)=
S(\mathfrak a_\bbC) \oplus \mathscr M \otimes S(\mathfrak a_\bbC)
\xrightarrow{\text{projection}}
S(\mathfrak a_\bbC) \xrightarrow{f(\lambda) \mapsto f(\lambda+\rho(\bsk))}
S(\mathfrak a_\bbC).
\]
\begin{lem}\label{lem:HCisomHO}
The subalgebra
\[(\mathscr R\otimes S(\mathfrak a_\bbC))^{W,L(\varSigma',\bsk)}:=\bigl\{D \in (\mathscr R\otimes S(\mathfrak a_\bbC))^{W} \,\bigm|\,
[L(\varSigma',\bsk), D]=0
\bigr\}
\]
coincides with
$(\mathscr R'\otimes S(\mathfrak a_\bbC))^{W,L(\varSigma',\bsk)}
$.
\end{lem}
\begin{proof}
By the same argument as in \cite[\S1.2]{He:white},
we can prove the restriction of $\gamma_{\rho(\bsk)}$ to $(\mathscr R\otimes S(\mathfrak a_\bbC))^{W,L(\varSigma',\bsk)}$ is injective homomorphism into $S(\mathfrak a_\bbC)^{W}$.
Hence the lemma follows from Proposition \ref{prop:HCisomHO}.
(Recall $W=W'$ by assumption.)
\end{proof}
\begin{prop}\label{prop:2diff-sys}
Suppose
for a choice of $\varSigma'$ and $\bsk$
the equality
\begin{equation}\label{eq:MainAssumption}
\tilde\delta(\varSigma',\bsk)^{\frac12}\circ(L(\varSigma',\bsk)+(\rho(\bsk),\rho(\bsk)))\circ \tilde\delta(\varSigma',\bsk)^{-\frac12}
=
\tilde\delta_{G/K}^{\frac12}\circ(\varDelta^\pi(\varOmega_{\mathfrak g}-\varpi^\pi)+||\rho||^2)\circ \tilde\delta_{G/K}^{-\frac12}
\end{equation}
holds, namely, the operators in \eqref{eq:Ltwist} and \eqref{eq:Omegatwist} coincide.
Then we have
\[
(\mathscr R'\otimes S(\mathfrak a_\bbC))^{W,L(\varSigma',\bsk)}
=
\tilde\delta(\varSigma',\bsk)^{-\frac12}\tilde\delta_{G/K}^{\frac12}\circ
\varDelta^\pi(\boldsymbol D^\pi)
\circ\tilde\delta_{G/K}^{-\frac12}\tilde\delta(\varSigma',\bsk)^{\frac12}.
\]
Moreover,
for any $D\in U(\mathfrak g_\bbC)^K$
it holds that
\[
\gamma_{\rho(\bsk)}
\Bigl(
\tilde\delta(\varSigma',\bsk)^{-\frac12}\tilde\delta_{G/K}^{\frac12}\circ \varDelta^\pi(D) \circ\tilde\delta_{G/K}^{-\frac12}\tilde\delta(\varSigma',\bsk)^{\frac12}
\Bigr)
=
\gamma^\pi(D).
\]
\end{prop}
\begin{proof}
Define an algebra homomorphism
$\tau : U(\mathfrak g_\bbC)^K \to (\mathscr R\otimes S(\mathfrak a_\bbC))^W$ by
\[
\tau(D)=
\tilde\delta(\varSigma',\bsk)^{-\frac12}\tilde\delta_{G/K}^{\frac12}\circ \varDelta^\pi(D) \circ\tilde\delta_{G/K}^{-\frac12}\tilde\delta(\varSigma',\bsk)^{\frac12}.
\]
Suppose $D\in U(\mathfrak g_\bbC)^K$.
Then Proposition \ref{prop:GeneralRad}
implies
$\gamma^\pi(D)=\gamma_\rho\circ \varDelta^\pi(D)$,
while one easily sees
$\gamma_{\rho(\bsk)}
\bigl(
\tilde\delta(\varSigma',\bsk)^{-\frac12}\tilde\delta_{G/K}^{\frac12}\circ E \circ\tilde\delta_{G/K}^{-\frac12}\tilde\delta(\varSigma',\bsk)^{\frac12}
\bigr)
=
\gamma_\rho(E)$
for any $E\in \mathscr R\otimes S(\mathfrak a_\bbC)$.
Thus we have $\gamma_{\rho(\bsk)}\circ \tau(D)=\gamma^\pi(D)$, the second assertion of the proposition.
Now it holds that
\[
[L(\varSigma',\bsk), \tau(D)]
=\tau\bigl(
[\varOmega_{\mathfrak g}-\varpi^\pi+||\rho||^2-(\rho(\bsk),\rho(\bsk)),D]
\bigr)=0.
\]
This shows $\tau(U(\mathfrak g_\bbC)^K)
\subset (\mathscr R\otimes S(\mathfrak a_\bbC))^{W,L(\varSigma',\bsk)}=(\mathscr R'\otimes S(\mathfrak a_\bbC))^{W,L(\varSigma',\bsk)}
$.
But these two subalgebras actually coincide by Theorem \ref{thm:HChomo},
Proposition \ref{prop:HCisomHO} and the second assertion.
\end{proof}

Since the functions $\sinh^{-2}\frac\alpha2$ ($\alpha\in\varSigma\cup2\varSigma$) are linearly independent,
\eqref{eq:MainAssumption} holds if and only if
\begin{equation}\label{eq:matching}
-\bsm_\alpha\bska_\alpha^\pi
+\tfrac12\bsm_{\frac\alpha2}\bigl(
1-\tfrac12\bsm_{\frac\alpha2}-\bsm_{\alpha}+2\bska^\pi_{\frac\alpha2}
\bigr)
=\bsk_\alpha(1-\bsk_\alpha-2\bsk_{2\alpha})
\quad\text{for any }\alpha\in\varSigma\cup2\varSigma.
\end{equation}
Here we suppose $\bsm_\alpha=\bska_\alpha^\pi=0$ for $\alpha\notin\varSigma$
and $\bsk_\alpha=0$ for $\alpha\notin\varSigma'$.

\begin{prop}\label{prop:deltaregular}
The function $\tilde\delta_{G/K}^{-\frac12}\, \tilde\delta(\varSigma',\bsk)^{\frac12} \in \mathscr A(\mathfrak a_\reg)$ extends to a real analytic function on
some open set $U$ containing $\mathfrak a_\reg \cup \{0\}$ if and only if
\begin{equation}\label{eq:regularity}
\frac{\bsm_{\alpha}+\bsm_{2\alpha}}2=\bsk_{\alpha}+\bsk_{2\alpha}+\bsk_{4\alpha}\quad
\text{for any }\alpha\in\varSigma\setminus2\varSigma.
\end{equation}
If this is the case then 
\begin{equation}\label{eq:coshfactor}
\tilde\delta_{G/K}^{-\frac12}\, \tilde\delta(\varSigma',\bsk)^{\frac12}
=\prod_{\alpha\in\varSigma^+\setminus2\varSigma^+}
\biggl(\cosh\frac\alpha2\biggr)^{-\bsk_\alpha}
(\cosh\alpha)^{\bsk_{4\alpha}-\frac{\bsm_{2\alpha}}2}.
\end{equation}
In particular, $\tilde\delta_{G/K}^{-\frac12}\, \tilde\delta(\varSigma',\bsk)^{\frac12}$ extends to a nowhere-vanishing real analytic function on $\mathfrak a$ taking $1$ at $0\in\mathfrak a$.
\end{prop}
\begin{proof}
The first statement is immediate since for each $\alpha\in\varSigma\cup2\varSigma$
we have the expansion
\[
\frac{\sinh(\alpha/2)}{||\alpha/2||}=
\frac\alpha{||\alpha||}\biggl(1+\frac{(\alpha/2)^2}{3!}+\frac{(\alpha/2)^4}{5!}+\cdots\biggr).
\]
Next, \eqref{eq:coshfactor} holds since for each $\alpha\in\varSigma^+\setminus2\varSigma^+$
\begin{align*}
\biggl|\frac{\sinh\alpha}{||\alpha||}&\biggr|^{-\frac{\bsm_{\alpha}}2}
\biggl|\frac{\sinh(2\alpha)}{||2\alpha||}\biggr|^{-\frac{\bsm_{2\alpha}}2}
\biggl|\frac{\sinh(\alpha/2)}{||\alpha/2||}\biggr|^{\bsk_\alpha}
\biggl|\frac{\sinh\alpha}{||\alpha||}\biggr|^{\bsk_{2\alpha}}
\biggl|\frac{\sinh(2\alpha)}{||2\alpha||}\biggr|^{\bsk_{4\alpha}}\\
&=
|\sinh\alpha|^{-\frac{\bsm_{\alpha}}2}
|\sinh\alpha\,\cosh\alpha|^{-\frac{\bsm_{2\alpha}}2}
\biggl|\frac{\sinh\alpha}{\cosh(\alpha/2)}\biggr|^{\bsk_\alpha}
|\sinh\alpha|^{\bsk_{2\alpha}}
|\sinh\alpha\,\cosh\alpha|^{\bsk_{4\alpha}}\\
&=
\biggl(\cosh\frac\alpha2\biggr)^{-\bsk_\alpha}
(\cosh\alpha)^{\bsk_{4\alpha}-\frac{\bsm_{2\alpha}}2}.\qedhere
\end{align*} 
\end{proof}
\begin{thm}\label{thm:matching}
Suppose \eqref{eq:matching} and \eqref{eq:regularity} hold
for a choice of $\varSigma'$ and $\bsk$.
Then $\bsk \in \mathcal K_\reg(\varSigma')$ and it holds that
\begin{equation}\label{eq:main2}
\varUpsilon^\pi(\phi^\pi_\lambda) =\tilde\delta_{G/K}^{-\frac12}\, \tilde\delta(\varSigma',\bsk)^{\frac12}\,F(\varSigma',\bsk,\lambda)
\quad\text{for any }\lambda\in \mathfrak a_\bbC^*.
\end{equation}
\end{thm}
\begin{proof}
Suppose $\lambda\in \mathfrak a_\bbC^*$.
By Proposition \ref{prop:deltaregular},
$\tilde\delta(\varSigma',\bsk)^{-\frac12}\,\tilde\delta_{G/K}^{\frac12}\, 
\varUpsilon^\pi(\phi^\pi_\lambda)$
extends to a real analytic function on $\mathfrak a$ taking $1$ at $0\in\mathfrak a$. 
This is clearly $W$-invariant.
Also, it follows from Corollary \ref{cor:spherical} and Proposition \ref{prop:2diff-sys} that
this function satisfies \ref{HG2} in \S\ref{subsec:HGD}.
Thus
$\tilde\delta(\varSigma',\bsk)^{-\frac12}\,\tilde\delta_{G/K}^{\frac12}\, 
\varUpsilon^\pi(\phi^\pi_\lambda)$ satisfies \ref{HG1}--\ref{HG3} for any $\lambda\in \mathfrak a_\bbC^*$.
Hence $\bsk \in \mathcal K_\reg(\varSigma')$ by the implication \ref{k6}$\Rightarrow$\ref{k1} in Theorem \ref{thm:existence}.
Finally, \eqref{eq:main2} follows from Corollary \ref{cor:existunique}.
\end{proof}

In the notation of Theorem \ref{thm:main} our result is summarized as follows.
Let $\varSigma^\pi$ be a root system in $\mathfrak a^*$
and $\bsk^\pi$ a multiplicity function on $\varSigma^\pi$.
Then \eqref{eq:main} holds if the following conditions are satisfied:
\begin{enumerate}[label=(MC\arabic*), leftmargin=*]
\item $\varSigma^\pi \subset \varSigma\cup 2\varSigma$;\label{MC1}
\item
\eqref{eq:matching} holds with $(\varSigma',\bsk)=(\varSigma^\pi,\bsk^\pi)$;\label{MC2}
\item
\eqref{eq:regularity} holds with $(\varSigma',\bsk)=(\varSigma^\pi,\bsk^\pi)$.\label{MC3}
\end{enumerate}
Note that if \ref{MC3} is true, then
each root in $\varSigma$ is proportional to some root in $\varSigma^\pi$.
Hence the Weyl group of $\varSigma^\pi$ equals $W$ under \ref{MC1} and \ref{MC3}.
It is not so hard to observe that
under \ref{MC1}, \eqref{eq:main} holds only if 
both \ref{MC2} and \ref{MC3} are true.
(We do not use this fact in the paper.)
Conditions \ref{MC1}--\ref{MC3} will be further simplified
after we look into the structure of $(\pi,V)$ more precisely.

\subsection{The associated split semisimple subgroup}\label{subsec:asssplit}
Let $\mathfrak b$ be a Cartan subalgebra of $\mathfrak m=\Lie M$.
Then $\mathfrak b_\bbC+\mathfrak a_\bbC$ is a Cartan subalgebra of $\mathfrak g_\bbC$.
A root $\mu$ for $(\mathfrak g_\bbC,\mathfrak b_\bbC+\mathfrak a_\bbC)$ is called \emph{real}
when $\mu|_{\mathfrak b_\bbC}=0$.
We denote the set of real roots by $\varSigma_\real$, which is naturally identified with a subset of $\varSigma$.
A restricted root $\alpha\in\varSigma$ belongs to  $\varSigma_\real$ if and only if $\bsm_\alpha$ is odd
(cf.~\cite[Chapter X, Exercises F]{Hel1}).
Now 
\begin{equation*}
{\mathfrak g}^{\mathfrak b}=\mathfrak a + \mathfrak b + \sum_{\alpha\in\varSigma_\real} \mathfrak g_\alpha^{\mathfrak b}
\end{equation*}
is a reductive subalgebra of $\mathfrak g$.
Its semisimple part is
\begin{equation}\label{eq:g_split}
\mathfrak g_\spt:=[{\mathfrak g}^{\mathfrak b},{\mathfrak g}^{\mathfrak b}]
= \mathfrak a_\spt + \sum_{\alpha\in\varSigma_\real} \mathfrak g_\alpha^{\mathfrak b}
\qquad
\biggl(
\mathfrak a_\spt:=\sum_{\alpha\in\varSigma_\real}\bbR H_\alpha
\biggr),
\end{equation}
which is a split semisimple Lie algebra with Cartan subalgebra $\mathfrak a_\spt$ (\cite[Chapter VII, \S5]{Kn}).
The restricted root system of $\mathfrak g_\spt$ is identified with $\varSigma_\real$.
Let $G_\spt$ be the analytic subgroup for $\mathfrak g_\spt$ (the \emph{associated split semisimple subgroup}).
Let $M_0$ denote the identity component of $M$.
If we put $K_\spt=K\cap G_\spt$, $M_\spt=M\cap K_\spt$,
then $M_\spt$ is the centralizer of $\mathfrak a_\spt$ in $K_\spt$.
Furthermore $M_\spt$ normalizes $M_0$, and
$M=M_0 M_\spt$ (cf.~\cite[Theorem 7.52]{Kn}).

\begin{prop}\label{prop:split_reduce}
The restriction of 
$(\pi,V)$ to $M_0$ is isomorphic to the direct sum of some copies of an irreducible representation of $M_0$:
$V|_{M_0} \simeq U^{\oplus r}$.
Let $\mu$ be an extremal $\mathfrak b_\bbC$-weight of $U$
and put
\begin{equation}\label{eq:V_mu}
V_\mu=\{ v\in V \,|\, \pi(H)v=\mu(H)v\ (\forall H\in \mathfrak b)\}.
\end{equation}
Then $\pi(K_\spt)V_{\mu} \subset V_{\mu}$
and $V_{\mu}$ is irreducible as an $M_\spt$-module.
That is, $(\pi|_{K_\spt}, V_\mu)$ is a small $K_\spt$-type of $G_\spt$ with dimension $r$.
\end{prop}
\begin{proof}
Let $\mu$ be the highest weight of $V|_{M_0}$
with respect to a lexicographical order of $\sqrt{-1} \mathfrak b^*$ and
define $V_\mu$ by \eqref{eq:V_mu}.
Since $K_\spt\subset G^{\mathfrak b}$, we have $\pi(K_\spt)V_{\mu} \subset V_{\mu}$.
Let $E\subset V_{\mu}$ be an irreducible $M_\spt$-submodule.
Since $M_\spt$ normalizes $M_0$,
$\pi(U(\mathfrak m_\bbC))E$ is $M$-stable and hence is equal to $V$.
By the highest weight theory,
$E= V_{\mu} \cap \pi(U(\mathfrak m_\bbC))E = V_{\mu}$.
Thus $V_{\mu}$ is an irreducible $M_\spt$-module.
By the highest weight theory again,
we see if $U$ is an irreducible $M_0$-module with highest weight $\mu$
then $V|_{M_0} \simeq U^{\otimes r}$ with $r=\dim V_{\mu}$.
\end{proof}
\begin{cor}\label{cor:evenis0}
For any $\alpha\in\varSigma$ with even $\bsm_\alpha$
we have $\bska_\alpha^\pi=0$.
\end{cor}
\begin{proof}
Let $\varDelta_{\mathfrak m}$ be the root system for $(\mathfrak m_\bbC,\mathfrak b_\bbC)$.
Suppose $\alpha\in\varSigma$ 
has an even root multiplicity.
Then $\alpha\notin\varSigma_\real$
and
all $\mathfrak b_\bbC$-weights of $\mathfrak m_\bbC$-module $(\mathfrak g_\alpha)_\bbC$
are not zero.
Hence by the representation theory of complex reductive Lie algebra,
any $\mathfrak b_\bbC$-weight of $(\mathfrak g_\alpha)_\bbC$
(or equivalently, that of  $\{X_\alpha+\theta X_\alpha\,|\,X_\alpha\in (\mathfrak g_\alpha)_\bbC\}$)
is outside the root lattice $\bbZ\varDelta_{\mathfrak m}$.
Let $\mu$ be as in Proposition \ref{prop:split_reduce}.
Then by the proposition,
$\mu-\lambda$ belongs to $\bbZ\varDelta_{\mathfrak m}$
for any $\mathfrak b_\bbC$-weight $\lambda$ of $V$.
But this means
the difference of $\mu$ and any  $\mathfrak b_\bbC$-weight of the $M$-submodule
$\{\pi(X_\alpha+\theta X_\alpha) v\,|\, X_\alpha\in \mathfrak g_\alpha, v\in V\}\subset V$
is also inside $\bbZ\varDelta_{\mathfrak m}$.
It is possible only when
$\{\pi(X_\alpha+\theta X_\alpha) v\,|\, X_\alpha\in \mathfrak g_\alpha, v\in V\}=\{0\}$.
Hence Proposition \ref{prop:ND} implies $\bska_\alpha^\pi=0$.
\end{proof}

\begin{cor}\label{cor:realinherit}
Let $V_\mu$ be as in Proposition \ref{prop:split_reduce}
and put $(\pi_\mu, V_\mu)=(\pi|_{K_\spt}, V_\mu)$.
For any $\alpha\in\varSigma$ with $\bsm_\alpha=1$,
we have $\alpha\in\varSigma_\real$
and $\bska^\pi_\alpha=\bska^{\pi_\mu}_\alpha$.
\end{cor}

\begin{proof}
Suppose $\alpha\in\varSigma$ with $\bsm_\alpha=1$ is given.
One has  $\alpha\in\varSigma_\real$ and $\mathfrak g_\alpha\subset \mathfrak g_\spt$.
Take $X_\alpha\in\mathfrak g_\alpha$
so that $-\frac{||\alpha||^2}2 B(X_\alpha,\theta X_\alpha)=1$.
Then
$\bska_\alpha^\pi \id_V=\pi(X_\alpha+\theta X_\alpha)^2$ by definition.
Note the normalization condition for $X_\alpha$
is rewritten as $[X_\alpha,[X_\alpha,\theta X_\alpha]]=2X_\alpha$,
which is common to both $\mathfrak g$ and $\mathfrak g_\spt$. 
Hence
$\bska_\alpha^{\pi_\mu} \id_{V_\mu}=\pi_\mu(X_\alpha+\theta X_\alpha)^2=
\pi(X_\alpha+\theta X_\alpha)^2|_{V_\mu}=\bska_\alpha^\pi \id_{V_\mu}$.
\end{proof}

\subsection{Simplifying matching conditions}

\begin{prop}\label{prop:simplified}
Let $R$ be a complete system of representatives for the $W$-orbits of $\varSigma\setminus2\varSigma$.
Let $\varSigma^\pi$ be a root system in $\mathfrak a^*$ satisfying \ref{MC1}.
Then a multiplicity function $\bsk^\pi$ on $\varSigma^\pi$
satisfies \ref{MC2} and \ref{MC3} if and only if
the following are valid:
\begin{enumerate}[label=(\arabic*), leftmargin=*, topsep=0pt]
\item\label{reduced1}
for any $\alpha\in R$ with $\bsm_{2\alpha}=0$
\[\left\{
\begin{aligned}
\bsk^\pi_\alpha &=\frac{\bsm_\alpha-1\pm\sqrt{(\bsm_\alpha-1)^2-4\bsm_\alpha\bska^\pi_\alpha}}2,\\
\bsk^\pi_{2\alpha} &=\frac{1\mp\sqrt{(\bsm_\alpha-1)^2-4\bsm_\alpha\bska^\pi_\alpha}}2;
\end{aligned} 
\right.\]
\item\label{reduced2}
for any $\alpha\in R$ with $\bsm_{2\alpha}>0$
\[
\left\{
\begin{aligned}
\bsk^\pi_\alpha &=0,\\
\bsk^\pi_{2\alpha} &=\frac{\bsm_\alpha+\bsm_{2\alpha}-1\pm\sqrt{(\bsm_{2\alpha}-1)^2-4\bsm_{2\alpha}\bska^\pi_{2\alpha}}}2,\\
\bsk^\pi_{4\alpha} &=\frac{1\mp\sqrt{(\bsm_{2\alpha}-1)^2-4\bsm_{2\alpha}\bska^\pi_{2\alpha}}}2,
\end{aligned} 
\right.
\]
or
\[
\bska^\pi_{2\alpha}=\frac14\bsm_{2\alpha}-\frac12
\quad\text{and}\quad
\left\{
\begin{aligned}
\bsk^\pi_\alpha &=\bsm_\alpha+\bsm_{2\alpha}-1,\\
\bsk^\pi_{2\alpha} &=1-\frac{\bsm_\alpha+\bsm_{2\alpha}}2,\\
\bsk^\pi_{4\alpha} &=0.
\end{aligned} 
\right.
\]
\end{enumerate}
\end{prop}
\begin{proof}
Suppose $\alpha\in R$ and $\bsm_{2\alpha}>0$.
Then $\bsm_{\alpha}$ is even (cf.~\cite[Chapter X, Exercises F]{Hel1})
and $\bska^\pi_\alpha=0$ by Corollary \ref{cor:evenis0}.
Thus those parts of \eqref{eq:matching} and \eqref{eq:regularity}
that relate to $\alpha,2\alpha$ and $4\alpha$ are reduced to:
\begin{equation}\label{eq:MCBC1}
\left\{
\begin{aligned}
0 &= \bsk^\pi_{\alpha}(1-\bsk^\pi_{\alpha}-2\bsk^\pi_{2\alpha}) ,\\
-\bsm_{2\alpha}\bska^\pi_{2\alpha}
+\tfrac12 \bsm_\alpha(1-\tfrac12\bsm_\alpha-\bsm_{2\alpha})
&=\bsk^\pi_{2\alpha}(1-\bsk^\pi_{2\alpha}-2\bsk^\pi_{4\alpha}),\\
\tfrac12\bsm_{2\alpha}(1-\tfrac12\bsm_{2\alpha}+2\bska^\pi_{2\alpha})
&=\bsk^\pi_{4\alpha}(1-\bsk^\pi_{4\alpha}),\\
\tfrac12(\bsm_\alpha+\bsm_{2\alpha})
&=\bsk^\pi_\alpha+\bsk^\pi_{2\alpha}+\bsk^\pi_{4\alpha}.
\end{aligned}
\right.
\end{equation}
In addition, since $\varSigma^\pi$ is a root system, either $\bsk^\pi_\alpha$ or $\bsk^\pi_{4\alpha}$ is zero.
Hence by an elementary argument
\eqref{eq:MCBC1} is still reduced to the condition in \ref{reduced2}.
The condition in  \ref{reduced1} for $\alpha\in R$ with $\bsm_{2\alpha}=0$ is obtained
in a similar way.
\end{proof}

\section{Case-by-case analysis}\label{sec:CC}

In this section,
all the results in \S\ref{sec:list}
will be proved
through case-by-case analysis.
We start with some preparation.
Let $G$ be a
non-compact real simple Lie group with finite center.
Note that $G$ is connected by definition.
\begin{lem}\label{lem:k_1m}
Suppose $\mathfrak k_1$ is an ideal of $\mathfrak k$ such that $\mathfrak k_1\subset \mathfrak m$.
Then $\mathfrak k_1=\{0\}$.
\end{lem}
\begin{proof}
By assumption,
one has for any $k\in K$
\[
[\mathfrak k_1, \Ad(k)\mathfrak a]
=
\Ad(k)[\Ad(k^{-1})\mathfrak k_1, \mathfrak a]
=
\Ad(k)[\mathfrak k_1, \mathfrak a]
\subset
\Ad(k)[\mathfrak m, \mathfrak a]
=\{0\}.
\]
But since $\mathfrak s=\bigcup_{k\in K}\Ad(k)\mathfrak a$,
$\mathfrak k_1$ is an ideal of $\mathfrak g$,
which must be $\{0\}$ since $\mathfrak k_1 \subset \mathfrak k \ne \mathfrak g$.
\end{proof}

\begin{cor}\label{cor:kgen}
The Lie algebra
$\mathfrak k$ is generated by $\{X_\alpha+\theta X_\alpha \,|\, X_\alpha\in\mathfrak g_\alpha\ (\alpha\in\varSigma)\}$.
\end{cor}
\begin{proof}
Note $\mathfrak k=\mathfrak m \oplus \sum_{\alpha\in\varSigma}\{X_\alpha+\theta X_\alpha\,|\,X_\alpha\in\mathfrak g_\alpha\}$
and $\{X_\alpha+\theta X_\alpha\,|\,X_\alpha\in\mathfrak g_\alpha\}$
is $M$-stable for each $\alpha\in\varSigma$.
It follows that the Lie subalgebra $\mathfrak k_0$ generated by $\{X_\alpha+\theta X_\alpha \,|\, X_\alpha\in\mathfrak g_\alpha\ (\alpha\in\varSigma)\}$
is an ideal of $\mathfrak k$.
Hence the orthogonal complement $\mathfrak k_1$ of $\mathfrak k_0$ in $\mathfrak k$
with respect to $B(\cdot,\cdot)$ is an ideal satisfying the assumption of Lemma \ref{lem:k_1m}.
Thus we get $\mathfrak k_1=\{0\}$ and $\mathfrak k_0=\mathfrak k$.
\end{proof}

\subsection{\boldmath The trivial $K$-type}\label{CC:trivial}
\begin{prop}\label{prop:trivial}
A small $K$-type $(\pi,V)$ is trivial if and only if
\begin{equation}\label{eq:kappa_cond_triv}
\bska^\pi_\alpha=0
\quad\text{for any }\alpha\in\varSigma.
\end{equation}
\end{prop}
\begin{proof}
Note $\Ker_{\mathfrak k}\pi:=\{X\in\mathfrak k\,|\,\pi(X)=0\}$
is an ideal of $\mathfrak k$  (and in particular, it is a subalgebra).
If we assume \eqref{eq:kappa_cond_triv},
then $\Ker_{\mathfrak k}\pi=\mathfrak k$
by Proposition \ref{prop:ND} and Corollary \ref{cor:kgen},
showing $\pi$ is trivial.
The converse is clear from Proposition \ref{prop:ND}.
\end{proof}

The result on the trivial $K$-type
stated in \S\ref{subsec:trivial} readily follows from \eqref{eq:kappa_cond_triv}
and Proposition \ref{prop:simplified}.

\subsection{Complex simple Lie groups}\label{CC:complex}
Let $G$ be a complex simple Lie group and $(\pi,V)$ a small $K$-type of $G$.
Then for any $\alpha\in\varSigma$
we have $\bsm_\alpha=2$, and hence $\bska^\pi_\alpha=0$ by Corollary \ref{cor:evenis0}.
Thus by Proposition \ref{prop:trivial} $(\pi,V)$ is the trivial $K$-type.
Moreover the right-hand side of \eqref {eq:Omegatwist} equals $\varOmega_{\mathfrak a}$.
From \cite[Chapter IV, \S5, No.2]{Hel2} we have
\[
\varUpsilon^\pi(\phi^\pi_\lambda) 
=\frac{\prod_{\alpha\in\varSigma^+} \rho(H_\alpha)}{\prod_{\alpha\in\varSigma^+} \lambda(H_\alpha)} \frac{\sum_{w\in W}\sgn(w)e^{w\lambda}}{\sum_{w\in W}\sgn(w)e^{w\rho}}.
\]
Put $\varSigma^\pi=c\varSigma$
and $\bsk^\pi\equiv1$ with any $c>0$.
Then one easily has for any $\lambda\in\mathfrak a_\bbC^*$
\begin{align*}
F(\varSigma^\pi,\bsk^\pi,\lambda)
&=\frac{\prod_{\alpha\in\varSigma^+} \frac{c\rho}2(H_\alpha)}{\prod_{\alpha\in\varSigma^+} \lambda(H_\alpha)} \frac{\sum_{w\in W}\sgn(w)e^{w\lambda}}{\sum_{w\in W}\sgn(w)e^{w\frac{c\rho}2}},\\
\tilde\delta_{G/K}^{-\frac12}\, \tilde\delta(\varSigma^\pi,\bsk^\pi)^{\frac12}
&=\frac{\prod_{\alpha\in\varSigma^+} \rho(H_\alpha)}{\prod_{\alpha\in\varSigma^+} \frac{c\rho}2(H_\alpha)} \frac{\sum_{w\in W}\sgn(w)e^{w\frac{c\rho}2}}{\sum_{w\in W}\sgn(w)e^{w\rho}}.
\end{align*}
Hence \eqref{eq:main} holds for infinitely many combinations of $\varSigma^\pi$ and $\bsk^\pi$.

\subsection{\boldmath Other simple Lie groups having no non-trivial small $K$-type}\label{CC:onlytrivial}
We have no non-trivial small $K$-type for
those real simple Lie groups $G$ with the following Lie algebras:\smallskip
\begin{center}
\begin{tabular}{c|ccc}
$\mathfrak g$&
$\mathfrak{sl}(p,\bbH)\ (p\ge2)$&
$\mathfrak{so}(2r+1,1)\ (r\ge1)$&
$\mathfrak e_{6(-26)}\ (\text{E\,IV})$\\
\hline
$\varSigma$&
$A_{p-1}$&
$A_1$&
$A_2$\\
$\bsm_{\alpha}$&
$4$&
$2r$&
$8$\\
$\mathfrak k$&
$\mathfrak{sp}(p)$&
$\mathfrak{so}(2r+1)$&
$\mathfrak{f}_4$
\end{tabular}\smallskip
\end{center}
\begin{flushright}
($\mathfrak{so}(3,1)\simeq \mathfrak{sl}(2,\bbC),\mathfrak{sl}(2,\bbH)\simeq\mathfrak{so}(5,1)$)
\end{flushright}
\begin{center}
\begin{tabular}{c|c}
$\mathfrak g$&
$\mathfrak f_{4(-20)}\ (\mathrm{F\,II})$\\
\hline
$\varSigma$&
$\{\pm\alpha,\pm2\alpha\}\ ((BC)_1)$\\
$\bsm_{\alpha}$&
$8$\\
$\bsm_{2\alpha}$&
$7$\\
$\mathfrak k$&
$\mathfrak{so}(9)$
\end{tabular}
\end{center}
\smallskip
The argument for the first three cases is the same as for the complex case.
Suppose $\mathfrak g=\mathfrak f_{4(-20)}$ and $(\pi,V)$ is a small $K$-type of $G$.
Let $\alpha$ is a short restricted root.
Then it follows from Corollary \ref{cor:evenis0} and Proposition \ref{prop:ND}
that $X_\alpha+\theta X_\alpha \in \Ker_{\mathfrak k} \pi\setminus\{0\}$
for any $X_\alpha \in \mathfrak g_\alpha\setminus\{0\}$.
Now since 
$\Ker_{\mathfrak k} \pi$ is an ideal of the simple Lie algebra $\mathfrak k\simeq\mathfrak{so}(9)$, one has
$\Ker_{\mathfrak k} \pi=\mathfrak k$ and hence $(\pi,V)$ is the trivial $K$-type.

\subsection{\boldmath The case $\mathfrak g=\mathfrak{sp}(p,q)$}\label{CC:sp(p,q)}
Suppose $G=\Sp(p,q)\ (p\ge q \ge 1)$ and $K=\Sp(p)\times \Sp(q)$.
Then $G$ is connected, simply-connected Lie group and $M=M_0$ (cf.~\cite[Appendix C, \S3]{Kn}).
Let $\bbH=\bbR + \bbR \boldsymbol i +\bbR \boldsymbol j +\bbR \boldsymbol k$
be the field of quaternions.
We use the following realization:
\begin{align*}
\mathfrak g&=\left\{
\begin{pmatrix}
A & B \\ \overline{ \trans B} & C
\end{pmatrix} \in \mathfrak{gl}(p+q,\bbH) \,\middle|\,
A\in \mathfrak{gl}(p,\bbH),
C\in \mathfrak{gl}(q,\bbH),
\overline{ \trans A}=-A, \overline{ \trans C}=-C
\right\},\\
\mathfrak k&=\left\{
\begin{pmatrix}
A & O_{p,q} \\ O_{q,p} & C
\end{pmatrix} 
\,\middle|\,
A\in \mathfrak{gl}(p,\bbH),
C\in \mathfrak{gl}(q,\bbH),
\overline{ \trans A}=-A, \overline{ \trans C}=-C
\right\}\simeq\mathfrak{sp}(p)\oplus\mathfrak{sp}(q),\\
\mathfrak a&=\left\{
H(\boldsymbol a):=\begin{pmatrix}
O_{q,q} & O_{q,p-q} & \diag(a_1,\ldots,a_q) \\
O_{p-q,q} & O_{p-q,p-q} & O_{p-q,q} \\
\diag(a_1,\ldots,a_q) & O_{q,p-q} & O_{q,q}
\end{pmatrix}
  \,\middle|\,
\boldsymbol a=(a_1,\ldots, a_q) \in \bbR^q 
\right\},\\
\mathfrak m&=\left\{
\begin{pmatrix}
\diag(m_1,\ldots,m_q) & O_{q,p-q} & O_{q,q} \\
O_{p-q,q} & Y & O_{p-q,q} \\
O_{q,q} & O_{q,p-q} & \diag(m_1,\ldots,m_q)
\end{pmatrix}
 \,\middle|\,
\begin{aligned}
&m_1,\ldots, m_q \in \bbH, \\
&\overline m_i + m_i = 0\ (1\le i \le q), \\
&Y \in \mathfrak{gl}(p-q,\bbH), \overline{\trans Y}=-Y
\end{aligned}
\right\}\\
& \simeq\mathfrak{su}(2)^q \oplus\mathfrak{sp}(p-q). 
\end{align*}
Define $e_i\in\mathfrak a^*$ by
$e_i(H(\boldsymbol a))=a_i$ $(i=1,\ldots,q)$.
Then $\varSigma\subset\{\pm e_i,\pm 2e_i\,|\,1\le i\le q\} \cup \{\pm e_i\pm e_j\,|\,1\le i< j \le q\}$ and the multiplicity of each restricted root is as follows:\smallskip
\begin{center}
\begin{tabular}{c|c}
$\mathfrak g$&
$\mathfrak{sp}(p,q)\ (p\ge q\ge1)$\\
\hline
real rank&
$q$\\
$\bsm_{\mathrm{short}}:=\bsm_{\pm e_i}\ (1\le i\le q)$&
$4(p-q)$\\
$\bsm_{\mathrm{middle}}:=\bsm_{\pm e_i\pm e_j}\ (1\le i<j\le q)$&
$4$ ($q\ge 2$)\\
$\bsm_{\mathrm{long}}:=\bsm_{\pm 2e_i}\ (1\le i\le q)$&
$3$
\end{tabular}
\end{center}
\smallskip
Let $\pr_1$ and $\pr_2$ be the projections of $K$
to $\Sp(p)$ and $\Sp(q)$ respectively.
\begin{thm}\label{thm:sp(p,q)}
If $p\ge q \ge2$ then $G=\Sp(p,q)$ has no non-trivial small $K$-type.
Suppose $p \ge q=1$.
Then for the irreducible representation $(\pi_n,\bbC^n)$ of\/ $\Sp(1)\simeq \SU(2)$
of dimension $n=1,2,\dotsc$,
$\pi_n\circ\pr_2$
is a small $K$-type of $G=\Sp(p,1)$
with $\bska^{\pi_n\circ\pr_2}_{\mathrm{short}}=0$ and $\bska^{\pi_n\circ\pr_2}_{\mathrm{long}}=-\frac{n^2-1}3$.
If $p>q$ then all the small $K$-types are constructed in this way.
If $p=q=1$ then the other small $K$-types are constructed
in the same way as above but using\/ $\pr_1$ instead of\/ $\pr_2$
and $\bska^{\pi_n\circ\pr_1}=\bska^{\pi_n\circ\pr_2}$ for any $n=1,2,\dotsc$.
\end{thm}
\begin{proof}
Suppose first $p\ge q\ge2$.
Then we have a restricted root vector
\[
X_{e_1-e_2}=
\left(\begin{array}{c|c|c|c}
\begin{matrix}
0 & -1 \\
1 & 0
\end{matrix}
& O_{2,p-2} &
\begin{matrix}
0 & -1 \\
-1 & 0
\end{matrix}
& O_{2,q-2} \\
\hline
\vphantom{\biggm|} O_{p-2,2} & O_{p-2,p-2} & O_{p-2,2} & O_{p-2,q-2} \\\hline
\begin{matrix}
0 & -1 \\
-1 & 0
\end{matrix}
& O_{2,p-2} &
\begin{matrix}
0 & -1 \\
1 & 0
\end{matrix}
& O_{2,q-2} \\\hline
\vphantom{\biggm|} O_{q-2,2} & O_{q-2,p-2} & O_{q-2,2} & O_{q-2,q-2}
\end{array}\right) \in \mathfrak g_{e_1-e_2}.
\]
Observe that $X_{e_1-e_2}+\theta X_{e_1-e_2}$ belongs to
neither $\mathfrak{sp}(p)$ nor $\mathfrak{sp}(q)$.
Thus there is no proper ideal of $\mathfrak k$ that contains $X_{e_1-e_2}+\theta X_{e_1-e_2}$.
Now, 
for any small $K$-type $(\pi,V)$ of $G$,
$X_{e_1-e_2}+\theta X_{e_1-e_2} \in \Ker_{\mathfrak k}\pi$ 
by Corollary \ref{cor:evenis0} and Proposition \ref{prop:ND}.
This means $\Ker_{\mathfrak k}\pi=\mathfrak k$ and hence $(\pi,V)$ is the trivial $K$-type.

Next suppose $p>q=1$ 
Then $\pi_n\circ\pr_2$ is small since $\pr_2(M)=\Sp(1)$.
Also, $\bska^{\pi_n\circ\pr_2}_{\mathrm{short}}=0$
by Corollary \ref{cor:evenis0}.
To calculate $\bska^{\pi_n\circ\pr_2}_{\mathrm{long}}$
take a root vector
\[
X_{2e_1}=\frac12\begin{pmatrix}
\boldsymbol i & O_{1,p-1} & -\boldsymbol i \\
O_{p-1,1} & O_{p-1,p-1} & O_{p-1,1} \\
\boldsymbol i & O_{1,p-1} & -\boldsymbol i
\end{pmatrix}
\in\mathfrak g_{2e_1},
\]
which is normalized as in Lemma \ref{lem:kappa}.
Under
\[
\mathfrak{sp}(1)
\ni
b \boldsymbol i
+c \boldsymbol j
+d \boldsymbol k
\mapsto
\begin{pmatrix}
b\sqrt{-1} & c + d\sqrt{-1}\\
-c + d\sqrt{-1} & -b \sqrt{-1}
\end{pmatrix}
\in \mathfrak{su}(2),
\]
$\pr_2(X_{2e_1}+\theta X_{2e_1})$
maps to $\begin{pmatrix} -\sqrt{-1} & 0 \\ 0 & \sqrt{-1}\end{pmatrix}$.
Hence $\pi_n\circ\pr_2(X_{2e_1}+\theta X_{2e_1}) \simeq \sqrt{-1}\diag(n-1,n-3,\ldots,-n+1)$
and from Lemma \ref{lem:kappa} one has $\bska^{\pi_n\circ\pr_2}_{\mathrm{long}}=-\frac{n^2-1}{3}$.
Take any $X_{e_1}\in \mathfrak g_{e_1}\setminus\{0\}$.
We have
$X_{e_1}+\theta X_{e_1} \in \Ker_{\mathfrak k}(\pi_n\circ\pr_2)$
by Corollary \ref{cor:evenis0} and Proposition \ref{prop:ND}.
Since $\Ker_{\mathfrak k}(\pi_n\circ\pr_2)=\mathfrak{sp}(p)$ for $n\ge 2$,
we see $X_{e_1}+\theta X_{e_1} \in \mathfrak{sp}(p)$
and $\mathfrak{sp}(p)$ is generated by $X_{e_1}+\theta X_{e_1}$
as an ideal of $\mathfrak k$.
Now, for any small $K$-type $(\pi,V)$ of $G$,
$X_{e_1}+\theta X_{e_1} \in \Ker_{\mathfrak k}\pi$ by the same reason as above.
This means $\mathfrak{sp}(p) \subset \Ker_{\mathfrak k}\pi$
and $\pi$ equals some $\pi_n\circ\pr_2$.

Finally suppose $p=q=1$.
Then $M=\SU(2)$ is diagonally embedded to $K=\SU(2)\times \SU(2)$.
Since any $K$-type is given as
the exterior tensor product $\pi_m\boxtimes \pi_n$
of two irreducible representations of $\SU(2)$,
its restriction to $M$ equals the interior tensor product $\pi_m\otimes \pi_n$.
By the representation theory of $\SU(2)$,
$\pi_m\otimes \pi_n$ is irreducible if and only if either $\pi_m$ or $\pi_n$
is trivial.
Hence $\{\pi_n\circ\pr_i\,|\,i=1,2,\, n=1,2,\ldots\}$
is the complete set of small $K$-types.
The values of $\bska^{\pi_n\circ\pr_i}$
are calculated in the same way as in the previous case.
\end{proof}

The result of \S\ref{subsec:onlytrivial} follows from this theorem
and what we discussed in \S\S\ref{CC:complex}--\ref{CC:sp(p,q)}.
Also, Theorem \ref{thm:sp(p,q)} and Proposition \ref{prop:simplified} easily
imply the result of \S\ref{subsec:sp(p,1)}.

\subsection{\boldmath The case $\mathfrak g=\mathfrak{so}(p,q)$}
Suppose $\mathfrak g =\mathfrak{so}(p,q)\ (p\ge q \ge1)$.
(We exclude the cases $\mathfrak{so}(1,1)\simeq \bbR$ and $\mathfrak{so}(2,2)\simeq \mathfrak{sl}(2,\bbR)^{\oplus2}$.)
Under the natural inclusion $\mathfrak{so}(p,q)\subset \mathfrak{sp}(p,q)$,
$\mathfrak k$, $\mathfrak a$ and $\mathfrak m$
are identified with the intersections of $\mathfrak{so}(p,q)$ and those for $\mathfrak{sp}(p,q)$.
In particular, $\mathfrak k=\mathfrak{so}(p)\oplus \mathfrak{so}(q)$ and
\[
\mathfrak m=
\left\{
\begin{pmatrix}
O_{q,q} & O_{q,p-q} & O_{q,q} \\
O_{p-q,q} & Y & O_{p-q,q}\\
O_{q,q} & O_{q,p-q} & O_{q,q}
\end{pmatrix}
 \,\middle|\,
Y \in \mathfrak{so}(p-q)
\right\}
\simeq \mathfrak{so}(p-q).
\]
One has $\varSigma\subset\{\pm e_i\,|\,1\le i\le q\} \cup \{\pm e_i\pm e_j\,|\,1\le i< j \le q\}$ and the multiplicity of each restricted root is as follows:
\begin{center}
\begin{tabular}{c|c}
$\mathfrak g$&
$\mathfrak{so}(p,q)\ (p\ge q\ge1)$\\
\hline
real rank&
$q$\\
$\bsm_{\mathrm{short}}:=\bsm_{\pm e_i}\ (1\le i\le q)$&
$p-q$\\
$\bsm_{\mathrm{long}}:=\bsm_{\pm e_i\pm e_j}\ (1\le i<j\le q)$&
$1$ ($q\ge 2$)
\end{tabular}
\end{center}
\smallskip
Taking some finite covering group of $G$ if necessary, we may assume $K=K_1\times K_2$
with $\mathfrak k_1:=\Lie K_1\simeq\mathfrak{so(p)}$ and $\mathfrak k_2:=\Lie K_2\simeq\mathfrak{so(q)}$.
Furthermore, if $\mathfrak k_i \simeq \mathfrak{so(r)}$ with $r\ge3$,
then we may assume $K_i \simeq \Spin(r)$ ($i=1,2$).
The projections $K\to K_i$ and $\mathfrak k\to \mathfrak k_i$ are denoted by $\pr_i$ ($i=1,2$).
\begin{thm}[{\cite[Theorem 1, Lemmas 4.2, 4.3]{SWL}}]\label{thm:split_so(p,q)}
\begin{enumerate*}[label=(\roman*), itemjoin=\!]
\item\label{i:Dsplit}
Suppose $p=q\ge 3$.
Then a non-trivial $K$-type $\pi$ is small if and only if
it is equivalent to $\sigma\circ\pr_i$
with $i=1,2$ and a (half-)spin representation  $\sigma$ of $K_i$.
For such $\pi$,
$\bska^\pi_{\mathrm{short}}=0$ and $\bska^\pi_{\mathrm{long}}=-\frac14$. \\
\item\label{i:Bsplit1}
Suppose $p=q+1$ with odd $q\ge3$.
Then there are three non-trivial small $K$-types:
$\pi= \sigma\circ\pr_1$  with either of two half-spin representations $\sigma$ of $K_1=\Spin(p)$ and
$\pi= \sigma\circ\pr_2$ with the spin representation $\sigma$ of $K_2=\Spin(q)$.
One has $\bska^\pi_{\mathrm{short}}=-1$, $\bska^\pi_{\mathrm{long}}=-\frac14$ in the former case and $\bska^\pi_{\mathrm{short}}=0$, $\bska^\pi_{\mathrm{long}}=-\frac14$ in the latter case.\\
\item\label{i:Bsplit2}
Suppose $p=q+1$ with even $q\ge4$.
Then a non-trivial $K$-type $\pi$ is small if and only if
it is equivalent to $\sigma\circ\pr_2$
with a half-spin representation $\sigma$ of $K_2=\Spin(q)$.
For such $\pi$,
$\bska^\pi_{\mathrm{short}}=0$ and $\bska^\pi_{\mathrm{long}}=-\frac14$. 
\end{enumerate*}
\end{thm}
We generalize this result to all cases in the subsequent two theorems.
\begin{thm}\label{thm:so(p,1)}
Suppose $p$ is even, $p\ge 4$ and  $q=1$.
Then $K=K_1=\Spin(p)$.
Fix a Cartan subalgebra and a system of positive roots of $\mathfrak k_\bbC$.
For $s=0,1,2,\dotsc$, let $\pi_s^\pm$ be
the irreducible representation of $K=\Spin(p)$ with
highest weight $(s/2,\ldots,s/2,\pm s/2)$ in the standard notation.
Then $\pi_s^\pm$ is a small $K$-type 
with $\bska^{\pi_s^\pm}_{\mathrm{short}}=-\frac{s(s+p-2)}{p-1}$.
There are no other small $K$-types.
\end{thm}
\begin{rem}
We already studied $\mathfrak{so}(2r+1,1)\ (r \ge 1)$ in \S\ref{CC:onlytrivial}
and $\mathfrak{so}(4,1)\simeq \mathfrak{sp}(1,1)$ in Theorem \ref{thm:sp(p,q)}.
Also, $\mathfrak{so}(2,1)\simeq \mathfrak{sl}(2,\bbR)\simeq \mathfrak{su}(1,1)$ will be covered in \S\ref{CC:Hermitian}.
\end{rem}
Let $E_{ij}$ denote a matrix whose entry is $1$ in the $(i,j)$-position
and $0$ elsewhere.
Let $F_{ij}=E_{ij}-E_{ji}$.
\begin{proof}[Proof of Theorem \ref{thm:so(p,1)}]
Suppose first $\pi$ is small.
Then by Proposition \ref{prop:split_reduce}
there is only one isotypic component in
the restriction of $\pi$ to $\mathfrak m=\mathfrak{so}(p-1)$.
In view of the \emph{branching law} for $\mathfrak{so}(p) \downarrow \mathfrak{so}(p-1)$ (cf.~\cite[Theorem 8.1.4]{GW}),
it is possible only when $\pi$ is equivalent to some $\pi_s^\pm$ in the theorem.
Conversely, let $\pi=\pi_s^\pm$ with representation space $V$.
Then $\pi$ is small since $\pi|_{M_0}$ is irreducible by the branching law.
To calculate $\bska^{\pi}_{\mathrm{short}}$
we take restricted root vectors
$X^{(i)} := F_{i,1}-E_{i,p+1}-E_{p+1,i}$ ($2\le i \le p$) 
for $e_1\in\varSigma$.
They constitute an orthonormal basis of $\mathfrak g_{e_1}$,
so that $\bska^{\pi}_{\mathrm{short}}\id_V=\frac1{p-1}\sum_{i=2}^{p}\pi\bigl(X^{(i)}+\theta X^{(i)}\bigr)^2 $.
Now, we assume $H_1,\ldots,{H_{\frac p2}}$
with $H_i:=\sqrt{-1}F_{2i,2i-1}$ ($1\le i \le \frac p2$)
constitute a basis of the Cartan subalgebra of $\mathfrak k_\bbC$
and 
$\varDelta_{\mathfrak k}^+:=\{\varepsilon_i\pm \varepsilon_j\,|\,1\le i <j \le\frac p2\}$
is the system of positive roots,
where we let $\{\varepsilon_i\}$ be the dual basis of $\{H_i\}$.
For $i=2,3,\ldots,\frac p2$ take root vectors
\begin{align*}
X_{\varepsilon_1+\varepsilon_i} &:= \frac12(F_{2i-1,1}+\sqrt{-1}F_{2i-1,2}+\sqrt{-1}F_{2i,1}-F_{2i,2}) \in (\mathfrak k_\bbC)_{\varepsilon_1+\varepsilon_i},\\
X_{\varepsilon_1-\varepsilon_i} &:= \frac12(F_{2i-1,1}+\sqrt{-1}F_{2i-1,2}-\sqrt{-1}F_{2i,1}+F_{2i,2}) \in (\mathfrak k_\bbC)_{\varepsilon_1-\varepsilon_i}.
\end{align*}
Then one has for $i=2,3,\ldots,\frac p2$
\begin{align*}
[X_{\varepsilon_1+\varepsilon_i},X_{\varepsilon_1-\varepsilon_i}]
&=[X_{\varepsilon_1+\varepsilon_i},\overline{X_{\varepsilon_1-\varepsilon_i}}]
=[\overline{X_{\varepsilon_1+\varepsilon_i}},X_{\varepsilon_1-\varepsilon_i}]
=[\overline{X_{\varepsilon_1+\varepsilon_i}},\overline{X_{\varepsilon_1-\varepsilon_i}}]=0,\\
[X_{\varepsilon_1+\varepsilon_i},\overline{X_{\varepsilon_1+\varepsilon_i}}]
&=-(H_{1}+H_i),
\qquad
[X_{\varepsilon_1-\varepsilon_i},\overline{X_{\varepsilon_1-\varepsilon_i}}]
=-(H_{1}-H_i),\\
X^{(2i-1)}+\theta X^{(2i-1)}&=X_{\varepsilon_1+\varepsilon_i}+X_{\varepsilon_1-\varepsilon_i}
+\overline{X_{\varepsilon_1+\varepsilon_i}}+\overline{X_{\varepsilon_1-\varepsilon_i}},\\
X^{(2i)}+\theta X^{(2i)}&=-\sqrt{-1}(X_{\varepsilon_1+\varepsilon_i}-X_{\varepsilon_1-\varepsilon_i}
-\overline{X_{\varepsilon_1+\varepsilon_i}}+\overline{X_{\varepsilon_1-\varepsilon_i}}).
\end{align*}
From these we calculate in $U(\mathfrak k_\bbC)$
\begin{align*}
\sum_{i=2}^{p}\bigl(X^{(i)}+\theta X^{(i)}\bigr)^2
&=(-2\sqrt{-1} H_1)^2+
\sum_{i=2}^{\frac p2} \biggl(\bigl(X^{(2i-1)}+\theta X^{(2i-1)}\bigr)^2
+\bigl(X^{(2i)}+\theta X^{(2i)}\bigr)^2\biggr)\\
&\equiv
-4  H_1^2 -2( p-2) H_1\quad\mod
U\biggl(\sum_{\alpha \in \varDelta_{\mathfrak k}^+} (\mathfrak k_\bbC)_{-\alpha} \biggr)
+ U(\mathfrak k_\bbC)
\sum_{\alpha \in \varDelta_{\mathfrak k}^+} (\mathfrak k_\bbC)_{\alpha}.
\end{align*}
Applying this  to the highest weight vector of $V$,
we obtain $(p-1)\bska^{\pi}_{\mathrm{short}}=-s(s+p-2)$.
\end{proof}
One easily has the result of \S\ref{subsec:so(2r,1)}
by Theorem \ref{thm:so(p,1)} and Proposition \ref{prop:simplified}.
\begin{thm}\label{thm:so(p,q)}
\begin{enumerate*}[label=(\roman*),itemjoin=\!]
\item\label{i:SO1}
Suppose $p>q=2$.
Then a $K$-type $\pi$ is small if and only if
it is equivalent to $\tau\circ\pr_2$
with a one-dimensional representation  $\tau$ of $K_2$.
For such $\pi$,
$\bska^\pi_{\mathrm{short}}=0$.\\
\item\label{i:SO2}
Suppose $p$ is even and $q$ is odd with $p>q\ge3$.
Then one has the same results as in the case of Theorem \ref{thm:split_so(p,q)} \ref{i:Bsplit1}.\\
\item\label{i:SO3}
Suppose $p>q\ge3$ and either $q$ or $p-q$ is even.
Then a non-trivial $K$-type $\pi$ is small if and only if
it is equivalent to $\sigma\circ\pr_2$
with a (half-)spin representation $\sigma$ of $K_2=\Spin(q)$.
For such $\pi$,
$\bska^\pi_{\mathrm{short}}=0$ and $\bska^\pi_{\mathrm{long}}=-\frac14$. 
\end{enumerate*}
\end{thm}
\begin{proof}
Choosing a Cartan subalgebra $\mathfrak b\subset \mathfrak m$ suitably,
we may assume
\begin{align*}
\mathfrak g_\spt&=\left\{
\begin{pmatrix}
A & O_{q,p-q} & B \\
O_{p-q,q} & O_{p-q,p-q} & O_{p-q,q} \\
\trans B & O_{q, p-q} & C
\end{pmatrix} \in \mathfrak{gl}(p+q,\bbR)
\,\middle|\,
\begin{aligned}
&A,C\in \mathfrak{so}(q), \\
&B\in \Mat(q,q,\bbR)
\end{aligned}
\right\}\simeq\mathfrak{so}(q,q)\\
\intertext{if $p-q$ is even, and}
\mathfrak g_\spt&=\left\{
\begin{pmatrix}
A & O_{q+1,p-q-1} & B \\
O_{p-q-1,q+1} & O_{p-q-1,p-q-1} & O_{p-q-1,q} \\
\trans B & O_{q, p-q-1} & C
\end{pmatrix} \in \mathfrak{gl}(p+q,\bbR)
\,\middle|\,
\begin{aligned}
&A\in \mathfrak{so}(q+1), \\
&B\in \Mat(q+1,q,\bbR),\\
&C\in \mathfrak{so}(q)
\end{aligned}
\right\}\\
&\simeq\mathfrak{so}(q+1,q)
\end{align*} 
if $p-q$ is odd.

Suppose $p\ge3$. Let $(\tau,V)$ be a non-trivial irreducible representation of $K_2$
and consider the $K$-type $(\pi,V):=(\tau\circ\pr_2,V)$.
If $q=2$ then this is always small since $V$ is one-dimensional.
For $q\ge3$, we assert that $(\pi,V)$ is small if and only if $\tau$ is
a (half)-spin representation $\sigma$ of $K_2=\Spin(q)$
and that $\bska^\pi_{\mathrm{long}}=-\frac14$ for such $(\pi,V)$.
Indeed, 
if $\tau=\sigma$, then $(\sigma\circ\pr_2|_{K_\spt},V)$ is a small $K$-type 
by Theorem \ref{thm:split_so(p,q)}
and $(\sigma\circ\pr_2|_{M_\spt},V)$ is irreducible.
Since $M_\spt\subset M$, $(\pi,V)$ is small.
Conversely,
if  $(\pi,V)=(\tau\circ\pr_2,V)$ is small, then $V_\mu$ in Proposition \ref{prop:split_reduce}
equals $V$ since $\mathfrak m\subset \mathfrak k_1$ acts on $V$ trivially.
Hence $(\tau\circ\pr_2|_{K_\spt},V)$ is a small $K_\spt$-type
and is non-trivial since $\pr_2(K_\spt)=\pr_2(K)$.
It then follows from Theorem \ref{thm:split_so(p,q)} that
$\tau$ is a (half)-spin representation.
We also have $\bska^\pi_{\mathrm{long}}=-\frac14$ by Corollary \ref{cor:realinherit}.

Next, note $p>q$
in all cases of the theorem and we have
a restricted root vector
\begin{equation}\label{eq:Xe}
X_{e_1}:=F_{q+1,1}-E_{q+1,p+1}-E_{p+1,q+1}
\in \mathfrak g_{e_1},
\end{equation}
for which $X_{e_1}+\theta X_{e_1} \in \mathfrak k_1\setminus\{0\}$.
It is easy to check
there is no proper ideal of $\mathfrak k_1=\mathfrak{so}(p)$ that contains $X_{e_1}+\theta X_{e_1}$. (Note $\mathfrak{so}(p)$ is simple for $p=3,5,6,\dotsc$.)
Hence by Proposition \ref{prop:ND}, a small $K$-type $(\pi, V)$ is written as
$(\pi, V)=(\tau\circ\pr_2,V)$ for some irreducible representation $\tau$ of $K_2$
if and only if $\bska^\pi_{\mathrm{short}}=\bska^\pi_{e_1}=0$.
It follows from Corollary \ref{cor:evenis0}
that if $p-q$ is even then all small $K$-type are of this type.
We claim the same thing holds if $q$ is even.
To show this, we may assume $p$ is odd.
If $(p,q)=(3,2)$, then $\mathfrak{so}(3,2)\simeq \mathfrak{sp}(2,\bbR)$
and the claim in this case will be shown in the first paragraph of the proof of Theorem \ref{thm:Hermite}.
Suppose $(p,q)$ is a general combination of an odd $p$ and an even $q$ ($p>q\ge2$)
and let $(\pi,V)$ be any small $K$-type. 
Let $V_\mu$ be as in Proposition \ref{prop:split_reduce}.
Since $(\pi|_{K_\spt},V_\mu)$ is a small $K_\spt$-type,
it follows from Theorem \ref{thm:split_so(p,q)} \ref{i:Bsplit2} or the claim for $(p,q)=(3,2)$
that $\mathfrak k_1\cap \mathfrak k_\spt =\mathfrak{so}(q+1)$ acts trivially on $V_\mu$.
Note that $F_{2,1}\in \mathfrak k_1\cap \mathfrak k_\spt$
commutes with $\mathfrak m_\bbC$
and that $V=\pi(U(\mathfrak m_\bbC))V_\mu$ by Proposition \ref{prop:split_reduce}.
Thus $F_{2,1}$ acts trivially on $V$.
By the simplicity of $\mathfrak k_1=\mathfrak{so}(p)$, 
we have $\mathfrak k_1\subset \Ker_{\mathfrak k}\pi$,
which proves our claim.
Note that \ref{i:SO1} and \ref{i:SO3} are already proved up to this point.

In order to show \ref{i:SO2},
suppose $p$ is even, $q$ is odd and $p>q\ge3$.
Thanks to Theorem \ref{thm:split_so(p,q)} \ref{i:Bsplit1},
we may also assume $p-q\ge3$.
Let $(\pi,V)$ is a small $K$-type such that $\pi|_{K_1}$ is non-trivial.
We assert that $\pi|_{K_2}$ is trivial.
In fact, if $\mathfrak k_1\cap \mathfrak k_\spt =\mathfrak{so}(q+1)$ acts
trivially on $V_\mu$ in Proposition \ref{prop:split_reduce},
then the same argument as in the last paragraph
implies $\pi|_{K_1}$ is trivial, a contradiction.
Thus the action of $\mathfrak k_1\cap \mathfrak k_\spt=\mathfrak{so}(q+1)$
on $V_\mu$ is non-trivial
and hence
that of $\mathfrak k_2\cap \mathfrak k_\spt=\mathfrak k_2=\mathfrak{so}(q)$
is trivial by Theorem \ref{thm:split_so(p,q)} \ref{i:Bsplit1}.
(We also see $\bska^\pi_{\mathrm{long}}=-\frac14$ by Corollary \ref{cor:realinherit}.)
Since $V=\pi(U(\mathfrak m_\bbC))V_\mu$ and since $\mathfrak k_2$ commutes with $\mathfrak m_\bbC$,
$\mathfrak k_2$ acts trivially on $V$, proving the assertion.
Therefore there exists a non-trivial irreducible representation $\tau$
of $\mathfrak k_1=\mathfrak{so}(p)$ such that $\pi=\tau\circ\pr_1$.
Now, 
by Proposition \ref{prop:split_reduce}
there is only one isotypic component in
the restriction of $(\tau,V)$ to $\mathfrak m=\mathfrak{so}(p-q)$.
In view of the branching laws for the orthogonal Lie algebras (cf.~\cite[Theorems 8.1.3, 8.1.4]{GW}),
it is possible only when $\tau$ is a half-spin representation.
Conversely,
let $(\sigma,V)$ be any of two half-spin representations of $K_1=\Spin(p)$.
The proof is complete if
we can show $\pi=\sigma\circ \pr_1$ is small and $\bska^\pi_{\mathrm{short}}=-1$.
Let $\varpi: \Spin(p)\to \SO(p)$ be the canonical projection.
Then
\begin{align*}
\pr_1(M)&=\varpi^{-1}\left(\left\{
\begin{pmatrix}
\diag(m_1,\ldots,m_q) & O_{q,p-q} \\
O_{p-q,q} & g
\end{pmatrix}\,\middle|\,
m_i=\pm1,
\prod_{i=1}^q m_i =1, g\in \SO(p-q)
\right\}\right)\\
&\supset
\varpi^{-1}\Bigl(\Bigl\{
\diag(m_1,\ldots,m_{p})
\,\Bigm|\, 
m_i=\pm1,
\prod_{i=1}^q m_i=\prod_{i=q+1}^p m_i =1
\Bigr\}\Bigr),\\
\pr_1(M)&\cup\pr_1(M)\cdot\varpi^{-1} (\diag(-1,\ldots,-1))\\
&\supset
\varpi^{-1}\Bigl(\Bigl\{
\diag(m_1,\ldots,m_{p})
\,\Bigm|\, 
m_i=\pm1,
\prod_{i=1}^p m_i =1
\Bigr\}\Bigr).
\end{align*}
Now $V$ is irreducible under the action of $\varpi^{-1}\Bigl(\Bigl\{
\diag(m_1,\ldots,m_{p})
\,\Bigm|\, m_i=\pm1,\prod_{i=1}^p m_i =1
\Bigr\}\Bigr)$.
(In fact, $(\sigma,V)$ is a small `$K$-type' of the double cover of $\SL(p,\bbR)$ by Theorem \ref{thm:split_sl}.)
But since $\varpi^{-1} (\diag(-1,\ldots,-1))$ is contained in the center of $\Spin(p)$, $V$ is irreducible as a $\pr_1(M)$-module.
This proves the smallness of $\pi=\sigma\circ \pr_1$.
Finally, we can directly check
$X_{e_1}$ in \eqref{eq:Xe} is normalized as in Lemma \ref{lem:kappa}
and the eigenvalues of
$\pi(X_{e_1}+\theta X_{e_1})$ are $\pm \sqrt{-1}$.
Thus $\bska^\pi_{\mathrm{short}}=-1$.
\end{proof}
All the results stated in \S\ref{subsec:so(p,q)} follow from
Theorem \ref{thm:so(p,q)} \ref{i:SO2}, \ref{i:SO3} and Proposition \ref{prop:simplified}.

\subsection{The Hermitian type}\label{CC:Hermitian}
Let $G$ be a non-compact real simple Lie group of Hermitian type.
There exists a central element $Z\in \mathfrak k$
such that $J=\ad(Z)$ is
a complex structure of $\mathfrak s=\mathfrak g^{-\theta}$.
Let $2e_1,\ldots2e_l$ be the longest roots in $\varSigma^+$.
Then $\varSigma\subset\{\pm e_i,\pm 2e_i\,|\,1\le i\le l\} \cup \{\pm e_i\pm e_j\,|\,1\le i< j \le l\}$.
Put $\bsm_{\mathrm{long}}=\bsm_{\pm 2e_i}\ (1\le i \le l)$, $\bsm_{\mathrm{middle}}=\bsm_{\pm e_i\pm e_j}\ (1\le i<j \le l)$ and $\bsm_{\mathrm{short}}=\bsm_{\pm e_i}\ (1\le i \le l)$.
Their values are listed below:\smallskip

\begin{center}
\begin{tabular}{c|ccccccc}
$\mathfrak{g}$&
$\mathfrak{su}(p,p)$&
$\mathfrak{sp}(p,\bbR)$&
$\mathfrak{so}^*(4p)\ (p\ge 2)$&
$\mathfrak{so}(p,2)\ (p\ge 3)$&
$\mathfrak{e}_{7(-25)}\ (\mathrm{E\,VII})$\\
\hline
real rank $l$&
$p$&
$p$&
$p$&
$2$&
$3$\\
$\bsm_{\mathrm{short}}$&
$0$&
$0$&
$0$&
$0$&
$0$\\
$\bsm_{\mathrm{middle}}$&
$2$ ($p\ge2$)&
$1$ ($p\ge2$)&
$4$&
$p-2$&
$8$\\
$\bsm_{\mathrm{long}}$&
$1$&
$1$&
$1$&
$1$&
$1$\\
\end{tabular}
\end{center}
\begin{flushright}
($\mathfrak{su}(1,1)\simeq \mathfrak{sp}(1,\bbR)\simeq \mathfrak{sl}(2,\bbR)$,
$\mathfrak{su}(2,2)\simeq \mathfrak{so}(4,2)$,
$\mathfrak{sp}(2,\bbR)\simeq \mathfrak{so}(3,2)$,
$\mathfrak{so}^*(8)\simeq\mathfrak{so}(6,2)$)
\end{flushright}
\smallskip
\begin{center}
\begin{tabular}{c|cccc}
$\mathfrak{g}$&
$\mathfrak{su}(p,q)\ (p>q\ge1)$&
$\mathfrak{so}^*(4p+2)\ (p\ge1)$&
$\mathfrak{e}_{6(-14)}\ (\mathrm{E\,III})$\\
\hline
real rank $l$&
$q$&
$p$&
$2$\\
$\bsm_{\mathrm{short}}$&
$2(p-q)$&
$4$&
$8$\\
$\bsm_{\mathrm{middle}}$&
$2$ ($q\ge2$)&
$4$ ($p\ge2$)&
$6$\\
$\bsm_{\mathrm{long}}$&
$1$&
$1$&
$1$
\end{tabular}
\end{center}
\begin{flushright}
($\mathfrak{su}(3,1)\simeq\mathfrak{so}^*(6)$)
\end{flushright}
\smallskip
Take $X_{2e_i}\in \mathfrak g_{2e_i}$ so that 
$-\frac{||2e_i||^2}{2}B(X_{2e_i},\theta X_{2e_i})=1$ ($1\le i\le l$).
\begin{lem}\label{lem:cpxStr}
By replacing $X_{2e_i}$ with $-X_{2e_i}$ if necessary, we have
\begin{equation}\label{eq:cpxStr}
Z=\frac12\sum_{i=1}^l (X_{2e_i} + \theta X_{2e_i})+Y
\end{equation}
for some $Y\in\mathfrak b$.
Here $\mathfrak b$ is a Cartan subalgebra of $\mathfrak m$.
\end{lem}
\begin{proof}
It is well known that $\mathfrak t:=\sum_{i=1}^l \bbR(X_{2e_i} + \theta X_{2e_i})+\mathfrak b$
is a Cartan subalgebra of $\mathfrak k$.
Since $Z\in\mathfrak t$, there exist constants $c_1,\ldots,c_l\in\bbR$ and $Y\in\mathfrak b$ such that
\[
Z=c_1 (X_{2e_1} + \theta X_{2e_1})+\cdots+c_l (X_{2e_l} + \theta X_{2e_l})+Y.
\]
Since
\begin{gather*}
[X_{2e_i},X_{2e_j}]
=[X_{2e_i},\theta X_{2e_j}]
=[H_{2e_i},X_{2e_j}]
=0\quad (i\ne j),\\
[X_{2e_i},\theta X_{2e_i}]=-\frac2{||2e_i||^2}H_{2e_i},\quad
[H_{2e_i},X_{2e_i}]
=||2e_i||^2X_{2e_i},
\end{gather*}
we have for $i=1,\ldots,l$
\[
-H_{2e_i}=J^2H_{2e_i}=\ad(Z)^2H_{2e_i}
=-4c_i^2 H_{\alpha_i}
-c_i||2e_i||^2([Y,X_{2e_i}]-[Y,\theta X_{2e_i}])
\]
and hence
$[Y,X_{2e_j}]=[Y,\theta X_{2e_j}]=0$
and $c_{i}=\pm\frac12$.
\end{proof}
\begin{cor}\label{cor:gM}
One has $\mathfrak g_{\pm2e_i} \subset \mathfrak g^M$ for $i=1,\ldots,l$.
\end{cor}
\begin{proof}
This is clear from \eqref{eq:cpxStr}
since $\Ad(m)Z=Z$ for any $m\in M$.
\end{proof}
\begin{thm}\label{thm:Hermite}
Suppose $(\pi,V)$ is
a small $K$-type.
Then $V$ is one-dimensional and $\bska^\pi_{\mathrm{short}}=\bska^\pi_{\mathrm{middle}}=0$.
\end{thm}
\begin{proof}
First we assume $\mathfrak g=\mathfrak{sp}(p,\bbR)$.
Then $\mathfrak m=\{0\}$ and $\mathfrak t=\sum_{i=1}^p \bbR(X_{2e_i} + \theta X_{2e_i})$
is a Cartan subalgebra of $\mathfrak k$.
Since $\mathfrak t\subset \mathfrak k^M$ by Corollary \ref{cor:gM},
$M$ is a finite subgroup of the Cartan subgroup corresponding to $\mathfrak t$
and in particular is Abelian.
Thus any small $K$-type $\pi$ is one-dimensional.
We claim $\bska^\pi_{\mathrm{middle}}=0$ for such $\pi$.
Indeed, $[\mathfrak k,\mathfrak k]$ equals the orthogonal complement $(\bbR Z)^\perp$ of $\bbR Z$ in $\mathfrak k$ with respect to $B(\cdot,\cdot)$.
Since
\[
\sum_{\alpha: \mathrm{middle}} \{X_\alpha+\theta X_\alpha \,|\, X_\alpha \in\mathfrak g_\alpha\} \,\perp\, \sum_{i=1}^p \bbR(X_{2e_i} + \theta X_{2e_i}),
\]
one sees by Lemma \ref{lem:cpxStr} that $\{X_\alpha+\theta X_\alpha \,|\, X_\alpha \in\mathfrak g_\alpha\} \subset [\mathfrak k,\mathfrak k]$ for any middle $\alpha$.
Since $[\mathfrak k,\mathfrak k]\subset \Ker_{\mathfrak k}\pi$, our claim follows from Proposition \ref{prop:ND}.

Next, suppose $\mathfrak g$ is a general simple Lie algebra of Hermitian type
and $(\pi,V)$ is a small $K$-type.
We claim $\bska^\pi_{\mathrm{short}}=\bska^\pi_{\mathrm{middle}}=0$.
This was shown in the previous paragraph for $\mathfrak g=\mathfrak{sp}(p,\bbR)$
and in Theorem \ref{thm:so(p,q)} \ref{i:SO1} for $\mathfrak g=\mathfrak{so}(p,2)$.
For the remaining cases, the claim follows from Corollary \ref{cor:evenis0}.
Now, by Proposition \ref{prop:ND}, $\pi(X_\alpha+\theta X_\alpha)=0$
for any restricted root vector $X_\alpha$ of any $\alpha\in\varSigma$ with short or middle length.
On the other hand, by Corollary \ref{cor:gM}, each $\pi(X_{2e_i}+\theta X_{2e_i})$ is a scalar operator.
Hence by Corollary \ref{cor:kgen},
$\pi(X)$ for any $X\in\mathfrak k$ is a scalar operator.
This proves $V$ is one-dimensional. 
\end{proof}

Let $\mathfrak z$ and $\pi_0\in\sqrt{-1}\mathfrak z^*$ be as in \S\ref{subsec:Hermitian}.
If we identify $\pi_0$ with one-dimensional representation of $\mathfrak k$,
then it follows from \cite[Proposition 5.3.2]{He:white} that
\[
\pi_0(X_{2e_i}+\theta X_{2e_i})\in\{\pm\sqrt{-1}\}\quad\text{for }i=1,\ldots,l.
\]
Hence if (the differentiation of) a small $K$-type $\pi$ is written as $\pi=\nu\pi_0$
for some $\nu\in\bbQ$, then $\bska^\pi_{\mathrm{long}}=-\nu^2$.
This and Proposition \ref{prop:simplified}
imply the result stated in \S\ref{subsec:Hermitian}.

\subsection{\boldmath The case $\varSigma$ is of type $F_4$}\label{CC:F4}
Let $G$ be a simply-connected real simple Lie group with $\varSigma$ of type $F_4$.
As in \S\ref{subsec:F4}, we exclude the complex simple Lie group of type $F_4$.
Thus $\mathfrak g$ is one of
$\mathfrak{f}_{4(4)}$,
$\mathfrak{e}_{6(2)}$,
$\mathfrak{e}_{7(-5)}$ and
$\mathfrak{e}_{8(-24)}$.
Among these $\mathfrak{f}_{4(4)}$ is of split type.
Let $\varSigma_{\mathrm{short}}$ and $\varSigma_{\mathrm{long}}$ 
be as in \S\ref{subsec:F4}.
From \cite{Berger}
one sees there exists a sequence of embeddings
\begin{equation}\label{eq:F4embeddings}
\mathfrak{f}_{4(4)}
\subset
\mathfrak{e}_{6(2)}
\subset
\mathfrak{e}_{7(-5)}
\subset
\mathfrak{e}_{8(-24)}.
\end{equation}
The following table summarizes some necessary data on these Lie algebras:\smallskip
\begin{center}
\begin{tabular}{c|cccc}
$\mathfrak{g}$ &
$\mathfrak{f}_{4(4)}\ (\mathrm{F\,I})$&
$\mathfrak{e}_{6(2)}\ (\mathrm{E\,II})$&
$\mathfrak{e}_{7(-5)}\ (\mathrm{E\,VI})$&
$\mathfrak{e}_{8(-24)}\ (\mathrm{E\,IX})$\\
\hline
$\bsm_{\mathrm{short}}$&
$1$&
$2$&
$4$&
$8$\\
$\bsm_{\mathrm{long}}$&
$1$&
$1$&
$1$&
$1$\\
$\mathfrak{k}$&
$\mathfrak{sp}(3)\oplus\mathfrak{su}(2)$&
$\mathfrak{su}(6)\oplus\mathfrak{su}(2)$&
$\mathfrak{so}(12)\oplus\mathfrak{su}(2)$&
$\mathfrak{e}_7\oplus\mathfrak{su}(2)$
\end{tabular}
\end{center}\smallskip
Thus in any case $K$ is the product of two simple compact groups.
Let $K=K_1\times K_2$ with $K_2=\SU(2)$.
Let $\pr_i :K\to K_i$ be the projection ($i=1,2$).
\begin{thm}[{\cite[Theorem 1, Lemmas 4.2, 4.3]{SWL}}]\label{thm:split_F4}
Suppose $\mathfrak g=\mathfrak{f}_{4(4)}$.
Let $(\sigma,\bbC^2)$ be the irreducible representation of\/ $\SU(2)$
of dimension $2$.
Then $\pi=\sigma\circ\pr_2$ is the only non-trivial small $K$-type.
Moreover, $\bska^\pi_{\mathrm{short}}=0$ and $\bska^\pi_{\mathrm{long}}=-\frac14$.
\end{thm}
We generalize this to
\begin{thm}
The last theorem also holds for $\mathfrak g=\mathfrak{e}_{6(2)}$,
$\mathfrak{e}_{7(-5)}$ and
$\mathfrak{e}_{8(-24)}$.
\end{thm}
\begin{proof}
Suppose $\mathfrak g=\mathfrak{f}_{4(4)}$
and let $\pi=\sigma\circ\pr_2$ be as in Theorem \ref{thm:split_F4}.
Since $\bska^\pi_{\mathrm{short}}=0$,
it follows from Proposition \ref{prop:ND} that
$\sum_{\alpha\in\varSigma_{\mathrm{short}}}\{X_\alpha + \theta X_\alpha \,|\,X_\alpha \in \mathfrak g_\alpha\}\subset \Ker_{\mathfrak k}\pi=\mathfrak k_1$.
Since
$\sum_{\alpha\in\varSigma_{\mathrm{short}}}\{X_\alpha + \theta X_\alpha \,|\,X_\alpha \in \mathfrak g_\alpha\}$
is the orthogonal complement of 
$\sum_{\alpha\in\varSigma_{\mathrm{long}}}\{X_\alpha + \theta X_\alpha \,|\,X_\alpha \in \mathfrak g_\alpha\}$
in $\mathfrak k$ with respect to $B(\cdot,\cdot)$
and $\mathfrak k_1$ is that of $\mathfrak k_2$,
one has
\[
\mathfrak k_2 \subset
\sum_{\alpha\in\varSigma_{\mathrm{long}}}\{X_\alpha + \theta X_\alpha \,|\,X_\alpha \in \mathfrak g_\alpha\}.
\]
Now let $\mathfrak g'$ be one of $\mathfrak{e}_{6(2)}$,
$\mathfrak{e}_{7(-5)}$ and
$\mathfrak{e}_{8(-24)}$.
We use with the obvious meanings the notation $G'$, $K'$, $\mathfrak k'$, $\mathfrak g'_\alpha$, and so on.
Fix an embedding $G\hookrightarrow G'$
so that
$\mathfrak k'\cap \mathfrak g= \mathfrak k$ and $\mathfrak a'= \mathfrak a$.
We claim $\mathfrak k_2\simeq \mathfrak{su}(2)$ is an ideal of $\mathfrak k'$.
Indeed,
since 
$\mathfrak g_\alpha\,(\,=\mathfrak g'_\alpha)$
commutes with $\mathfrak m'$ for each $\alpha\in\varSigma_{\mathrm{long}}$,
\begin{equation}\label{eq:m'su2}
[\mathfrak m', \mathfrak k_2] 
\subset \biggl[\mathfrak m', \sum_{\alpha\in\varSigma_{\mathrm{long}}}\{X_\alpha + \theta X_\alpha \,|\,X_\alpha \in \mathfrak g_\alpha\}\biggr] =\{0\}.
\end{equation}
This proves our claim
since
$\mathfrak m'$ and $\mathfrak k$ generate the Lie algebra $\mathfrak k'$
by Theorem \ref{thm:Ko}.
Thus $\mathfrak k_2=\mathfrak k_2' \simeq \mathfrak{su}(2)$
and $K_2=K_2'$.
Since $\mathfrak k=(\mathfrak k'_1 \cap \mathfrak k)\oplus \mathfrak k_2$ is
a decomposition of $\mathfrak k$ into two ideals,
we have $\mathfrak k_1=\mathfrak k'_1 \cap \mathfrak k$
and $K_1 \subset K_1'$. 
Hence $\pi$ extends to a $K'$-type $\pi'=\sigma\circ\pr'_2$.
This is small since $M\subset M'$.
Since $\mathfrak g_\alpha \subset \mathfrak g'_\alpha$ for any $\alpha\in\varSigma$,
we have $\bska^{\pi'}=\bska^{\pi}$ by Lemma \ref{lem:kappa}.

Conversely, let $\nu$ be any non-trivial small $K'$-type.
Then it follows from Corollary \ref{cor:evenis0} and Proposition \ref{prop:ND}
that
$\sum_{\alpha\in\varSigma_{\mathrm{short}}}\{X_\alpha + \theta X_\alpha \,|\,X_\alpha \in \mathfrak g_\alpha\}\subset \Ker_{\mathfrak k'}\nu \cap \mathfrak k'_1$.
By the simplicity of $\mathfrak k'_1$,
$\mathfrak k'_1 \subset \Ker_{\mathfrak k'}\nu$
and there exists a non-trivial irreducible representation $\tau$ of $K_2'=\SU(2)$ such that
$\nu=\tau\circ\pr_2'$.
Now we claim $\pr_2'(M')=\pr_2(M)$.
Indeed, since
\[
\mathfrak g'_\spt =
\mathfrak a + \sum_{\alpha \in \varSigma_{\mathrm{long}}}
\mathfrak g_\alpha
\subset
\mathfrak g,
\]
we have $G'_\spt\subset G$ and $M'_\spt \subset M$.
On the other hand, since \eqref{eq:m'su2} implies
$\mathfrak m' \subset \mathfrak k_1'$,
one has $M'_0\subset K_1'$.
Hence $\pr_2'(M')=\pr_2'(M'_0M'_\spt)=\pr_2'(M'_\spt)\subset \pr_2'(M)=\pr_2(M)$.
Since the opposite inclusion is obvious, we get the claim.
Thus $\nu|_K=\tau\circ\pr_2$ is a small $K$-type.
This is non-trivial since $\dim \tau >1$.
By Theorem \ref{thm:split_F4} we conclude $\tau=\sigma$.
\end{proof}
These two theorems and Proposition \ref{prop:simplified}
imply the result of \S\ref{subsec:F4}.

\subsection{\boldmath Split Lie groups with simply-laced $\varSigma$}\label{CC:split}
Let $G$ be one of the simply-connected split real simple Lie groups of type $A_l$ ($l\ge2$),
$D_l$ ($l\ge3$) and $E_l$ ($l=6,7,8$). 
Let $(\pi,V)$ be a small $K$-types listed in Theorem \ref{thm:split_sl}.
Then it follows from \cite[Lemma 4.2]{SWL} that $\bska^\pi_\alpha=-\frac14$
for any $\alpha\in\varSigma$.
(For the type $D_l$ case, we already know this by Theorem \ref{thm:split_so(p,q)} \ref{i:Dsplit}.)
Thus Proposition \ref{prop:simplified}
implies the result of \S\ref{subsec:split}.

\subsection{\boldmath The split Lie group of type $G_2$}\label{CC:G2}
Let $G=\tilde G_2$, the simply-connected split real simple Lie group of type $G_2$.
We use the same notation as in \S\ref{subsec:G2}.
Let $\varSigma_{\mathrm{short}}\sqcup\varSigma_{\mathrm{long}}$ 
be the division of $\varSigma$ according to the root lengths.
From \cite[Lemmas 4.2, 4.3]{SWL}
the values of $\bska^{\pi_1}$ and $\bska^{\pi_2}$ are as follows:\smallskip
\begin{center}
\begin{tabular}{c|cc}
$\pi$ & $\pi_1$ & $\pi_2$ \\
\hline
$\bska^\pi_{\mathrm{short}}$
& $-\frac14$ & $-\frac94$ \\
$\bska^\pi_{\mathrm{long}}$
& $-\frac14$ & $-\frac14$
\end{tabular}
\end{center}\smallskip
Hence if $\pi=\pi_1$ then we have the same result as for the split simply-laced case.

Suppose $\pi=\pi_2$
and let us prove 
we cannot find any combination of $\varSigma^\pi$ and $\bsk^\pi$
for which \eqref{eq:main} holds.
To do so, assume 
\eqref{eq:main} holds for some $(\varSigma^\pi,\bsk^\pi)$.
Then
$\tilde\delta_{G/K}^{-\frac12}\,\tilde\delta(\varSigma^\pi,\bsk^\pi)^{\frac12}$is non-singular at $0$ and hence
for each $\alpha\in \varSigma$ there exists $\beta\in \varSigma^\pi$
which is proportional to $\alpha$.
This implies $\varSigma^\pi$ is of type $G_2$.
Now
$\tilde\delta_{G/K}^{\frac12}\, \varUpsilon^\pi(\phi^\pi_\lambda) = \tilde\delta(\varSigma^\pi,\bsk^\pi)^{\frac12}\,F(\varSigma^\pi,\bsk^\pi,\lambda) \in \mathscr A(\mathfrak a_\reg)$ is an eigenfunction of
both \eqref{eq:Omegatwist} and \eqref{eq:Ltwist} with $(\varSigma',\bsk)=(\varSigma^\pi,\bsk^\pi)$.
Hence
\begin{multline*}
\sum_{\alpha\in\varSigma_{\mathrm{short}}\cap\varSigma^+}
\frac{||\alpha||^2}{16}
\biggl(
\frac{4}{\sinh^2\frac\alpha2}
-
\frac{32}{\sinh^2\alpha}
\biggr)
+
\sum_{\alpha\in\varSigma_{\mathrm{long}}\cap\varSigma^+}
\frac{||\alpha||^2}{16\sinh^2\frac\alpha2}\\
=
\sum_{\alpha\in\varSigma^{\pi+}}\frac{\bsk^\pi_\alpha(1-\bsk^\pi_\alpha-2\bsk^\pi_{2\alpha})||\alpha||^2}{4\sinh^2\frac\alpha2}+C
\end{multline*}
for some constant $C$.
Since the members of $\{\sinh^{-2}\frac\alpha2\,|\,\alpha\in\varSigma\cup2\varSigma_{\mathrm{short}}\cup\varSigma^\pi\}\cup\{1\}$ are linearly independent in $\mathscr A(\mathfrak a_\reg)$,
we have $\bsk^\pi_\alpha(1-\bsk^\pi_\alpha-2\bsk^\pi_{2\alpha})\ne0$ for each $\alpha \in \varSigma\cup2\varSigma_{\mathrm{short}}$.
Hence $\varSigma^\pi \supset \varSigma\cup2\varSigma_{\mathrm{short}}$,
a contradiction.
The results in \S\ref{subsec:G2} are thus proved.

\section{Spherical transforms}\label{sec:application}
Let $G$ be a non-compact real simple Lie group with finite center and $(\pi,V)$ a small $K$-type.
In this section we apply our main formula \eqref{eq:main}
to the calculation of Harish-Chandra's $c$-function for $G\times_K V$
and the theory of $\pi$-spherical transform.
We note each combination of $\varSigma^\pi$ and $\bsk^\pi$
in \S\ref{sec:list} is chosen
so that
\begin{equation}\label{eq:MC1}
\varSigma^\pi\subset\varSigma\cup2\varSigma.
\end{equation}

\subsection{Weight functions}
The weight functions $\tilde\delta_{G/K}$ and
$\tilde\delta(\varSigma^\pi,\bsk^\pi)$ 
defined by \eqref{eq:GKdelta} and \eqref{eq:HOdelta}
are normalized so that $\tilde\delta_{G/K}^{-\frac12}\, \tilde\delta(\varSigma^\pi,\bsk^\pi)^{\frac12}$ in \eqref{eq:main} takes $1$ at $0\in\mathfrak a$.
In the literature,
\[
\delta_{G/K}
=\prod_{\alpha \in \varSigma^+}|2\sinh \alpha|^{\bsm_\alpha}\quad\text{and}\quad
\delta(\varSigma^\pi,\bsk^\pi)
=\prod_{\alpha\in {\varSigma^\pi}^+} \biggl|2\sinh\frac\alpha2\biggr|^{2\bsk^\pi_\alpha}
\]
are often used.
\begin{lem}
Suppose \eqref{eq:main} and \eqref{eq:MC1} are valid for $\varSigma^\pi$ and $\bsk^\pi$.
Then
\begin{align}\label{eq:2deltas}
\tilde\delta_{G/K}^{-\frac12}\, \tilde\delta(\varSigma^\pi,\bsk^\pi)^{\frac12}
&=
2^{e(\varSigma^\pi,\bsk^\pi)}\,
\delta_{G/K}^{-\frac12}\, \delta(\varSigma^\pi,\bsk^\pi)^{\frac12}
\intertext{with}
e(\varSigma^\pi,\bsk^\pi)&=
\sum_{\alpha\in\varSigma^+\setminus2\varSigma^+}\biggl(\bsk^\pi_\alpha-\bsk^\pi_{4\alpha}+\frac{\bsm_{2\alpha}}2\biggr).\label{eq:expo}
\end{align}
\end{lem}
\begin{proof}
By a calculation similar to the one for \eqref{eq:coshfactor} we have
\begin{equation*}%\label{eq:coshfactor2}
\delta_{G/K}^{-\frac12}\, \delta(\varSigma^\pi,\bsk^\pi)^{\frac12}
=\prod_{\alpha\in\varSigma^+\setminus2\varSigma^+}
\biggl(2\cosh\frac\alpha2\biggr)^{-\bsk^\pi_\alpha}
(2\cosh\alpha)^{\bsk^\pi_{4\alpha}-\frac{\bsm_{2\alpha}}2}.
\end{equation*} 
The lemma then follows from this and \eqref{eq:coshfactor}.
\end{proof}

\subsection{\boldmath Harish-Chandra's $c$-function}
Let $\mathfrak{a}_+:=\{H\in\mathfrak{a}\,|\,\alpha(H)>0 \text{ for any }\alpha\in \varSigma^+\}$ and
$\mathfrak{a}_+^*=
\{\lambda\in\mathfrak{a}^*\,|\,\lambda(\alpha^\vee)>0 \text{ for any } \alpha\in \varSigma^+\}$.
For $\lambda\in\mathfrak{a}^*_++\sqrt{-1}\mathfrak{a}^*$ put
\begin{equation}
c^\pi(\lambda)=\int_{\bar{N}}e^{-(\lambda+\rho)(H(\bar{n}))}\pi(\kappa(\bar{n}))d\bar{n}
\label{eq:cfgroup}
\end{equation}
where the Haar measure $d\bar n$ on $\bar{N}:=\theta N$ is normalized so that 
\[
\int_{\bar{N}}e^{-2\rho(H(\bar{n}))}d\bar{n}=1.
\]
The integral in \eqref{eq:cfgroup} absolutely converges
and defines an $\End_MV$-valued holomorphic function known as
Harish-Chandra's $c$-function.
This satisfies
for any $H\in\mathfrak a_+$ and $\lambda\in\mathfrak{a}^*_++\sqrt{-1}\mathfrak{a}^*$
\begin{equation}
\lim_{t\to\infty}e^{t(-\lambda+\rho)(H)}\phi^\pi_\lambda(e^{tH})=c^\pi(\lambda). 
\label{eq:limcf1} 
\end{equation}
These things are shown using the integral formula \eqref{eq:eisenstein}
in the same way as in the case of the trivial $K$-type
(cf.~\cite[Chapter IV, \S6, No.6]{Hel2}),
or are deduced as a  special case of
the asymptotic behavior of Eisenstein integrals
(cf.~\cite[Theorem 9.1.6.1]{War}, \cite[Theorem 14.7, (14.29)]{Kn0}).
It is known that $c^\pi(\lambda)$ extends to a meromorphic function 
on $\mathfrak{a}_\mathbb{C}^*$. 
We regard $c^\pi(\lambda)$ as a $\bbC$-valued function
by $\text{End}_M V\simeq \mathbb{C}$.
\begin{thm}\label{thm:2cs}
Suppose \eqref{eq:main} and \eqref{eq:MC1} are valid for $\varSigma^\pi$ and $\bsk^\pi$.
With $e(\varSigma^\pi,\bsk^\pi)$ in \eqref{eq:expo}
we have
\begin{equation}\label{eq:2cs}
c^\pi(\lambda)=2^{e(\varSigma^\pi,\bsk^\pi)}\,c(\varSigma^\pi,\bsk^\pi,\lambda).
\end{equation}
(Recall $c(\varSigma^\pi,\bsk^\pi,\lambda)$ is defined by \eqref{eq:cfho}.)
\end{thm}
\begin{proof}
Note that
\[
\lim_{t\to\infty} e^{t\rho(H)}\,\delta_{G/K}^{-\frac12}(tH)
=
\lim_{t\to\infty} e^{-t\rho(\bsk^\pi)(H)}\,\delta(\varSigma^\pi,\bsk^\pi;tH)^{\frac12}=1
\]
and that
\[
e^{-\lambda+\rho}\,\varUpsilon^\pi(\phi^\pi_\lambda)
=
2^{e(\varSigma^\pi,\bsk^\pi)}
\bigl(e^\rho\,\delta_{G/K}^{-\frac12}\bigr)
\bigl(e^{-\rho(\bsk^\pi)}\,\delta(\varSigma^\pi,\bsk^\pi)^{\frac12}\bigr)
e^{-\lambda+\rho(\bsk^\pi)}F(\varSigma^\pi,\bsk^\pi,\lambda).
\]
Hence \eqref{eq:2cs} follows from \eqref{eq:limcf1} and Corollary \ref{cor:HOasymp}.
\end{proof}

\subsection{\boldmath The $\pi$-spherical transform}
Let $C_\cpt^\infty(G,\pi,\pi)$ be the subspace of $C^\infty(G,\pi,\pi)$ consisting of the compactly supported $\pi$-spherical functions.
For $\phi_1\in C^\infty(G,\pi,\pi)$ and
$\phi_2\in C^\infty_\cpt(G,\pi,\pi)$ define the convolution $\phi_1*\phi_2\in C^\infty(G,\pi,\pi)$ by
\begin{equation*}
(\phi_1*\phi_2)(x)=\int_G \phi_1(g^{-1}x)\phi_2(g)dg,
\end{equation*}
where $dg$ is a Haar measure on $G$.
Since $(\pi,V)$ is a small, $C^\infty_\cpt(G,\pi,\pi)$ is a commutative algebra by \cite[Theorem 3]{Deit}. 
For $\phi\in C^\infty_\cpt(G,\pi,\pi)$
we define its $\pi$-spherical 
transform by
\begin{equation}\label{eq:sphericaltransform1}
\hat{\phi}(\lambda)=\int_G \phi_\lambda^\pi(g^{-1})\phi(g)dg=(\phi^\pi_\lambda*\phi)(1_G),
\end{equation}
which is, by \eqref{eq:integral-rep}, a holomorphic function on $\mathfrak a_\bbC^*$ taking values in $\End_K V\simeq\bbC$.
For each $\lambda\in\mathfrak a_\bbC^*$
\[
C_\cpt^\infty(G,\pi,\pi) \ni \phi\mapsto \hat{\phi}(\lambda) \in \bbC
\]
is an algebra homomorphism since one has
$\phi^\pi_\lambda*\phi=\hat{\phi}(\lambda)\phi^\pi_\lambda$ by Theorem \ref{thm:unique-sph}.

Now we normalize the Haar measure $dH$ on $\mathfrak a$ so that
for any compactly supported continuous $K$-bi-invariant function $\psi$ on $G$
\begin{equation*}%\label{eq:intf}
\int_G \psi(g)dg=\frac1{\#W}\int_{\mathfrak a} \psi(e^H)\, \delta_{G/K}(H)dH
\end{equation*}
(cf.~\cite[Ch.~I, Theorem 5.8]{Hel2}, \cite[Proposition 2.4.6]{GV}).
Since the trace of the integrand of \eqref{eq:sphericaltransform1}
is $K$-bi-invariant, we have using \eqref{eq:sphnegative}
\begin{align*}
\hat{\phi}(\lambda)&=
\frac1{\#W}\int_{\mathfrak a} \varUpsilon^\pi(\phi^\pi_\lambda)(-H)\,
\varUpsilon^\pi(\phi)(H)\, \delta_{G/K}(H)dH\\
&=\frac1{\#W}\int_{\mathfrak a} \varUpsilon^\pi(\phi^\pi_{-\lambda})(H)\,
\varUpsilon^\pi(\phi)(H)\, \delta_{G/K}(H)dH.
\end{align*}
Hence by Theorem \ref{thm:Chevalley} the $\pi$-spherical transform
is identified with the integral transform
\begin{equation}\label{eq:HCtrans}
f
\mapsto \hat f(\lambda):=\frac1{\#W}\int_{\mathfrak a} f(H)\,\varUpsilon^\pi(\phi^\pi_{-\lambda})(H)\,\delta_{G/K}(H)dH.
\end{equation}
for $f\in C^\infty_\cpt(\mathfrak a)^W$.
Here $C^\infty_\cpt(\mathfrak a)^W=\{f\in C^\infty(\mathfrak a)\,|\,f\text{ with compact support}\}$.

Let $\varSigma'$ be a root system in $\mathfrak{a}^*$
and $\bsk$ a multiplicity function on $\varSigma'$. 
For $f\in C_c^\infty(\mathfrak{a})^W$ we define its hypergeometric Fourier transform 
$\mathcal F=\mathcal{F}(\varSigma',\bsk)$ by
\begin{equation}\label{eq:choptransform}
\mathcal{F}f(\lambda):=\frac{1}{\#W}\int_{\mathfrak{a}}f(H)\,F(\varSigma',\boldsymbol{k},-\lambda;H)\,
\delta(\varSigma',\boldsymbol{k};H)dH.
\end{equation}
This makes sense when $\delta(\varSigma',\boldsymbol{k})$ is locally integrable.
\begin{rem}
The hypergeometric Fourier transform is introduced by Opdam \cite{Op:Cherednik} as the \emph{Cherednik transform}.
If $\varSigma'=2\varSigma$ and $\bsk_{2\alpha}=\bsm_\alpha/2$, 
then $\mathcal{F}(\varSigma',\bsk)$ equals \eqref{eq:HCtrans} for the trivial $K$-type $\pi$, that is,
the Harish-Chandra transform (cf.~\cite[Chapter 6]{GV}, 
\cite[Ch.~IV]{Hel2}, \cite[Chapter 9]{War}). 
If $\varSigma'$ has real rank one, then $\mathcal{F}(\varSigma',\bsk)$ reduces to the \emph{Jacobi transform} (cf.~\cite{Koornwinder}). 
\end{rem}
\begin{thm}\label{thm:2Fourier}
Suppose \eqref{eq:main} and \eqref{eq:MC1} are valid for $\varSigma^\pi$ and $\bsk^\pi$.
Then $\delta(\varSigma^\pi,\bsk^\pi)$ is locally integrable
and it holds with $\mathcal F=\mathcal{F}(\varSigma^\pi,\bsk^\pi)$
and $e(\varSigma^\pi,\bsk^\pi)$ in \eqref{eq:expo} that
\begin{equation}\label{eq:2Fourier}
\hat f
=2^{e(\varSigma^\pi,\bsk^\pi)}\mathcal F\bigl(f\,\delta_{G/K}^{\frac12}\,\delta(\varSigma^\pi,\bsk^\pi)^{-\frac12}\bigr)
\quad\text{for any }f\in C^\infty_\cpt(\mathfrak a)^W.
\end{equation}
\end{thm}
\begin{proof}
The local integrability follows from \eqref{eq:regularity},
while \eqref{eq:2Fourier} is direct from
\eqref{eq:main}, \eqref{eq:2deltas}, \eqref{eq:HCtrans} and \eqref{eq:choptransform}.
\end{proof}

\subsection{Inversion formulas and Plancherel formulas}
We normalize the Haar measure $d\lambda$ on $\sqrt{-1}\mathfrak{a}^*$
 so that the Euclidean Fourier transform and its inversion are given by 
\[
\tilde{f}(\lambda)=\int_{\mathfrak{a}}f(H)e^{-\lambda(H)}dH, \qquad
f(H)=\int_{\sqrt{-1}\mathfrak{a}^*} \tilde{f}(\lambda)e^{\lambda(H)}d\lambda. 
\]
On the hypergeometric Fourier transform we have
\begin{thm}[\cite{Op:Cherednik, Op:book}]\label{th:chopinversion}
Let $\varSigma'$ be a root system in $\mathfrak{a}^*$
and $\bsk$ a real-valued multiplicity function on $\varSigma'$
such that $\bsk_\alpha\ge0$ for any $\alpha\in \varSigma'$.
Let $\mathcal F=\mathcal F(\varSigma',\bsk)$.
Then for any $f\in  C_\cpt^\infty(\mathfrak{a})^W$ we have
the inversion formula
\[
f(H)=
\frac{1}{\#W}
\int_{\sqrt{-1}\mathfrak{a}^*} \mathcal{F}f(\lambda)
\,F(\varSigma',\bsk,\lambda;H)
\, |c(\varSigma',\bsk,\lambda)|^{-2}d\lambda.
\]
and the Plancherel-type formula
\[
\frac{1}{\#W}
\int_{\mathfrak{a}}|f(H)|^2 {\delta}(\varSigma',\boldsymbol{k};H)dH=
\frac{1}{\#W}
\int_{\sqrt{-1}\mathfrak{a}^*} |\mathcal{F}f(\lambda)|^2 |c(\varSigma',\bsk,\lambda)|^{-2}d\lambda.
\]
Moreover $\mathcal{F}$ uniquely extends to the isometry
\[
L^2(\mathfrak{a},\tfrac{1}{\#W}\delta(\varSigma',\bsk;H)dH)^W
\simarrow
L^2(\sqrt{-1}\mathfrak{a}^*,\tfrac{1}{\#W}|c(\varSigma',\bsk,\lambda)|^{-2}d\lambda)^W.
\]\end{thm}

From this
the following result on the $\pi$-spherical transform is deduced:

\begin{cor}\label{cor:pisphericalinversion} 
Suppose \eqref{eq:main} is valid for $\varSigma^\pi$ and $\bsk^\pi$ such that
\begin{equation}\label{eq:kpositive}
\bsk^\pi_\alpha\ge 0\quad\text{for any }\alpha\in\varSigma^\pi.
\end{equation}
Then for $f\in C_c^\infty(\mathfrak{a})^W$ we have the inversion formula
\[
f(H)=\frac{1}{|W|}\int_{\sqrt{-1}\mathfrak{a}^*} \hat{f} (\lambda)\,
\varUpsilon^\pi(\phi_\lambda)(H)\,
|c^\pi(\lambda)|^{-2}d\lambda
\]
and the Plancherel-type formula
\[
\frac{1}{\#W}\int_{\mathfrak{a}}|f(H)|^2 \tilde{\delta}_{G/K}(H)dH=
\frac{1}{\#W}\int_{\sqrt{-1}\mathfrak{a}^*} |\hat{f}(\lambda)|^2 |c^\pi(\lambda)|^{-2}d\lambda.
\]
Moreover the $\pi$-spherical transform $f\mapsto \hat{f}$
uniquely extends to the isometry 
\[
L^2(\mathfrak{a},\tfrac{1}{\#W}\tilde{\delta}_{G/K}(H)dH)^W
\simarrow
L^2(\sqrt{-1}\mathfrak{a}^*,\tfrac{1}{\#W}|c^\pi(\lambda)|^{-2}d\lambda)^W.
\]
\end{cor} 
\begin{proof}
Under \eqref{eq:MC1} all the statements are immediate from
\eqref{eq:main}, \eqref{eq:2deltas}, \eqref{eq:2cs}, \eqref{eq:2Fourier} and Theorem \ref{th:chopinversion}.
So suppose \eqref{eq:MC1} is not valid.
Even in this case, one easily sees that $\delta(\varSigma^\pi,\bsk^\pi)$ is
locally integrable and that Formulas
\eqref{eq:2deltas}, \eqref{eq:2cs} and \eqref{eq:2Fourier} hold
by replacing $2^{e(\varSigma^\pi,\bsk^\pi)}$ with some constant.
Hence the corollary follows.
\end{proof}

As we can see from the list in \S\ref{sec:list},
there exists at least a pair of $\varSigma^\pi$ and $\bsk^\pi$
satisfying \eqref{eq:main} and \eqref{eq:kpositive} unless $(\pi,V)$ is one of the following:

\begin{enumerate}[label=(\arabic*), leftmargin=*]
\item $(\pi,V)$ for $\mathfrak g=\mathfrak{sp}(p,1)\ (p\ge1)$;\label{bad1}
\item $(\pi_s^\pm,V)$ in \S\ref{subsec:so(2r,1)}
for $\mathfrak g=\mathfrak{so}(2r,1)\ (r\ge2)$\label{bad2};
\item $(\pi,V)$ in \S\ref{subsec:so(p,q)} \ref{sopq2}
for $\mathfrak g=\mathfrak{so}(p,q)\ (p>q\ge3, p:\text{even}, q:\text{odd})$;\label{bad3}
\item $(\pi,V)$ for a Hermitian $G$;\label{bad4}
\item $(\pi_2,\bbC^2)$ in Theorem \ref{thm:G2} for $\tilde G_2$.\label{bad5}
\end{enumerate}

For \ref{bad5}, 
the elementary $\pi$-spherical functions cannot be expressed by Heckman-Opdam hypergeometric functions. 
The harmonic analysis in this exceptional case is therefore to be studied separately.
But we do not go further with it in this paper.

For \ref{bad1}--\ref{bad4},
every $\varSigma^\pi$ chosen in \S\ref{sec:list}
is of type $BC$.
Suppose first that $G$ has real rank one
and let $\varSigma^\pi=\{\pm\alpha,\pm2\alpha\}$.
Then the inversion formula and the Plancherel-type formula
for the Jacobi transform $\mathcal F$ are available under the assumption
\[
\bsk^\pi_{\alpha},\,
\bsk^\pi_{2\alpha} \in \bbR
\quad\text{and}\quad
\bsk^\pi_{\alpha} + \bsk^\pi_{2\alpha} > -\frac12,
\]
which is much weaker than that of Theorem \ref{th:chopinversion} (cf.~\cite[Appendix 1]{FJ}, \cite{Koornwinder}).
From this we can deduce a result on the $\pi$-spherical transform
for \ref{bad1}, \ref{bad2}, and \ref{bad4} with $\mathfrak g=\mathfrak{su}(p,1)$
corresponding to Corollary \ref{cor:pisphericalinversion}.
In fact, such a result is shown by \cite{vDP, S:hyperbolic} for \ref{bad1},
by \cite{CP} for \ref{bad2} with $s=1$,
and by \cite{FJ} for \ref{bad4} with $\mathfrak g=\mathfrak{su}(p,1)$.
If $c^\pi(\lambda)$ has zeros in $\mathfrak a^*_++\sqrt{-1}\mathfrak a^*$
then the inversion and Plancherel-type formulas contain
discrete spectra in addition to the same continuous spectrum as in the formulas of Corollary \ref{cor:pisphericalinversion}.

Next, let us consider \ref{bad4} for $G$ with higher rank.
If the parameter $\nu$ of the one-dimensional $K$-type $\pi$ given in 
\S\ref{subsec:Hermitian} satisfies $2|\nu|\le\max\{\bsm_{\frac{\alpha}2},1\}$
for $\alpha\in\varSigma_{\mathrm{long}}$,
then we can apply Theorem \ref{th:chopinversion} to deduce Corollary \ref{cor:pisphericalinversion}. 
The general inversion and Plancherel-type formulas are given by \cite[Chapter 5]{He:white} for $|\nu|<\frac12\bsm_{\frac{\alpha}{2}}\,\,
(\alpha\in\varSigma_{\mathrm{long}})$ and by \cite{S:Plancherel} for an arbitrary $\nu$. 
Both \cite{He:white} and \cite{S:Plancherel} employ Rosenberg's method (\cite{R}) for the classical Harish-Chandra transform.
If $|\nu|$ is sufficiently large, the Plancherel measure contains 
spectra with lower-dimensional support along with the most continuous spectrum in Corollary \ref{cor:pisphericalinversion}. 
Possible spectra with lower-dimensional support are obtained by calculating residues of 
$c^\pi(\lambda)^{-1}$ in $\mathfrak{a}_+^*+\sqrt{-1}\mathfrak{a}^*$ (see \cite{S:Plancherel} for details). 

Finally suppose $(\pi,V)$ is \ref{bad3}.
In the notation of \S\ref{subsec:so(p,q)}
one has from Theorem~\ref{thm:2cs} that
\begin{align*}
c^\pi(\lambda)&=
C\prod_{i=1}^q\frac{\varGamma(\lambda((2e_i)^\vee)+\frac12)}{\varGamma(\lambda((2e_i)^\vee)+\frac12(p-q+1))}\cdot\\
&\qquad\qquad\cdot \prod_{1\le i<j\le q}\frac{\varGamma(\lambda((e_i-e_j)^\vee))\varGamma(\lambda((e_i+e_j)^\vee))}{\varGamma(\lambda((e_i-e_j)^\vee)+\frac12)\varGamma(\lambda((e_i+e_j)^\vee)+\frac12)}
\end{align*}
with a positive constant $C$.
Thus $c^\pi(\lambda)$ has no zero in $\overline{\mathfrak a^*_+}+\sqrt{-1}\mathfrak a^*$.
Hence it is very likely that we could prove the same result as
Corollary \ref{cor:pisphericalinversion} for $\pi$
by Rosenberg's method.
However it is more desirable that
a general result on $\mathcal F$ for $\varSigma'$ of type $BC$
would hold under a weaker assumption than that of Theorem \ref{th:chopinversion}
and from this we could deduce all the results for \ref{bad1}--\ref{bad4} in a uniform way.
We will discuss this problem elsewhere.

\section*{Acknowledgments}
The authors wish to thank Professor Hiroyuki Ochiai for helpful comments on an earlier version of this paper.

\end{document}